\documentclass[11pt,a4paper]{article}

\usepackage{titlesec}
\usepackage{fancyhdr}
\usepackage{a4wide}
\usepackage{graphicx}
\usepackage{float}
\usepackage{amssymb}
\usepackage{amsmath}
\usepackage{amsthm}
\usepackage{xcolor,colortbl}
\usepackage{mathrsfs}
\usepackage{array}
\usepackage{eucal}
\usepackage{tikz}
\usetikzlibrary{patterns}
\usepackage{multido}
\usepackage[T1]{fontenc}
\usepackage{hhline}
\usepackage{inputenc}
\usepackage[english]{babel}
\usetikzlibrary{shapes.misc}
\tikzset{cross/.style={cross out, draw=blue, minimum size=1*(#1-\pgflinewidth), inner sep=0pt, outer sep=0pt},
cross/.default={6pt}}
\usepackage{lmodern}
\usepackage{hyperref}
\usepackage{geometry}
\usepackage{changepage}
\changepage{0pt}{}{}{}{}{0pt}{}{0pt}{10pt}
\usepackage[numbers]{natbib}
\usepackage{hyperref}
\hypersetup{
pdfpagemode=none,
pdftoolbar=true,        
pdfmenubar=true,        
pdffitwindow=false,     
pdfstartview={Fit},    
pdftitle={On well-posedness of time-harmonic problems in an unbounded strip for a thin plate model},    
pdfauthor={L. Bourgeois, L. Chesnel, S. Fliss},     
pdfsubject={},  
pdfcreator={L. Bourgeois, L. Chesnel, S. Fliss},   
pdfproducer={L. Bourgeois, L. Chesnel, S. Fliss}, 
pdfkeywords={}, 
pdfnewwindow=true,      
colorlinks=true,       
linkcolor=magenta,          
citecolor=red,        
filecolor=cyan,      
urlcolor=blue           
}

\newcommand{\dsp}{\displaystyle}

\newcommand{\eps}{\varepsilon}
\newcommand{\om}{\omega}
\newcommand{\Om}{\Omega}
\newcommand{\mrm}[1]{\mathrm{#1}}

\newcommand{\Cplx}{\mathbb{C}}
\newcommand{\N}{\mathbb{N}}
\newcommand{\R}{\mathbb{R}}

\newcommand{\PetitCarre}{\scalebox{0.5}{$\square$}}
\newcommand{\mL}{\mrm{L}}
\newcommand{\mH}{\mrm{H}}
\newcommand{\mV}{\mrm{V}}

\newcommand{\mW}{\mrm{W}}

\newcommand{\supp}{\mrm{supp}}
\renewcommand{\ker}{\mrm{ker}}
\newcommand{\coker}{\mrm{coker}}
\newcommand{\ind}{\mrm{ind}}
\renewcommand{\dim}{\mrm{dim}}

\definecolor{Gray}{gray}{0.90}
\newcommand {\be}{\begin{equation}}
\newcommand {\ee}{\end{equation}}

\newcommand {\la} {\lambda}

\newcommand {\deri}[2] {\frac {\partial #1}{\partial #2}}
\newcommand {\derd}[3] {\frac {\partial^{2} #1}{\partial #2 \partial
#3}}
\newcommand {\ders}[2] {\frac {\partial^{2} #1}{\partial {#2} ^2 }}
\newtheorem{theorem}{Theorem}[section]
\newtheorem{lemma}{Lemma}[section]
\newtheorem{remark}{Remark}[section]

\newtheorem{corollary}{Corollary}[section]
\newtheorem{proposition}{Proposition}[section]

\begin{document}

~\vspace{-0.3cm}
\begin{center}
{\sc \bf\LARGE On well-posedness of time-harmonic problems \\[6pt] in an unbounded strip for a thin plate model}
\end{center}

\begin{center}
\textsc{Laurent Bourgeois}$^1$, \textsc{Lucas Chesnel}$^{2}$, \textsc{Sonia Fliss}$^1$\\[16pt]
\begin{minipage}{0.95\textwidth}
{\small
$^1$ Laboratoire Poems, CNRS/ENSTA/INRIA, Ensta ParisTech, Universit\'e Paris-Saclay, 828, Boulevard des Mar\'echaux, 91762 Palaiseau, France; \\
$^2$ INRIA/Centre de math\'ematiques appliqu\'ees, \'Ecole Polytechnique, Universit\'e Paris-Saclay, Route de Saclay, 91128 Palaiseau, France.\\[10pt]
E-mails: \textit{Laurent.Bourgeois@ensta-paristech.fr},\, \textit{Lucas.Chesnel@inria.fr},\, \textit{Sonia.Fliss@ensta-paristech.fr}\\[-14pt]
\begin{center}
(\today)
\end{center}
}
\end{minipage}
\end{center}
\vspace{0.4cm}

\noindent\textbf{Abstract.} 
We study the propagation of elastic waves in the time-harmonic regime in a waveguide which is unbounded in one direction and bounded in the two other (transverse) directions. We assume that the waveguide is thin in one of these transverse directions, which leads us to consider a Kirchhoff-Love plate model in a locally perturbed 2D strip. For time harmonic scattering problems in unbounded domains, well-posedness does not hold in a classical setting and it is necessary to pre-
scribe the behaviour of the solution at infinity. This is challenging for the model that we consider and constitutes our main contribution. Two types of boundary conditions are considered: either the strip is simply supported or the strip is clamped. The two boundary conditions are treated with two different methods. For the simply supported problem, the analysis is based on a result of Hilbert basis in the transverse section. For the clamped problem, this property does not hold. Instead we adopt the Kondratiev's approach, based on the use of the Fourier transform in the unbounded direction, together with techniques of weighted Sobolev spaces with detached asymptotics. After introducing radiation conditions, the corresponding scattering problems are shown to be well-posed in the Fredholm sense. We also show that the solutions are the physical (outgoing) solutions in the sense of the limiting absorption principle.
~

\section{Introduction}

The Kirchhoff-Love model for thin elastic plates has now a quite long history and is of practical use in the field of 
mechanical engineering. From the mathematical and the numerical point of view, there is a considerable amount of contributions concerning the static case. In this field, we can for example refer to the monographs \cite{ciarlet,destuynder_salaun,chen_zhou,hsiao_wendland}. 
Many authors have also analyzed the behaviour of Kirchhoff-Love plates in the dynamic case, at least in the time domain.
Here, we can refer for example to \cite{lagnese_lions,becache_derveaux_joly}. In particular, the various models for plate problems in the time domain are derived and justified in \cite{lagnese_lions}.
However, the number of contributions concerning time-harmonic problems for infinite Kirchhoff-Love plates at non zero frequencies seems much smaller. From the theoretical point of view, the scattering solutions in the restricted case of purely radial inhomogeneities are analytically computed in \cite{norris_vermula}, while well-posedness in the presence of a potential is rigorously established in \cite{tyni_serov} for a large enough frequency.   
From the numerical point of view, some finite element computations with the help of Perfectly Matched Layers can be found in \cite{farhat_guenneau_enoch}. Let us also mention the studies concerning the so-called platonic crystals \cite{EvPo08,HMMM11,SmPW12,SmMM14,HJMM18} (by analogy with photonic, phononic or plasmonic crystals). In these works, the authors investigate the propagation of time harmonic waves 
in waveguides which consist of rigid pins embedded within an elastic Kirchhoff plate. \\ 
\newline
Our paper focuses on a two dimensional waveguide which is infinite in one direction and bounded in the perpendicular direction: it will be referred to as the strip in the following. We consider the Kirchhoff-Love model in such strip. We acknowledge that the Kirchhoff-Love model is the simplest possible one to describe plates --~see for instance \cite{lagnese_lions} where various models for plate problems are derived and justified. However, to the best of our knowledge, well-posedness of time-harmonic problems in a strip for such model has not been investigated up to now. This study can be considered as a first step for the analysis of richer plate models.
\\\\
The standard Helmholtz equation in a waveguide has been extensively studied (see for example \cite{lenoir-tounsi,harari,Hagstrom:1999,bonnet_dahi_luneville_pagneux,bourgeois_luneville}). Let us remind the reader of the main results for this simpler case. In the classical functional framework ($\mL^2$), existence of a solution may fail (the physical solution may propagate towards infinity without attenuation). If we extend the framework to only locally $\mL^2$ functions, in turn uniqueness may fail. To cope with this problem, one has additionally to prescribe the behaviour of the solution at infinity imposing so-called \emph{radiation conditions}. These radiation conditions are expressed thanks to a modal decomposition which is obtained by using the self-adjointness of the Laplace transverse operator, so that the corresponding eigenfunctions form a Hilbert basis. Some Dirichlet to Neumann operators, enclosing the radiation conditions, can then be introduced to reduce the problem to one set in a bounded domain. Finally, well-posedness in the Fredholm sense can be proved (see \cite{McLe00} for more details on the Fredholm theory). More precisely, if uniqueness holds (which arises except for a countable set of frequencies, which corresponds in part to the trapped modes, see for example \cite{davies_parnovski,maciver_linton_maciver_zhang_porter}) then existence holds as well. The solution is said to be physical if it satisfies the limiting absorption principle: it is the limit, in a certain sense, of the solutions to the Helmholtz equation in the presence of a damping term, when this damping term tends to zero.\\
\newline
 In the present paper, for the strip governed by the Kirchhoff-Love model, we introduce radiation conditions and prove that the corresponding scattering problem  is of Fredholm type, both in the case of a clamped strip and in the case of a simply supported strip. Let us mention that some analysis of modal solutions in a strip for various boundary conditions have already been conducted (see for example \cite{norris,hu_fang_long_huang}). But a rigorous existence and uniqueness analysis of the scattering problem, whatever the boundary conditions, seems not to exist. \\  
\newline
In our article, we propose two angles of attack, depending on the boundary condition.
In the case of the simply supported strip, we benefit from the factorization of the transverse underlying differential operator to decompose any scattering solution in terms of the modes of the waveguide. Then we prescribe the radiation conditions with the help of these modes and introduce Dirichlet-to-Neumann operators --~based on these radiation conditions~-- in order to reduce the analysis to the one of a problem set in a bounded domain. Such strategy also offers a method to compute the solution numerically. However this approach is not applicable to the case of a clamped strip, see Section \ref{sub:Preliminaries_clamped} for more details. For this problem, we shall obtain the result of modal decomposition needed to express the radiation conditions at infinity using a different approach due to Kondratiev \cite{Kond67} (see also \cite{MaPl77,NaPl94,KoMR97,KoMR01}). It consists in applying the Fourier transform in the unbounded direction. Then working in weighted Sobolev spaces and using the residue theorem, we shall get our decomposition. In a second step, in order to impose radiation conditions, we shall integrate it to the functional space in which we look for the solution. To proceed, we shall work with spaces with detached asymptotics introduced in \cite{NaPl91} (see also the reviews \cite{Naza99a,Naza99b}). Let us mention that the methodology we follow to study the clamped problem could be used also to deal with the simply supported problem. We would obtain completely similar results. The goal of the present paper is first to investigate problems of thin plates in unbounded strips, as mentioned above, but also to show that when the result of Hilbert basis in the transverse section is not available, we can still use an alternative route. We hope that the successive presentation of the two methods will help the reader to get familiar with the second approach which may be less known and which requires a slightly longer analysis. For application of the technique to other situations, one may consult \cite{NaTa11,BoCCSub,Naza13b,BoCh13,NaPo17}. In \cite{Naza82,NaTa17}, periodic problems are also considered.\\ 
\newline
The outline of the article is as follows. First, we describe the setting of our problems in Section \ref{SectionSetting}. Then in Section \ref{SectionModalExpo}, we compute the modal exponents both for the simply supported and clamped cases. The results of these computations are summarized in Proposition \ref{PropoCardSimply} and Proposition \ref{propoCard}. In Section \ref{SectionSimply}, we detail the analysis for the simply supported problem. Section \ref{SectionClamped} is dedicated to the study of the clamped problem. Note that Sections \ref{SectionSimply} and \ref{SectionClamped} can be read quite independently from Section \ref{SectionModalExpo}. Finally, we justify the selection of the outgoing modes in Section \ref{SectionOutgoing} before giving some short concluding remarks in Section \ref{Conclusion}. The main results of this article are Theorem \ref{thmIsomObstacleSimply} (Fredholmness in the simply supported case) and Theorem \ref{thmIsomObstacle} (Fredholmness in the clamped case).

\section{Setting of the problem}\label{SectionSetting}

\begin{figure}[!ht]
\centering
\begin{tikzpicture}
\draw[fill=gray!30,draw=none](-2.3,0) rectangle (2.3,2);
\draw (-2.3,0)--(2.3,0);
\draw (-2.3,2)--(2.3,2);
\draw [dashed](-3,0)--(-2.3,0);
\draw [dashed](3,0)--(2.3,0);
\draw [dashed](-3,2)--(-2.3,2);
\draw [dashed](3,2)--(2.3,2);
\node at (-1.7,0.2){\small$\Om$};
\end{tikzpicture}\qquad\qquad\begin{tikzpicture}
\draw[fill=gray!30,draw=none](-2.3,0) rectangle (2.3,2);
\draw (-2.3,0)--(2.3,0);
\draw (-2.3,2)--(2.3,2);
\draw [dashed](-3,0)--(-2.3,0);
\draw [dashed](3,0)--(2.3,0);
\draw [dashed](-3,2)--(-2.3,2);
\draw [dashed](3,2)--(2.3,2);
\begin{scope}[scale=0.7,yshift=0.4cm]
\draw [fill=white] plot [smooth cycle, tension=1] coordinates {(-0.6,0.9) (0,0.5) (0.7,1) (0.5,1.5) (-0.2,1.4)};
\end{scope}
\node at (0,1){\small$\mathscr{O}$};
\node at (-1.7,0.2){\small$D$};
\begin{scope}[shift={(3,1)}]
\draw[->] (0,0)--(0.5,0);
\draw[->] (0,0)--(0,0.5);
\node at (0.65,0){\small$x$};
\node at (0,0.65){\small$y$};
\end{scope}
\end{tikzpicture}
\caption{Domains $\Om$ (left) and $D$ (right). \label{PictureDomains}} 
\end{figure}
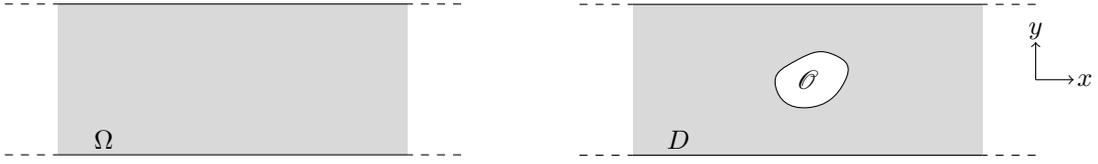
\noindent We consider a waveguide $\Omega=\{(x,y)\in\mathbb{R} \times (0;1)\}$, the boundary of which is denoted $\partial \Omega$.
Let $ \mathscr{O}\subset \Omega$ be a $\mathscr{C}^{1,1}$ domain such that $\overline{\mathscr{O}}\subset\Om$. We define $D:=\Omega \setminus \overline{\mathscr{O}}$ (see Figure \ref{PictureDomains}). We assume that the domain $D$ is occupied by a thin elastic plate described with the help of the Kirchhoff-Love model in the purely bending case. 
We will consider two kinds of boundary conditions on $\partial \Omega$:
the plate is either simply supported by $\partial \Omega$ or clamped on $\partial \Omega$, while $\mathscr{O}$ is a hole within it. In our analysis, we will study the following source term problem: find $u$ in $D$ such that
\be
\left\{
\begin{array}{rcll}
\Delta^2 u -k^4 u& = & f  & \mbox{ \rm{in} }D \\
u=Cu & = & 0 & \mbox{ \rm{on} }\partial \Omega \\
Mu=Nu & = & 0 & \mbox{ \rm{on} } \partial \mathscr{O} \\
\multicolumn{4}{c}{u \mbox{ \rm{satisfies} } \mrm{(RC)}.}
\end{array}\right.
\label{u_model}
\ee
Let us describe this system. From the physical point of view, the first equation of (\ref{u_model}) comes from the equation of the motion of the strip
\[D\Delta^2 u+\rho h \ders{u}{t}=p\]
in the time harmonic regime. Here, we have $D=Eh^3/12(1-\nu^2)$, where $E$ is the Young's modulus, $\nu \in [0;1)$ is the Poisson's ratio and $h$ is the thickness of the strip. Moreover, $\rho$ refers to the density per unit of volume and $p$ corresponds to the pressure applied to the strip. Hence the wavenumber $k$ is defined by $k^4=\rho h\omega^2/D$ and the volume source $f$ by $p/D$.
\\\\
 In the third equation of (\ref{u_model}), $M$ and $N$ are the boundary differential operators defined by
\begin{equation}\label{defOpBord}
Mu = \nu \Delta u + (1-\nu) M_0 u,\qquad\qquad Nu=-\deri{}{n}\Delta u -(1-\nu) \deri{}{s}N_0u,
\end{equation}
where $n=(n_x,n_y)$ is the outward unit normal to $\partial D$ and $s=(-n_y,n_x)$. Above, we use the notation
\[\deri{}{n} :=n_x \deri{}{x} + n_y \deri{}{y},\qquad\qquad \deri{}{s} :=-n_y \deri{}{x} + n_x\deri{}{y}.
\]
Moreover, in (\ref{defOpBord}), the operators $M_0$ and $N_0$ are respectively defined by
\[
M_0 u=\ders{u}{x}n_x^2 +2\derd{u}{x}{y}n_x n_y  + \ders{u}{y}n_y^2,\qquad N_0u=\derd{u}{x}{y}(n_x^2-n_y^2)-\bigg(\ders{u}{x}-\ders{u}{y} \bigg)n_x n_y.
\]
In order to interpret the boundary conditions in (\ref{u_model}), we recall that $D \times Mu$ is the bending moment while $D \times Nu$ is the transverse force. The boundary condition $Mu=0$ and $Nu=0$ on $\partial \mathscr{O}$ corresponds to a free obstacle, that is a hole.
\\\\
Concerning the boundary condition on $\partial \Omega$ (\emph{i.e.} the second equation of (\ref{u_model})), we shall consider the following two cases.
\begin{itemize}
\item[$i)$] When $C=M$, we have $u=0$ and $Mu=0$ on $\partial \Omega$. This corresponds to the simply supported strip.
\item[$ii)$] When $C=\partial_n$, we have $u=0$ and $\partial_n u=0$ on $\partial \Omega$. This corresponds to the clamped strip.
\end{itemize}
\noindent It should be noted that simplified expressions of $Mu$ and $Nu$ on straight parts of the boundary are
\be Mu=\ders{u}{n} + \nu \ders{u}{s},\qquad\qquad Nu=-\bigg(\frac{\partial^3 u}{\partial n^3} + (2-\nu) \frac{\partial^3 u}{\partial n\,\partial s^2} \bigg).\label{MN}\ee
This implies that in the case of a simply supported strip, the boundary condition can by simplified as $u=0$ and $\partial_{nn}u=0$ on $\partial \Omega$, or equivalently $u=0$ and $\Delta u=0$ on $\partial \Omega$.
\\\\ Finally, $\mrm{(RC)}$ stands for the radiation conditions which will be specified later on, for the simply supported and the clamped cases.
\\\\
The goal of the present article is to study the well-posedness of Problem (\ref{u_model}). For $k$ larger than a given threshold, in order to obtain well-posedness for (\ref{u_model}), we will have to impose radiation conditions to prescribe the behaviour of the solution at infinity. To proceed, we will show that every function satisfying the first two equations of (\ref{u_model}) decomposes on what we call the modes of the waveguide. These modes are computed in the next section, for the simply supported case and then the clamped case. Later on, they will be helpful to define the radiation conditions.
\section{Computation of modal exponents}\label{SectionModalExpo}
\noindent
The modes of the waveguide are defined as the functions of the form $u(x,y)=e^{\la x}\varphi(y)$, where $\lambda\in\Cplx$ and where $\varphi$ is a function to determine, which satisfy the equations $\Delta^2 u -k^4 u=0$ in $\Omega$ (the reference strip without the obstacle) and $u=Cu=0$ on $\partial \Omega$. In this section, we compute the modal exponents, that is the values of $\lambda\in\Cplx$ such that $u(x,y)=e^{\la x}\varphi(y)$ is a mode. The results of the computations are summarized in Proposition \ref{lambda} and Proposition \ref{PropositionSpectrumSymbol}. The reader who wishes to skip details can proceed directly to Sections \ref{SectionSimply} and \ref{SectionClamped}.\\
\newline
Setting $I:=(0;1)$, one finds that $u(x,y)=e^{\la x}\varphi(y)$ is a mode if and only if, the pair $(\lambda,\varphi)\in\Cplx \times \mH^2(I) \setminus \{0\}$ solves, depending on the problem considered, 
\be 
{i)}\quad\left\{
\begin{array}{rcll}
(\lambda^2+d_{yy})^2 \varphi-k^4\varphi & = & 0 & \mbox{ in } I\\
\varphi = d_{yy} \varphi& = & 0 & \mbox{ on } \partial I
\end{array}
\right.\qquad{ii)}\quad\left\{
\begin{array}{rcll}
(\lambda^2+d_{yy})^2 \varphi-k^4\varphi & = & 0 & \mbox{ in } I\\
\varphi = d_{y} \varphi& = & 0 & \mbox{ on } \partial I. 
\end{array}
\right.
\label{fortInitial}
\ee
The first problem is related to the simply supported plate while the second one is related to the clamped plate.
\\\\
Defining the Hilbert spaces $\mH_0^1(I):=\{\psi\in\mH^1(I)\,|\,\psi=0\mbox{ on }\partial I\}$ and $\mH_0^2(I):=\{\psi\in\mH^2(I)\,|\,\psi=d_y\psi=0\mbox{ on }\partial I\}$, the variational formulations of these two spectral problems write
\begin{enumerate}
\item[$i)$] Find $(\la,\varphi) \in  \mathbb{C} \times \mH_0^1(I) \cap \mH^2(I) \setminus \{0\}$ such that
\begin{equation}\label{SpectralVaria1}
\int_I (\lambda^2\varphi+d_{yy} \varphi)(\lambda^2\overline{\psi}+d_{yy} \overline{\psi})-k^4\varphi\overline{\psi}\,dy=0, \qquad \forall \psi \in \mH_0^1(I) \cap \mH^2(I).
\end{equation} 
\item[$ii)$] Find $(\la,\varphi) \in \mathbb{C} \times \mH_0^2(I) \setminus \{0\}$ such that
\begin{equation}\label{SpectralVaria2}
\int_I (\lambda^2\varphi+d_{yy} \varphi)(\lambda^2\overline{\psi}+d_{yy} \overline{\psi})-k^4\varphi\overline{\psi}\,dy=0, \qquad \forall \psi \in \mH_0^2(I).
\end{equation}
\end{enumerate}
Denoting $(\mH_0^1(I) \cap \mH^2(I))^\ast$ (resp. $\mH^{-2}(I)$) the topological dual space of $\mH_0^1(I) \cap \mH^2(I)$ (resp. $\mH_0^2(I)$), the underlying fourth-order differential operator $\mathscr{L}(\lambda)$ appearing in the analysis of (\ref{SpectralVaria1}), (\ref{SpectralVaria2}) is,  alternatively:
\[
{i)}\quad \begin{array}{rrcl}
\mathscr{L}(\lambda) : &\mH_0^1(I) \cap \mH^2(I) &\rightarrow &(\mH_0^1(I) \cap \mH^2(I))^\ast\\
& \varphi & \mapsto & \mathscr{L}(\lambda) \varphi=(\lambda^2+d_{yy})^2 \varphi-k^4\varphi
\end{array}
\] 
or
\begin{equation}\label{defOperator}
{ii)}\quad\qquad\qquad\begin{array}{rrcl}
\mathscr{L}(\lambda) : &\mH_0^2(I) &\rightarrow &\mH^{-2}(I)\\
& \varphi & \mapsto & \mathscr{L}(\lambda) \varphi=(\lambda^2+d_{yy})^2 \varphi-k^4\varphi.
\end{array}
\end{equation} 
For any of the two spectral problems, 
if $(\la,\varphi)$ is a solution then $\la$ is called an eigenvalue of the symbol $\mathscr{L}$ while $\varphi$ is called an eigenfunction of $\mathscr{L}$.
We denote $\Lambda$ the set of all eigenvalues of $\mathscr{L}$. This set will be referred to as the set of modal exponents.  Let us now solve these two spectral problems. We begin with the first one which, by using a factorization of the operator $\mathscr{L}(\lambda)$ and the very special nature of the boundary condition, is much simpler.
\\\\
In this article, the complex square root will be chosen so that for $z=\rho e^{i\nu}$, with $\rho\ge0$ and $\nu\in[0;2\pi)$, we have $\sqrt{z}=\sqrt{\rho}e^{i\nu/2}$. In particular, we always have $\Im m\,z\ge0$.
\subsection{Modal exponents in the simply supported case}\label{paragraphModesSimply}

In order to solve (\ref{fortInitial})-(\emph{i}), or equivalently (\ref{SpectralVaria1}), first we introduce the eigenvalues $\mu_n$ and eigenfunctions $\theta_n$ of the auxiliary spectral problem: find $(\mu,\theta) \in \mathbb{C} \times \mH^1(I) \setminus \{0\}$ such that
\be
\left\{
\begin{array}{rcll}
d_{yy}\theta+\mu\,\theta & = & 0  & \mbox{ \rm{in} }I\\
\theta & = & 0 & \mbox{ \rm{on} } \partial I. 
\end{array}\right.
\label{transverse}
\ee
A straightforward computation leads to $\mu_p=\pi^2 p^2$ and $\theta_p(y)=\sqrt{2}\sin(\pi p y)$ for $p \in \N^\ast:=\{1,2,\dots\}$.
Let us remark that the $\mu_p$ form a positive and increasing sequence of real numbers that tends to $+\infty$ while the family $(\theta_p)$ forms a complete orthonormal basis of $\mL^2(I)$. 
\begin{proposition}
\label{lambda}
Assume that $k>0$. Then the set of modal exponents $\Lambda$ for (\ref{fortInitial})-(i) is given by
\begin{equation}\label{etan_gamman}
\Lambda= \left\{\pm i\eta_p,\,p \in \mathbb{N}^\ast\right\} \cup \left\{\pm \gamma_p,\,p \in \mathbb{N}^\ast\right\}\ \mbox{ with }\ \eta_p:=\sqrt{k^2-\pi^2 p^2}\ \mbox{ and }\ \gamma_p:= \sqrt{k^2+\pi^2 p^2}.
\end{equation}
\end{proposition}
\begin{proof}
Let us consider some solution $(\la,\varphi)$ to the spectral problem 
(\ref{fortInitial})-(\emph{i}) and let us define
\[\tilde{\varphi}:=(\la^2+d_{yy})\varphi-k^2 \varphi,\qquad\qquad \check{\varphi}:=(\la^2+d_{yy})\varphi+k^2 \varphi.\]
Using that $d_{yy}\varphi=0$ on $\partial I$, we observe that $\tilde{\varphi}$ and $\check{\varphi}$ satisfy the following problems:
\[\left\{
\begin{array}{rcll}
(\la^2+d_{yy})\tilde{\varphi} + k^2 \tilde{\varphi} & = & 0 & \mbox{ in } I\\
\tilde{\varphi} & = & 0 & \mbox{ on } \partial I 
\end{array}
\right.\qquad {\rm and}\qquad
\left\{
\begin{array}{rcll}
(\la^2+d_{yy})\check{\varphi} - k^2 \check{\varphi} & = & 0 & \mbox{ in } I\\
\check{\varphi} & = & 0 & \mbox{ on } \partial I .
\end{array}
\right.\]
Introducing the solutions $(\mu_p,\theta_p)$ to Problem (\ref{transverse}), 
we only have two possibilities: either $\la^2=\mu_p-k^2$, $\tilde{\varphi}=\tilde{C} \theta_p$, $\check{\varphi}=0$ for some $p \in \mathbb{N}^\ast$, $\tilde{C}\in\Cplx$, or $\la^2=\mu_p+k^2$, $\check{\varphi}=\check{C}\theta_p$, $\tilde{\varphi}=0$ for some $p \in \mathbb{N}^\ast$, $\check{C}\in\Cplx$.
Conversely, for any value $\la$ such that either $\la^2=\mu_p-k^2$ or $\la^2=\mu_p+k^2$, by choosing $\varphi=\theta_p$, one finds that $(\la,\varphi)$ is an eigenpair of (\ref{fortInitial})-(\emph{i}).
\end{proof}
\noindent Now let us focus our attention on the set $\Lambda \cap \mathbb{R}i$. The reason is that if $\la \in \Lambda \cap \mathbb{R}i \setminus \{0\}$ and $\varphi$ corresponds to a non zero element of $\ker\,\mathscr{L}(\lambda)$, then the so-called mode $(x,y) \mapsto e^{\la x}\varphi(y)$ is propagating.  Such modes play a particular role in the definition of the radiation conditions and the well-posedness of the initial problem. 
Let us first remark from Proposition \ref{lambda} that $0 \in \Lambda$ if and only if there exists $n \in \mathbb{N}^\ast$ such that $k= n\pi$.
These particular values are the so-called threshold wavenumbers.
We have the following proposition, the proof of which is straightforward. We denote $\left\lfloor \cdot \right\rfloor$ the floor function. 
\begin{proposition}\label{PropoCardSimply}
For $k\in(0;\pi)$, we have $\Lambda\cap\R i=\emptyset$. For $k\ge\pi$, we have 
\[\Lambda \cap \mathbb{R}i=\left\{\pm i \eta_p,\quad p=1,\cdots,
\left\lfloor k/\pi \right\rfloor\right\}.
\]
This implies $\mrm{card}\,(\Lambda\cap\R i)=2n$ when $k\in(n\pi;(n+1)\pi)$, $n\in\N^{\ast}$ and $\mrm{card}\,(\Lambda\cap\R i)=2n-1$ when $k=n\pi$, $n\in\N^{\ast}$.
\end{proposition}

\subsection{Modal exponents in the clamped case}

In this paragraph, we solve (\ref{fortInitial})-(\emph{ii}), or equivalently (\ref{SpectralVaria2}). We assume that $k>0$ is given. We remark that for $\lambda\in\Cplx$ such that $\lambda^4\ne k^4$, the linearly independent functions $a_1$, $a_2$ such that 
\begin{equation}\label{defBase1}
a_1(y)=\cfrac{\sin(\sqrt{\lambda^2+ k^2}y)}{\sqrt{\lambda^2+k^2}}-\cfrac{\sin(\sqrt{\lambda^2-k^2}y)}{\sqrt{\lambda^2-k^2}},\qquad a_2(y)=\cos(\sqrt{\lambda^2+ k^2}y)-\cos(\sqrt{\lambda^2-k^2}y)
\end{equation}
satisfy the first equation of (\ref{fortInitial})-(\emph{ii}) as well as the boundary conditions $\varphi(0)=d_y\varphi(0)=0$. On the other hand, for $\lambda\in\Cplx$ such that $\lambda^4=k^4$, the linearly independent functions $b_1$, $b_2$ defined by
\begin{equation}\label{defBase2}
b_1(y)=\cfrac{\sin(\sqrt{2}\lambda y)}{\sqrt{2}\lambda }-y,\qquad b_2(y)=\cos(\sqrt{2}\lambda y)-1.
\end{equation}
are solutions of the first equation of (\ref{fortInitial})-(\emph{ii}) satisfying $\varphi(0)=d_y\varphi(0)=0$. In the analysis below, we will meet the following two sets
\begin{eqnarray}
\label{defSetK}\mathscr{K}&\hspace{-0.2cm}:=&\hspace{-0.2cm}\Big\{\cfrac{\pi}{\sqrt{2}}\,\sqrt{m^2-n^2},\,\mbox{ with $m,\,n\in\N^{\ast}$, $m>n$, such that $m-n$ is even} \Big\}\\[2pt]
\label{defSetLambdaPart}
\Lambda_{\mrm{part}}&\hspace{-0.2cm}:=&\hspace{-0.2cm}\Big\{\cfrac{\pi}{\sqrt{2}}\,\sqrt{m^2+n^2},\,\mbox{ with $m,\,n\in\N^{\ast}$, $m>n$, such that $m-n$ is even} \Big\}.
\end{eqnarray}
In the proposition below, we give a characterization of the set of modal exponents $\Lambda$ for the clamped problem. We remind the reader that the geometric multiplicity of an eigenvalue $\lambda$ of $\mathscr{L}$ is by definition equal to $\dim\,\ker\,\mathscr{L}(\lambda)$.

\begin{proposition}\label{PropositionSpectrumSymbol}
Assume that $k>0$. 
Let $\Lambda$ refer here to the set of modal exponents for (\ref{fortInitial})-(ii).\\[3pt]
1) The number $\lambda\in\Cplx$ such that $\lambda^4\ne k^4$ belongs to $\Lambda$ if and only if $\lambda$ satisfies
\begin{equation}\label{dispersion1}
\hspace{-0.2cm}\left(\sqrt{\cfrac{\lambda^2+ k^2}{\lambda^2-k^2}}+\sqrt{\cfrac{\lambda^2- k^2}{\lambda^2+k^2}}\right)\sin(\sqrt{\lambda^2+ k^2})\sin(\sqrt{\lambda^2-k^2})=2-2\cos(\sqrt{\lambda^2+ k^2})\cos(\sqrt{\lambda^2-k^2}).\hspace{-0.1cm}
\end{equation}
Moreover, if $k\notin \mathscr{K}$ (see definition (\ref{defSetK}) above), then for all $\lambda\in\Lambda$, we have $\ker\,\mathscr{L}(\lambda)=\mrm{span}(\varphi_0)$ (geometric multiplicity equal to one) with $\varphi_0(y)=A_1\,a_1(y)+A_2\,a_2(y)$. Here $(A_1,A_2)^{\top}$ is an eigenvector of the matrix $\mathbb{A}(\lambda)$ defined in (\ref{defMatSystem}).\\[5pt]
If $k\in \mathscr{K}$, then for $\lambda\in\Lambda_{\mrm{part}}\cap\Lambda$ (see (\ref{defSetLambdaPart})), we have $\ker\,\mathscr{L}(\lambda)=\mrm{span}(a_1,\,a_2)$ (geometric multiplicity equal to two). For $\lambda\in\Lambda\setminus\Lambda_{\mrm{part}}$, we have $\ker\,\mathscr{L}(\lambda)=\mrm{span}(\varphi_0)$ (geometric multiplicity equal to one) with $\varphi_0(y)=A_1\,a_1(y)+A_2\,a_2(y)$ (again here $(A_1,A_2)^{\top}$ is an eigenvector of the matrix $\mathbb{A}(\lambda)$).\\
\newline
2) The number $\lambda\in\Cplx$ such that $\lambda^4= k^4$ belongs to $\Lambda$ if and only if $\lambda$ satisfies
\begin{equation}\label{dispersion2}
\sqrt{2}\lambda\sin(\sqrt{2}\lambda)=2-2\cos(\sqrt{2}\lambda).
\end{equation}
In that case, we have $\ker\,\mathscr{L}(\lambda)=\mrm{span}(\varphi_0)$ (geometric multiplicity equal to one) with $\varphi_0(y)=B_1\,b_1(y)+B_2\,b_2(y)$.  Here $(B_1,B_2)^{\top}$ is an eigenvector of the matrix $\mathbb{B}(\lambda)$ defined in (\ref{defMatSystem2}).
\end{proposition}
\begin{proof}
1) First we study the eigenvalues $\lambda\in\Cplx$ of $\mathscr{L}$ such that $\lambda^4\ne k^4$. As 
\[
	\text{dim}\;\Big\{\varphi,\; (\lambda^2+d_{yy})^2\varphi-k^4\varphi=0,\;\varphi(0)=\varphi'(0)=0\Big\}=2,
\] 
if $\varphi$ satisfies $\mathscr{L}(\lambda)\varphi=0$ then there are constants $A_1$, $A_2\in\Cplx$ such that $\varphi(y)=A_1\,a_1(y)+A_2\,a_2(y)$ where $a_1$, $a_2$ are defined in (\ref{defBase1}). Writing the two boundary conditions at $y=1$, we obtain that $\varphi$ is a non-zero function satisfying $\mathscr{L}(\lambda)\varphi=0$ if and only if the matrix 
\begin{equation}\label{defMatSystem}
\mathbb{A}(\lambda):=\left(\begin{array}{cc}
a_1(1) & a_2(1)\\[4pt]
a'_1(1) & a'_2(1)
\end{array}\right)
\end{equation}
has a non trivial kernel. An explicit computation shows that $\det\mathbb{A}(\lambda)=0$ if and only if (\ref{dispersion1}) holds. Moreover, one sees that the geometric multiplicity of $\lambda$ coincides with $\dim\,\ker\,\mathbb{A}(\lambda)$. Clearly, if $\lambda\in\Lambda$, then $\dim\,\ker\,\mathbb{A}(\lambda)=1$ except if $\mathbb{A}(\lambda)=0$ (in this case $\dim\,\ker\,\mathbb{A}(\lambda)=2$). Assume that $\mathbb{A}(\lambda)=0$. Then in particular, we must have $a_1(1)=a'_2(1)=0$. Using expressions (\ref{defBase1}), this implies $\sin(\sqrt{\lambda^2+k^2})=\sin(\sqrt{\lambda^2-k^2})=0$ leading to $k=\pi\sqrt{m^2-n^2}/\sqrt{2}$ and $\lambda=\pi\sqrt{m^2+n^2}/\sqrt{2}$, where $m$, $n\in\N^{\ast}$ are such that $m>n$. The additional constrain $a_2(1)=a'_1(1)=0$ imposes that $m$, $n$ must have same parity. This leads to the definition of the set $\mathscr{K}$ in (\ref{defSetK}) and to the statement of the proposition.\\[5pt]
2) Then we study the eigenvalues $\lambda\in\Cplx$ of $\mathscr{L}$ such that $\lambda^4=k^4$. If $\varphi$ satisfies $\mathscr{L}(\lambda)\varphi=0$ then there are constants $B_1$, $B_2\in\Cplx$ such that $\varphi(y)=B_1\,b_1(y)+B_2\,b_2(y)$ where $b_1$, $b_2$ are defined in (\ref{defBase2}). Writing the two boundary conditions at $y=1$, we obtain that $\varphi$ is a non-zero function satisfying $\mathscr{L}(\lambda)\varphi=0$ if and only if the matrix 
\begin{equation}\label{defMatSystem2}
\mathbb{B}(\lambda):=\left(\begin{array}{cc}
b_1(1) & b_2(1)\\[4pt]
b'_1(1) & b'_2(1)
\end{array}\right)
\end{equation}
has a non trivial kernel. An explicit computation shows that $\det\mathbb{B}(\lambda)=0$ if and only (\ref{dispersion2}) holds. Moreover, one can check that one has always $\mathbb{B}(\lambda)\ne0$. As a consequence, if $\lambda$ is an eigenvalue of $\mathscr{L}$ such that $\lambda^4=k^4$, then $\dim\,\ker\,\mathscr{L}(\lambda)=1$. 
\end{proof}
\noindent In the remaining part of the paragraph, we focus our attention on the set $\Lambda\cap\R i$, in other words on the propagating modes. \\
From (\ref{SpectralVaria2}), we remark that if $\lambda$ belongs to $\Lambda$, then $-\lambda$ is also an element of $\Lambda$. Therefore, it is sufficient to study $\Lambda\cap[0;+i\infty)$. In the proof of Lemma \ref{resultPrelim} below, we will see that $\Lambda\cap[ik;+i\infty)=\emptyset$. As a consequence, we can look for $\lambda\in\Lambda$ writing as $\lambda=i\tau k$ with $\tau\in[0;1)$. From (\ref{dispersion1}), we see that we must have $h_{k}(\tau)=0$ with 
\begin{equation}\label{dispersion3}
\begin{array}{l}
h_{k}(\tau)=\left(\sqrt{\cfrac{1-\tau^2}{1+\tau^2}}-\sqrt{\cfrac{1+\tau^2}{1-\tau^2}}\right)\sin(k\sqrt{1-\tau^2})\sinh(k\sqrt{1+\tau^2})\\[18pt]
\hspace{2cm}-(2-2\cos(k\sqrt{1-\tau^2})\cosh(k\sqrt{1+\tau^2})).
\end{array}
\end{equation}
In Corollary \ref{coroInvertibility}, we show that such dispersion relation is satisfied only by a finite number of $\tau\in[0,1)$.
From Proposition \ref{PropositionSpectrumSymbol}, we know that if $\lambda$ belongs to $\Lambda\cap\R i$, then its geometric multiplicity is equal to one. In the following, we will also need to know the algebraic multiplicity of $\lambda$ (see the definition e.g. in \cite[\S5.1.1]{KoMR97}). 
\begin{proposition}\label{PropositionAlgMult}
Assume that $k>0$ is given. If $\lambda\in\Lambda\cap \R i\setminus\{0\}$, then its algebraic multiplicity is equal to one. If $\lambda=0\in\Lambda$ then its algebraic multiplicity is equal to two. 
\end{proposition}
\begin{proof}
We remind the reader that for $\lambda\in\Cplx$, we denote $\mathscr{L}(\lambda):\mH^2_0(I)\to \mH^{-2}(I)$ the operator such that $\mathscr{L}(\lambda)\varphi =d^4_y\varphi+2\lambda^2d^2_y\varphi+(\lambda^4-k^4)\varphi$. Assume that $\lambda_0\in\Lambda$. Then there is $\varphi_0\not\equiv0$ such that $\mathscr{L}(\lambda_0)\varphi_0=0$. Assume that the algebraic multiplicity of $\lambda_0$ is larger than one. By definition, this means that there is $\varphi_1\in\mH^2_0(I)$, with $\varphi_1\not\equiv0$, such that 
\begin{equation}\label{identityAlgebraicMulti}
\mathscr{L}(\lambda_0)\varphi_1+\cfrac{d\mathscr{L}}{d\lambda}|_{\lambda=\lambda_0}\,\varphi_0=0\qquad\Leftrightarrow\qquad 4\lambda_0(d_y^2\varphi_0+\lambda_0^2\varphi_0)=-\mathscr{L}(\lambda_0)\varphi_1.
\end{equation}
Multiplying by $\overline{\varphi_0}$ the identities of (\ref{identityAlgebraicMulti}) and integrating by parts, we obtain 
\begin{equation}\label{AlgMul1}
4\lambda_0(\|d_y \varphi_0\|_{\mrm{L}^2(I)}^2-\lambda_0^2\|\varphi_0\|_{\mrm{L}^2(I)}^2)=-\langle \mathscr{L}(\lambda_0)\varphi_1,\overline{\varphi_0}\rangle_I=-\langle \mathscr{L}(\lambda_0)\overline{\varphi_0},\varphi_1\rangle_I.
\end{equation}
where  $\langle\cdot,\cdot\rangle_I$ stands for the bilinear duality pairing between $\mH^{-2}(I)$ and $\mH^2_0(I)$.
Assume that $\lambda_0\in\Lambda\cap \R i\setminus\{0\}$. Then we have $\mathscr{L}(\lambda_0)\overline{\varphi_0}=\overline{\mathscr{L}(\lambda_0)\varphi_0}=0$. Therefore, since  identity (\ref{AlgMul1}) leads to $\varphi_0\equiv0$. This is absurd and shows that the algebraic multiplicity of the elements of $\Lambda\cap \R i\setminus\{0\}$ is equal to one.\\
\newline
Now, let us focus on the algebraic multiplicity of $\lambda_0=0$, assuming that $\lambda_0=0$ belongs to $\Lambda$. From Equation (\ref{identityAlgebraicMulti}), by taking $\varphi_1=\varphi_0$, we see that its algebraic multiplicity is at least two. Assume that it is larger than two. Then there is $\varphi_2\in\mH^2_0(I)$, with $\varphi_2\not\equiv0$, such that 
\begin{equation}\label{identityAlgebraicMultiBis}
\mathscr{L}(0)\varphi_2+\cfrac{d\mathscr{L}}{d\lambda}|_{\lambda=0}\,\varphi_1+\cfrac{d^2\mathscr{L}}{d\lambda^2}|_{\lambda=0}\,\varphi_0=0\qquad\Leftrightarrow\qquad 4\,d_y^2\varphi_0=-\mathscr{L}(0)\varphi_2.
\end{equation}
Multiplying by $\overline{\varphi_0}$ the identities of (\ref{identityAlgebraicMultiBis}) and integrating by parts, this implies 
\[
4 \|d_y \varphi_0\|_{\mrm{L}^2(I)}^2=\langle \mathscr{L}(0)\overline{\varphi_0},\varphi_2\rangle_I=\langle \overline{\mathscr{L}(0)\varphi_0},\varphi_2\rangle_I=0.
\]
 Thus, we obtain a contradiction and we can conclude that if $\lambda_0=0$ is an eigenvalue of $\mathscr{L}$, then its algebraic multiplicity is equal to two.
\end{proof}
\begin{remark}
	We specify the algebraic multiplicity of modal exponents in the previous proposition because it will be required in the proof of Proposition \ref{propoDecomposition} (where the Residue theorem is implicitly used).
	\end{remark}
\noindent In the following, for a given $k>0$, we will need to know the cardinal of the set $\Lambda\cap\R i$. From (\ref{dispersion3}), we find that $0$ belongs to $\Lambda$ if $k>0$ is such that 
\begin{equation}\label{EquationLambdaNul}
h_{k}(0)=0\qquad\Leftrightarrow\qquad \cos(k)\cosh(k)=1.
\end{equation}
The set of $k>0$ such that (\ref{EquationLambdaNul}) holds (threshold wavenumbers) forms an increasing unbounded sequence 
\begin{equation}\label{defThreshold}
0<k_1<k_2<\dots<k_n<\dots\qquad \mbox{ such that }\qquad k_n\underset{n\to+\infty}{\sim}\pi/2+n\pi.
\end{equation}
Taking $\lambda=0$ in (\ref{fortInitial})-(\emph{ii}), we observe that $k_n^4$ corresponds to the $n^{th}$ eigenvalue of the problem 
\be \label{EigenPbBilapl1D}
\left\{
\begin{array}{rcll}
 d^4_y\varphi-\mu\,\varphi& = & 0 & \mbox{ in } I\\
\varphi=d_{y} \varphi & = & 0 & \mbox{ on } \partial I .
\end{array}
\right.
\ee
In the proposition below, we prove that for all $n \in \mathbb{N}^\ast$, the threshold wavenumbers $k_n$ for the clamped strip are larger than the threshold wavenumbers $n\pi$ for the simply supported strip.
\begin{proposition}
For all $n \in \mathbb{N}^\ast$, we have $k_n \geq n\pi$.
\end{proposition}
\begin{proof}
By the {\rm min-max} principle, the $n^{th}$ eigenvalue of the problem 
(\ref{EigenPbBilapl1D}) is given by
\[\mu_n:=\min_{V_n \in \mathcal{V}_n(\mH_0^2(I))} \max_{\varphi \in V_n} \frac{\|d_{yy} \varphi\|^2_{\mL^2(I)}}{\|\varphi\|^2_{\mL^2(I)}},\]
where $\mathcal{V}_n(\mH)$ denotes the set of all $n$ dimensional subspaces of $\mH$.
Since $\mH_0^2(I) \subset \mH_0^1(I) \cap \mH^2(I)$, we have
\[\mu_n \geq \tilde{\mu}_n:=\min_{V_n \in \mathcal{V}_n(\mH_0^1(I) \cap \mH^2(I))} \max_{\varphi \in V_n} \frac{\|d_{yy} \varphi\|^2_{\mL^2(I)}}{\|\varphi\|^2_{\mL^2(I)}}.\]
We observe that $\tilde{\mu}_n$ coincides with the $n^{th}$ eigenvalue of the problem 
\[
\left\{
\begin{array}{rcll}
 d^4_y\varphi-\tilde{\mu}\,\varphi& = & 0 & \mbox{ in } I\\
\varphi=d_{yy} \varphi & = & 0 & \mbox{ on } \partial I,
\end{array}
\right.
\]
that is $\tilde{\mu}_n=n^4\pi^4 $. We end up with $k_n=(\mu_n)^{1/4} \geq n\pi$ for all $n \geq 1$.
\end{proof}
\noindent The approximated values of the first $k_n$ are given in Figure \ref{ValuesKn}.
\begin{figure}[!ht]
\centering
\renewcommand{\arraystretch}{1.4}
\begin{tabular}{|c|c|c|}
 \hhline{~|*2-}
\multicolumn{1}{c|}{\cellcolor{white}} & \multicolumn{1}{c|}{\cellcolor{Gray}$k_n$} & \multicolumn{1}{c|}{\cellcolor{Gray}$\pi/2+n\pi$} \\
\hline 
$n=1$ & $ $4.730040745$   $ & $4.712388981$  \\
\hline 
$n=2$ & $ 7.853204624   $ & $ 7.853981635 $ \\
\hline 
$n=3$ & $ 10.99560784  $ & $ 10.99557429 $\\
\hline 
$n=4$ & $ 14.13716549   $ & $ 14.13716694 $\\
\hline 
$n=5$ & $ 17.27875966   $ & $ 17.27875960 $\\
\hline
\end{tabular}
\caption{Approximated values of the first $k_n$. \label{ValuesKn}} 
\end{figure}

\begin{proposition}\label{propoCard}
For $k\in(0;k_1)\Leftrightarrow k^4<\mu_1$ where $\mu_1$ is the first eigenvalue of Problem (\ref{EigenPbBilapl1D}), we have $\Lambda\cap\R i=\emptyset$. For $k\in(k_n;k_{n+1})$, $n\in\N^{\ast}$, we have $\mrm{card}\,(\Lambda\cap\R i)=2P$ where $P$ is the number of zeros of the function $h_{k}(\cdot)$ defined in (\ref{dispersion3}) on $(0;1)$. For $k=k_n$, $n\in\N^{\ast}$, we have $\mrm{card}\,(\Lambda\cap\R i)=2P-1$ where $P$ is the number of zeros of the function $h_{k}(\cdot)$ on $(0;1)$.
\end{proposition}
\begin{remark}
Numerically, it seems that $P=n$ as in the simply supported case.
\end{remark}
\section{Well-posedness in the simply supported case}\label{SectionSimply}
In this section, we suppose that $k$ is not a threshold wavenumber, \emph{i.e.} $k\notin \mathbb{N}\pi$.
\subsection{Construction of Dirichlet-to-Neumann operators}\label{paragraphDtN}
In order to study Problem (\ref{u_model}) in the case when $C=M$, let us first consider the following system of equations set in the reference strip (without hole):
\be
\left\{
\begin{array}{rcll}
\Delta^2 u -k^4 u& = & 0  & \mbox{ \rm{in} }\Omega\\
u =\Delta u& = & 0 & \mbox{ \rm{on} }\partial \Omega.
\end{array}\right.
\label{modes}
\ee
We remind the reader that since $\partial\Om$ is made of straight lines, we have $u=Mu=0\mbox{ on }\partial\Om\Leftrightarrow u=\Delta u=0$ on $\partial\Om$ (see (\ref{MN})). In (\ref{modes}), we do not prescribe any behaviour at infinity. As a consequence, this problem can have non zeros solutions. Let us compute them. Noting (again) that $\Delta^2  -k^4 =(\Delta -k^2)(\Delta +k^2)$ and  defining
\begin{equation}\label{SomDif}
\tilde{u}:=(\Delta -k^2)u,\qquad \qquad \check{u}:=(\Delta+k^2)u,\end{equation}
we see that $\tilde{u},\check{u}$ solve the problems
\[
\left\{
\begin{array}{rcll}
\Delta \tilde{u} +k^2 \tilde{u}& = & 0  & \mbox{ \rm{in} }\Omega \\
\tilde{u} & = & 0 & \mbox{ \rm{on} }\partial \Omega 
\end{array}\right.
\qquad\quad {\rm and} \qquad\quad
\left\{
\begin{array}{rcll}
\Delta \check{u} -k^2 \check{u}& = & 0  & \mbox{ \rm{in} }\Omega\\
\check{u} & = & 0 & \mbox{ \rm{on} }\partial \Omega.
\end{array}\right.
\]
Using that the family $(\theta_p)$ of the eigenfunctions of Problem (\ref{transverse}) forms a Hilbert basis of $\mL^2(I)$, we can decompose $\tilde{u}$, $\check{u}$ as 
\[\tilde{u}(x,y)=\sum_{p=1}^{+\infty}\tilde{u}_p(x)\theta_p(y),\qquad\qquad    \check{u}(x,y)=\sum_{p=1}^{+\infty} \check{u}_p(x)\theta_p(y).\]
Then we find that the $\tilde{u}_p$, $\check{u}_p$ satisfy 
\[d_{xx}\tilde{u}_p+(k^2-\mu_p) \tilde{u}_p=0 \qquad{\rm and}\qquad  d_{xx}\check{u}_p-(k^2+\mu_p)\check{u}_p=0  \quad \rm{in}\quad  \mathbb{R}.\]
Since $k\notin \mathbb{N}\pi$, we obtain that $\tilde{u}$, $\check{u}$ are given by
\[
\dsp\tilde{u}(x,y) =\dsp \sum_{p=1}^{+\infty}\left(a_p e^{i\eta_p x} + b_p e^{-i\eta_p x}\right)\theta_p(x)\qquad{\rm and}\qquad  \dsp \check{u}(x,y)= \dsp \sum_{p=1}^{+\infty}\left(c_p e^{-\gamma_p x} + d_p e^{\gamma_p x}\right)\theta_p(u),
\]
where $\eta_p,\gamma_p$ are defined in (\ref{etan_gamman}) and where  $a_p,b_p,c_p,d_p$ are complex numbers. Observing that $u=(\check{u}-\tilde{u})/2k^2$ (see (\ref{SomDif})), we deduce that the general form of the solutions to Problem (\ref{modes}) is
\be 
u(x,y)=\sum_{p=1}^{+\infty}\left(a_p e^{i\eta_p x} + b_p e^{-i\eta_p x} + c_p e^{-\gamma_p x} + d_p e^{\gamma_p x}\right)\theta_p(y),
\label{general_form}\ee
with new complex numbers $a_p,b_p,c_p,d_p$.
\begin{remark}\label{RmkBasis}
We see that (\ref{general_form}) is an expansion on the modes
$e^{\la x}\varphi(y)$ computed in \S\ref{paragraphModesSimply} (here $\la$ belongs to $\Lambda$, the set of modal exponents given in Proposition \ref{lambda}, and $\varphi\in\ker\,\mathscr{L}(\lambda)$). Note that this strong result of modal decomposition has been obtained thanks to the fact that the family $(\theta_p)$ forms a Hilbert basis of $\mL^2(I)$. 
\end{remark}
\noindent For $k\in(0;\pi)$ all the modes appearing in  (\ref{general_form}) are exponentially growing at one end of $\Om$ and exponentially decaying at the other end. In this case, we shall look for  solutions to (\ref{u_model}) which are exponentially decaying at infinity. For $k\in(n\pi;(n+1)\pi)$ with $n\in\N^{\ast}$, the modes $e^{\pm i\eta_p x}\,\theta_p(y)$, $p=1,\dots,n$, are propagating while the other ones are exponentially growing at one end of $\Om$ and exponentially decaying at the other end. 
\\\\
In the sequel, we will say that 
\be\label{eq:rightgoing}
	u \text{ is rightgoing iff}\quad\begin{array}{|l} \dsp\text{for some $L>0$ and some complex numbers $a^+_p,b^+_p$},\\
	\dsp u(x,y)=\sum_{p=1}^{+\infty}\left(a^+_p e^{+ i\eta_p x} + b^+_p e^{-\gamma_p x}\right)\theta_p(y)\qquad\mbox{ for } x\ge L,\end{array}
\ee
\be\label{eq:leftgoing}
	u\text{ is leftgoing iff}\quad \begin{array}{|l} \dsp\text{for some $L>0$ and some complex numbers $a^-_p,b^-_p$},\\
	\dsp u(x,y)=\sum_{p=1}^{+\infty}\left(a^-_p e^{- i\eta_p x} + b^-_p e^{\gamma_p x}\right)\theta_p(y)\qquad\mbox{ for } x\le -L,\end{array}
\ee
\be\label{outgoing}\tag{RC}
	u \text{ is outgoing iff } u \text{ is rightgoing and leftgoing.}
\ee
\noindent This terminology will be justified in Section \ref{selection}. Equivalently, a function $u$ satisfies the radiation conditions or is outgoing.
\\\\
Now, we introduce adapted Dirichlet-to-Neumann (DtN) operators in order to enclose this outgoing behaviour. On the transverse sections $\Sigma_{\pm L}:=\{\pm L\}\times(0;1)$, we define for $j \in \{-3/2,-1/2,1/2,3/2\}$, the spaces
\[\mH^j(\Sigma_{\pm L})=\{u|_{\Sigma_{\pm L}}\,|\,u \in \mH^j(\{\pm L\}\times\mathbb{R})\},\quad  \tilde{\mH}^j(\Sigma_{\pm L})=\{u \in \mH^j(\{\pm L\}\times\mathbb{R})\,|\, {\supp}(u) \in \overline{\Sigma_{\pm L}}\}.\]
It is well-known (see \cite{McLe00}) that $(\mH^j(\Sigma_{\pm L}))^{\ast}=\tilde{\mH}^{-j}(\Sigma_{\pm L})$.
We define the two operators 
\[T_{\pm} : \tilde{\mH}^{3/2}(\Sigma_{\pm L}) \times \mH^{1/2}(\Sigma_{\pm L}) \to  \mH^{-3/2}(\Sigma_{\pm L}) \times \tilde{\mH}^{-1/2}(\Sigma_{\pm L})\]
by
\[T_\pm \bigg(\begin{array}{c}
g_\pm\\
h_\pm
\end{array}\bigg)=\bigg(\begin{array}{c}
Nu_\pm|_{\Sigma_{\pm L}}\\
Mu_\pm|_{\Sigma_{\pm L}}
\end{array}\bigg),\]
where $u_+$ (resp. $u_-$) is rightgoing as defined in \eqref{eq:rightgoing} (resp. leftgoing as defined in \eqref{eq:leftgoing}), satisfies (\ref{modes}) and is such that $(u_+,\partial_{x} u_+)|_{\Sigma_L}=(g_+,h_+)$ on $\Sigma_L$ (resp. $(u_-,-\partial_{x} u_-)|_{\Sigma_{-L}}=(g_-,h_-)$ on $\Sigma_{-L}$). Let us give an explicit definition of $u_{\pm}$ leading to an explicit expression for $T_\pm$. We detail the computation for $T_+$. Since $\partial_n=\partial_x$ and $\partial_s=\partial_y$ on $\Sigma_L$, we have 
\[Mu_+=\ders{u_+}{x} +\nu \ders{u_+}{y},\quad\qquad Nu_+=-\frac{\partial^3 u_+}{\partial{x}^3} - (2-\nu) \frac{\partial^3 u_+}{\partial x\,\partial{y}^2}\qquad\mbox{ on }\Sigma_L.\] 
Decomposition (\ref{outgoing}) and the fact that $u_+$ is rightgoing imply the following expansion for $u_+$ 
\be 
u_+(x,y)=\sum_{p=1}^{+\infty}\left(a_p e^{i\eta_p (x-L)} + b_p e^{-\gamma_p (x-L)}\right)\theta_p(y).
\label{outgoing_plus}
\ee
Hence we have
\[
\left. u_+\right|_{\Sigma_L}=\sum_{p=1}^{+\infty}(a_p + b_p )\theta_p(y) \qquad {\rm and} \qquad
\left. \deri{u_+}{x} \right|_{\Sigma_L}=\sum_{p=1}^{+\infty}(i\eta_p a_p  - \gamma_p b_p)\theta_p(y).
\]
By using the decompositions
\[g_+= \left. u_+\right|_{\Sigma_L} =\sum_{p=1}^{+\infty}g^+_p \theta_p, \quad\qquad
h_+= \left. \deri{u_+}{x} \right|_{\Sigma_L}=\sum_{p=1}^{+\infty}h^+_p \theta_p,\]
we obtain $g^+_p = a_p+b_p$ and $h^+_p = i\eta_p a_p  - \gamma_p b_p$ for all $p \in \mathbb{N}^\ast$. Inverting this system gives
\be  
\bigg(\begin{array}{c}
a_p\\
b_p
\end{array}\bigg)=\frac{1}{\gamma_p+ i\eta_p}
\bigg(\begin{array}{cc}
\gamma_p& 1 \\
i\eta_p & -1 
\end{array}\bigg) \bigg(\begin{array}{c}
g^+_p\\
h^+_p
\end{array}\bigg).\label{a+}\ee
From the above expressions of $Mu$ and $Nu$, is follows that
\[
\left\{
\begin{array}{lcl}
\dsp \left.Mu_+\right|_{\Sigma_L} &=& \dsp\sum_{p=1}^{+\infty}(-\eta_p^2 a_p + \gamma_p^2 b_p)\theta_p(y)  
-\nu \sum_{p=1}^{+\infty}(\mu_p a_p + \mu_p b_p)\theta_p(y), \\
\dsp\left.Nu_+\right|_{\Sigma_L}&=& \dsp-\sum_{p=1}^{+\infty}(-i\eta_p^3 a_p - \gamma_p^3 b_p)\theta_p(y)  
+(2-\nu) \sum_{p=1}^{+\infty}(i\mu_p\eta_p a_p - \mu_p \gamma_p b_p)\theta_p(y). 
\end{array}\right.
\]
We hence have
\[\bigg(\begin{array}{c}
\left. Nu_+\right|_{\Sigma_L}\\
\left. Mu_+\right|_{\Sigma_L}
\end{array}\bigg)= \sum_{p=1}^{+\infty}\bigg(\begin{array}{cc}
i\eta_p^3 +i (2-\nu) \mu_p\eta_p & \gamma_p^3-(2-\nu)\mu_p\gamma_p \\
-(\eta_p^2+\nu \mu_p) & \gamma_p^2-\nu \mu_p 
\end{array}\bigg) \bigg(\begin{array}{c}
a_p\\
b_p
\end{array}\bigg)\theta_p.
\]
We are now in position to obtain the expression of $T_+$: 
\[T_+\bigg(\begin{array}{c}
g_+\\
h_+
\end{array}\bigg)=\sum_{p=1}^{+\infty} T_p\bigg(\begin{array}{c}
g^+_p\\
h^+_p
\end{array}\bigg)\theta_p,\]
where the 2 by 2 matrices $T_p$ are given by
\[T_p=\frac{1}{\gamma_p + i\eta_p} \bigg(\begin{array}{cc}
i\eta_p^3 +i (2-\nu) \mu_p\eta_p & \gamma_p^3-(2-\nu)\mu_p\gamma_p \\
-(\eta_p^2+\nu \mu_p) & \gamma_p^2-\nu \mu_p 
\end{array}\bigg) \bigg(\begin{array}{cc}
\gamma_p& 1 \\
i\eta_p & -1 
\end{array}\bigg).\]
Using (\ref{etan_gamman}), we deduce that 
\be T_p=\bigg(\begin{array}{cc}
i \gamma_p \eta_p (\gamma_p-i\eta_p) &  i\gamma_p \eta_p -\nu \mu_p\\
 i\gamma_p \eta_p -\nu \mu_p & -(\gamma_p-i\eta_p)
\end{array}\bigg).\label{tprop}\ee
Concerning $T_-$, we prove similarly that for
\[g_-= \left. u_-\right|_{\Sigma_{-L}} =\sum_
{p=1}^{+\infty}g^-_p \theta_p, \quad
h_-= \left. -\deri{u_-}{x} \right|_{\Sigma_{-L}}=\sum_{p=1}^{+\infty} h^-_p \theta_p,\]
we have
\[T_-\bigg(\begin{array}{c}
g_-\\
h_-
\end{array}\bigg)=\sum_{p=1}^{+\infty} T_p\bigg(\begin{array}{c}
g^-_p\\
h^-_p
\end{array}\bigg)\theta_p,\]
where the matrix $T_p$ is defined by (\ref{tprop}). This concludes the construction of the Dirichlet-to-Neumann operators $T_{\pm}$ which enclose the outgoing behaviour for Problem (\ref{u_model}) with $C=M$ as $x\to\pm\infty$. In the following, we explain how to use these operators to reduce the analysis of (\ref{u_model}) to a bounded domain and establish Fredholmness. 

\subsection{Source term problem in the reference strip with radiation conditions}

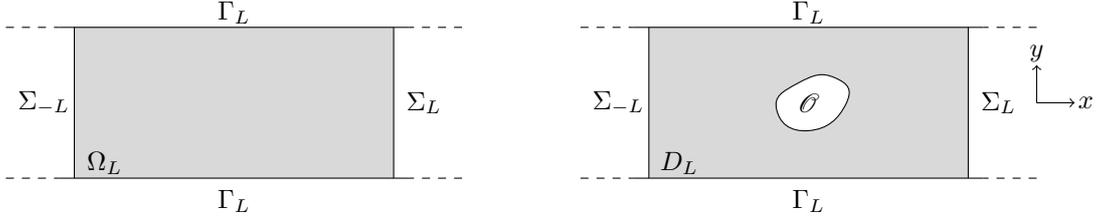
\begin{figure}[!ht]
\centering
\begin{tikzpicture}
\draw[fill=gray!30,draw=none](-2.1,0) rectangle (2.1,2);
\draw (-2.3,0)--(2.3,0);
\draw (-2.3,2)--(2.3,2);
\draw (-2.1,0)--(-2.1,2);
\draw (2.1,0)--(2.1,2);
\draw [dashed](-3,0)--(-2.3,0);
\draw [dashed](3,0)--(2.3,0);
\draw [dashed](-3,2)--(-2.3,2);
\draw [dashed](3,2)--(2.3,2);
\node at (-1.7,0.2){\small$\Om_L$};
\node at (-2.5,1){\small$\Sigma_{-L}$};
\node at (2.5,1){\small$\Sigma_L$};
\node at (0,-0.3){\small$\Gamma_L$};
\node at (0,2.2){\small$\Gamma_L$};
\end{tikzpicture}\qquad\qquad\begin{tikzpicture}
\draw[fill=gray!30,draw=none](-2.1,0) rectangle (2.1,2);
\draw (-2.3,0)--(2.3,0);
\draw (-2.3,2)--(2.3,2);
\draw (-2.1,0)--(-2.1,2);
\draw (2.1,0)--(2.1,2);
\draw [dashed](-3,0)--(-2.3,0);
\draw [dashed](3,0)--(2.3,0);
\draw [dashed](-3,2)--(-2.3,2);
\draw [dashed](3,2)--(2.3,2);
\begin{scope}[scale=0.7,yshift=0.4cm]
\draw [fill=white] plot [smooth cycle, tension=1] coordinates {(-0.6,0.9) (0,0.5) (0.7,1) (0.5,1.5) (-0.2,1.4)};
\end{scope}
\node at (0,1){\small$\mathscr{O}$};
\node at (-1.7,0.2){\small$D_L$};
\node at (-2.5,1){\small$\Sigma_{-L}$};
\node at (2.5,1){\small$\Sigma_L$};
\node at (0,-0.3){\small$\Gamma_L$};
\node at (0,2.2){\small$\Gamma_L$};
\begin{scope}[shift={(3,1)}]
\draw[->] (0,0)--(0.5,0);
\draw[->] (0,0)--(0,0.5);
\node at (0.65,0){\small$x$};
\node at (0,0.65){\small$y$};
\end{scope}
\end{tikzpicture}\vspace{-0.2cm}
\caption{Domains $\Om_L$ (left) and $D_L$ (right). \label{PictureBoundedDomains}} 
\end{figure}
\noindent Before addressing Problem (\ref{u_model}) for $C=M$ with a hole, let us consider the simpler problem in the reference strip $\Om$ without hole. For  some compactly supported function $f \in \mL^2(\Omega)$, this problem states: find $u$ in $\mH^2_{\rm loc}(\Omega)$ such that
\be
\left\{
\begin{array}{rcll}
\Delta^2 u -k^4 u& = & f  & \mbox{ \rm{in} }\Omega \\
u =Mu & = & 0 & \mbox{ \rm{on} }\partial \Omega \\
\multicolumn{4}{c}{u \mbox{ \rm{satisfies} } \mrm{(RC).}} 
\end{array}\right.
\label{u_ref}
\ee
Here $\mH^2_{\rm loc}(\Omega)$ denotes the set of distributions $u$ in $\Omega$ such that $\varphi(x) u(x,y) \in \mH^2(\Omega)$, for all $\varphi \in \mathscr{C}^{\infty}_0(\mathbb{R})$. Again, we assume that $k\in (n\pi;(n+1)\pi)$ for some $n \in \mathbb{N}$. 
\\\\
For $k\in(0;\pi)$, the radiation conditions (RC) imply that the solution is exponentially decaying at $\pm\infty$. In this case, the analysis is a bit simpler. We can prove the following proposition. 
\begin{proposition}\label{RmkBelowThreshold}
When $k\in(0;\pi)$, for all $f\in(\mH_0^1(\Om) \cap \mH^2(\Om))^\ast$, the problem
\[
\left\{
\begin{array}{rcll}
\Delta^2 u -k^4 u& = & f  & \mbox{ \rm{in} }\Omega \\
u =Mu & = & 0 & \mbox{ \rm{on} }\partial \Omega
\end{array}\right.
\]
admits a unique solution in $\mH_0^1(\Om) \cap \mH^2(\Om)$.
\end{proposition}
\begin{proof}
	Using the integration by parts formula given in Lemma \ref{IPP} and the Lax-Milgram theorem, one finds that proving the well-posedness of the problem amounts to showing the coercivity in $\mH_0^1(\Om) \cap \mH^2(\Om)$ of the sesquilinear form given by 
\[
a(u,v)=\int_{\Omega}\nu \Delta u \Delta \overline{v}\,dxdy + \int_{\Omega}(1-\nu) \bigg(\ders{u}{x}\ders{\overline{v}}{x} + 2\derd{u}{x}{y}\derd{\overline{v}}{x}{y} + \ders{u}{y}\ders{\overline{v}}{y} \bigg)\,dxdy-\int_{\Omega} k^4u  \overline{v}\,dxdy.
\]We can easily check that
\begin{equation}\label{eq:IPP}
\forall u \in \mH_0^1(\Om) \cap \mH^2(\Om),\quad \int_{\Omega}\derd{u}{x}{y}\derd{\overline{u}}{x}{y}\,dxdy=\int_{\Omega}\ders{u}{x}\ders{\overline{u}}{y}\,dxdy=\int_{\Omega}\ders{u}{y}\ders{\overline{u}}{x}\,dxdy
\end{equation}
where we have used that $u\in\mH^1_0(\Om)$ so $\partial_xu=0$ on $\partial\Om$ and that $\nu_x=0$ on $\partial\Om$. This implies that
\begin{equation}\label{eq:IPPbis}
\dsp\int_{\Omega} \left|\ders{u}{x}\right|^2+2\left|\derd{u}{x}{y}\right|^2+\left|\ders{u}{y}\right|^2\,dxdy=\dsp\int_{\Omega} |\Delta u |^2\,dxdy.
\end{equation}
By using successively (\ref{eq:IPPbis}) and the Poincar\'e inequality $\|u\|^2_{\mL^2(\Om)} \le \pi^{-4} \|\partial_{yy}u\|^2_{\mL^2(\Om)}$ for all $u\in\mH_0^1(\Om) \cap \mH^2(\Om)$, we can write
\[
\begin{array}{lcl}
a(u,u)&=&\dsp\int_{\Omega} |\Delta u |^2\,dxdy-\int_{\Omega} k^4|u|^2\,dxdy\\[10pt]
&\ge& \dsp\int_{\Omega} |\Delta u |^2\,dxdy-(k/\pi)^4\int_{\Omega} |\partial_{yy}u|^2\,dxdy\;\ge\;(1-(k/\pi)^4)\dsp\int_{\Omega} |\Delta u |^2\,dxdy\\[10pt]
&\ge&  \alpha\dsp\int_{\Omega}( |\Delta u |^2+|u|^2)\,dxdy.
\end{array}
\]
for some $\alpha>0$ since $k\in(0,\pi)$.
The identity (\ref{eq:IPPbis}) can be rewritten $|u|^2_{\mH^2(\Om)}=\|\Delta u\|^2_{\mL^2(\Om)}$. Moreover, we have
\[
\|\nabla u\|^2_{\mL^2(\Om)}=-\dsp\int_{\Omega} \Delta u\,\overline{u} \,dxdy\le (1/2)(\|\Delta u\|_{\mL^2(\Om)}^2+\| u\|_{\mL^2(\Om)}^2).
\]
Thus, for $k\in(0;\pi)$, there exists $\tilde{\alpha}>0$ such that $
a(u,u) \ge \tilde{\alpha}\,\|u\|^2_{\mH^2(\Om)},\;\text{for all }u\in \mH_0^1(\Om) \cap \mH^2(\Om).$
	\end{proof}
\noindent Now, for general $k\notin \mathbb{N}\pi$, we use the DtN operators we have constructed in the previous paragraph to derive a problem equivalent to (\ref{u_ref}) set in a bounded domain $\Omega_L:=(-L;L)\times(0;1)$. Here $L>0$ is chosen so that we have $\supp(f)\subset(-L;L)\times[0;1]$. In what follows, we set $\Gamma_L:=\partial\Om_L\setminus(\Sigma_L\cup\Sigma_{-L})$ (see Figure \ref{PictureBoundedDomains} left). Classical operations allow one to check that Problem (\ref{u_ref}) is equivalent to find $u \in\mH^2(\Omega_L)$ such that
\be
\left\{
\begin{array}{ccll}
\Delta^2 u -k^4 u& = & f  & \mbox{ \rm{in} }\Omega_L \\
u=Mu & = & 0 & \mbox{ \rm{on} }\Gamma_L  \\[3pt]
\bigg(\begin{array}{c}
Nu\\
Mu
\end{array}\bigg) & = & T_{\pm} \bigg(\begin{array}{c}
u\\
\partial_n u
\end{array}\bigg) & \mbox{ \rm{on} } \Sigma_{\pm L}.
\end{array}\right.
\label{bounded}
\ee
Let us give an equivalent variational formulation to Problem (\ref{bounded}). Define the Hilbert space $\mV_L:=\{u \in \mH^2(\Omega_L)\,|\,u=0\mbox{ \rm{on} }\Gamma_L\}$. 
We have the following integration by parts formula:
\begin{lemma}
\label{IPP}
For all $u\in V_L \cap H^4(\Omega_L)$ and for all $v \in V_L$
\[\int_{\Omega_L}\Delta^2 u\,v\,dxdy=a(u,v)-\int_{\Gamma_L}(Mu)\deri{\overline{v}}{n}\,ds-\int_{\Sigma_{\pm L}} \left((Nu)\overline{v}+(Mu) \deri{\overline{v}}{n}\right)\,ds,\]
where
\be a(u,v)=\int_{\Omega_L}\nu \Delta u \Delta \overline{v}\,dxdy + \int_{\Omega_L}(1-\nu) \left(\ders{u}{x}\ders{\overline{v}}{x} + 2\derd{u}{x}{y}\derd{\overline{v}}{x}{y} + \ders{u}{y}\ders{\overline{v}}{y} \right)\,dxdy.\label{defa}\ee
The above integration by parts formula is still valid for $u \in V_L$ such that $\Delta^2 u \in \mL^2(\Omega_L)$ provided
we interpret the integrals with the help of suitable duality brackets. In particular it is true if $(Mu)|_{\Gamma_L} \in \tilde{\mH}^{-1/2}(\Gamma_L)$, $(Mu)|_{\Sigma_{\pm L}} \in \tilde{\mH}^{-1/2}(\Sigma_{\pm L})$ and $(Nu)|_{\Sigma_{\pm L}} \in \mH^{-3/2}(\Sigma_{\pm L})$.
In this case the integral on $\Gamma_L$ has to be understood in the sense of duality pairing between
$\tilde{\mH}^{-1/2}(\Gamma_L)$ and $\mH^{1/2}(\Gamma_L)$
while the
integrals on $\Sigma_{\pm L}$ have the sense of duality pairing on the one hand between $\mH^{-3/2}(\Sigma_{\pm L})$ and $\tilde{\mH}^{3/2}(\Sigma_{\pm L})$ and on the other hand between $\tilde{\mH}^{-1/2}(\Sigma_{\pm L}) $ and $\mH^{1/2}(\Sigma_{\pm L})$.
\end{lemma}
\begin{remark}
The integration by parts formula given in Lemma \ref{IPP} is justified in \cite{hsiao_wendland} for $\mathscr{C}^{1,1}$ domains.
Note that our domain $\Omega_L$ is not $\mathscr{C}^{1,1}$ but is polygonal, which is why we need compatibility conditions at corners (see for example \cite{grisvard}).
Here, due do the chosen spaces for the traces of $u$ on the different edges, such compatibility conditions are satisfied.
\end{remark}
\noindent By Lemma \ref{IPP}, Problem (\ref{bounded}) is equivalent to the following variational formulation: find $u \in \mV_L$ such that
for all $v \in \mV_L$,
\be a(u,v)-k^4(u,v)_{\mL^2(\Omega_L)}-t(u,v)=\ell(v),\label{weak}\ee 
where
\be t(u,v)=\int_{\Sigma_{\pm L}}T_{\pm }\bigg(\begin{array}{c}
u\\
\partial_n u
\end{array}\bigg)\cdot \bigg(\begin{array}{c}
\overline{v}\\
\partial_n \overline{v}
\end{array}\bigg)\,d\sigma\label{deft}\qquad\mbox{ and }\qquad\ell(v)=\int_{\Omega_L} f\,\overline{v}\,dxdy.\ee
Define the linear and bounded operator $A^{\mrm{out}}:\mV_L\to\mV_L^{\ast}$ such that 
\begin{equation}\label{DefOperateursAout}
\langle A^{\mrm{out}} u,\overline{v}\rangle_{\Omega_L}=a(u,v)-k^4(u,v)_{\mL^2(\Omega_L)}-t(u,v),\qquad \forall (u,v)\in \mV_L\times\mV_L.
\end{equation}
Here $\langle \cdot,\cdot\rangle_{\Omega_L}$ refers to the bilinear duality pairing between $\mV_L^{\ast}$ and $\mV_L$. 
\\\\
Let us prove that the operator $A^{\mrm{out}}$ defined in (\ref{DefOperateursAout}) is Fredholm of index 0.
We first need the following Poincar\'e type lemma.
\begin{lemma}
There exists $c_0>0$ such that for all $v \in \mV_L$,
\[|v|_{\mH^2(\Omega_L)} \geq c_0\, \|v\|_{\mH^2(\Omega_L)},\]
where $|\cdot|_{\mH^2(\Omega_L)}$ and $\|\cdot\|_{\mH^2(\Omega_L)}$ stand for the semi-norm and the norm in $\mH^2(\Omega_L)$, respectively.
\label{poincare}
\end{lemma}
\begin{proof}
By contradiction, assume that for all $n \in \mathbb{N}^\ast$, there exists some $v_n \in \mV_L$ such that
\[|v_n|_{\mH^2(\Omega_L)} \leq \frac{1}{n} \|v_n\|_{\mH^2(\Omega_L)}.\]
Setting $u_n =v_n/\|v_n\|_{\mH^2(\Omega_L)}$, we obtain that
\[|u_n|_{\mH^2(\Omega_L)} \leq \frac{1}{n}\qquad\mbox{ and }\qquad \|u_n\|_{\mH^2(\Omega_L)}=1.\]
We conclude that there exists some subsequence of $(u_n)$, still denoted $(u_n)$, such that 
\[u_n \rightharpoonup u \quad{\rm in} \quad \mH^2(\Omega_L)\qquad\mbox{ and }\qquad  u_n \rightarrow u \quad{\rm in} \quad \mH^1(\Omega_L).\] 
Hence, $(u_n)$ is a Cauchy sequence in $\mH^2(\Omega_L)$, that is $(u_n)$ converges to some $w \in \mH^2(\Omega_L)$, which coincides with $u$.
Then $u_n \rightarrow u$ in $\mH^2(\Omega_L)$, which then satisfies $|u|_{\mH^2(\Omega_L)}=0$. In other words, all the second derivatives of $u$ vanish. Therefore, we get $u(x,y)=ax+by+c$ for some constants $a,b,c$.
From the boundary condition in the space $\mV_L$, we have $u(x,0)=0$ and $u(x,1)=0$ for all $x \in(-L;L)$, hence $a=b=c=0$, 
that is $u=0$. We obtain a contradiction with $ \|u\|_{\mH^2(\Omega_L)}=1$.
\end{proof}
\noindent We also need the following lemma.
\begin{lemma}\label{ret}
There exists $c_1>0$ such that for all $u \in \mV_L$,
\begin{equation}\label{estimRealPart}
-{\rm Re}\,t(u,u) \geq - c^2_1 \|u\|^2_{\mL^2(\Sigma_{\pm L})}.\end{equation}
\end{lemma}
\begin{proof}
Let $u$ be an element of $\mV_L$. Using the obvious decompositions $t=t_+ + t_-$ and $(u,\partial_{x} u)=\sum_{p} (g_p,h_p)\theta_p$ on $\Sigma_{+L}$, we find
\[
\begin{array}{lcl}
\dsp t_+(u,u) &= &\dsp \sum_{p=1}^{+\infty} T_p\bigg(\begin{array}{c}
g_p\\
h_p
\end{array}\bigg)\cdot (\overline{g_p},\overline{h_p}) \\
\dsp &=& \dsp \sum_{p=1}^{+\infty}\left\{i\gamma_p\eta_p(\gamma_p-i\eta_p) |g_p|^2 -(\gamma_p-i\eta_p) |h_p|^2 
+2(i\gamma_p \eta_p-\nu \mu_p){\rm Re}(g_p \overline{h_p})\right\}.
\end{array}\]
Assume that $k\in(n\pi;(n+1)\pi)$ with $n\in\N^{\ast}$ (the case $k\in(0;\pi)$, simpler to study, is left to the reader). Since $\eta_p=\sqrt{k^2-\pi^2p^2}$ (see (\ref{etan_gamman})), we observe that for $p=1,\dots,n$, the number $\eta_p$ is purely real, while for $p \geq n+1$, we have $\eta_p=i\beta_p$ with $\beta_p=\sqrt{\pi^2p^2-k^2}\in\R$. Hence
\[
\begin{array}{lcl}
\dsp t_+(u,u) &=& \dsp \sum_{p=1}^{n}\left\{i\gamma_p\eta_p(\gamma_p-i\eta_p) |g_p|^2 -(\gamma_p-i\eta_p) |h_p|^2 
+2(i\gamma_p \eta_p-\nu \mu_p){\rm Re}(g_p \overline{h_p})\right\}\\
\dsp && -\dsp \sum_{p \geq n+1} \left\{\gamma_p\beta_p(\gamma_p+\beta_p) |g_p|^2 + (\gamma_p + \beta_p) |h_p|^2 + 2(\gamma_p \beta_p +\nu \mu_p){\rm Re}(g_p\overline{h_p})\right\}.
\end{array}
\]
We show that
\begin{equation}\label{decompoRe}
-{\rm Re}\, t_+(u,u)= \sum_{p=1}^{n} u_p +  \sum_{p=n+1}^{+\infty} v_p.
\end{equation}
where 
\[\left\{
\begin{array}{lcl}
\dsp u_p &:=&\dsp -\gamma_p \eta_p^2 |g_p|^2 + \gamma_p |h_p|^2 +2\nu \mu_p {\rm Re}(g_p \overline{h_p})\\[6pt]
\dsp v_p &:=& \dsp \gamma_p\beta_p(\gamma_p+\beta_p) |g_p|^2 + (\gamma_p + \beta_p) |h_p|^2 + 2(\gamma_p \beta_p +\nu \mu_p){\rm Re}(g_p\overline{h_p})
\end{array}\right..\]
Since we have $2\nu \mu_p {\rm Re}(g_p \overline{h_p}) \geq -\gamma_p |h_p|^2 - \gamma_p^{-1}\nu^2 \mu_p^2 |g_p^2|$, we deduce $u_p \geq -(\gamma_p \eta_p^2  + \gamma_p^{-1}\nu^2 \mu_p^2 )|g_p|^2$. Therefore, for $p=1,\dots,n$, we obtain
\begin{equation}\label{decompoRe1}
u_p \geq -c^2_1 |g_p|^2
\end{equation}
for some constant $c_1>0$. On the other hand, we can write
\[v_p =(\gamma_p+\beta_p)(\gamma_p\beta_p |g_p|^2 + |h_p|^2 +2c_p {\rm Re}(g_p \overline{h_p}))\quad\mbox{ with }\quad c_p:=(\gamma_p\beta_p + \nu \mu_p)/(\gamma_p + \beta_p).\]
This gives $v_p \geq (\gamma_p + \beta_p)(\gamma_p\beta_p |g_p|^2 + |h_p|^2- c_p^2|g_p|^2-|h_p|^2)=(\gamma_p+\beta_p)(\gamma_p\beta_p -c_p^2) |g_p|^2$. Using the fact that $\gamma_p^2=k^2+\mu_p$, $\beta_p^2=\mu_p-k^2$ and $\nu \in [0;1)$, we find
\[c_p\leq \frac{\gamma_p\beta_p+ \mu_p}{\gamma_p+\beta_p} = \frac{\gamma_p\beta_p+(\gamma_p^2+\beta_p^2)/2}{\gamma_p+\beta_p}=\frac{1}{2}(\gamma_p+\beta_p).\]
Hence, we get $\gamma_p\beta_p -c_p^2 \geq \gamma_p\beta_p -\frac{1}{4}(\gamma_p+\beta_p)^2=-\frac{1}{4}(\gamma_p-\beta_p)^2$, and so 
\[v_p \geq -\frac{1}{4}(\gamma_p+\beta_p)(\gamma_p-\beta_p)^2 |g_p|^2=- \frac{k^4}{\gamma_p+\beta_p} |g_p|^2 \geq -k^3|g_p|^2.\]
where the last inequality is due to $\gamma_p>k$ and $\beta_p>0$.
As a consequence, there is a constant $c_1>0$ such that for $p\geq n+1$, there holds
\begin{equation}\label{decompoRe2}
v_p \geq -c^2_1 |g_p|^2.
\end{equation}
Using (\ref{decompoRe1}) and (\ref{decompoRe2}) in (\ref{decompoRe}), we get $-{\rm Re}\, t_+(u,u)\ge -c^2_1\sum_{p=1}^{\infty}|g_p|^2=-c_1\|u\|^2_{\mL^2(\Sigma_{L})}$. Working analogously with $-{\rm Re}\, t_-(u,u)$, we obtain the desired result (\ref{estimRealPart}).
\end{proof}
\noindent Let us now state the main result of this section.
\begin{theorem}
\label{th_ref}
Assume that $k\in (n\pi;(n+1)\pi)$ with $n \in \mathbb{N}$. The operator $A^{\mrm{out}}$ defined in \eqref{DefOperateursAout} is an isomorphism. As a consequence, for any compactly supported function $f \in \mL^2(\Omega)$, Problem (\ref{u_ref}) has a unique solution in $\mH^2_{\rm loc}(\Omega)$.
\end{theorem}
\begin{proof}
Let us decompose the operator $A^{\mrm{out}}$ defined in \eqref{DefOperateursAout} as
\[
	A^{\mrm{out}}=A_0+A_c
\]
with$\qquad\left\{\begin{array}{lcl}
\dsp \langle A_0 u,\overline{v}\rangle_{\Omega_L} &=& \dsp a(u,v)-t(u,v) + c^2_1 (u,v)_{\mL^2(\Sigma_{\pm L})}\\[3pt]
\dsp \langle A_c u,\overline{v}\rangle_{\Omega_L} &=& \dsp -k^4(u,v)_{\mL^2(\Omega_L)}-c^2_1 (u,v)_{\mL^2(\Sigma_{\pm_L})},
\end{array}\right.\quad\forall u,\,v\in \mV_L.$\\[5pt]
\newline
From Lemma \ref{poincare} and Lemma \ref{ret}, we have, for all $u \in \mV_L$,
\[{\rm Re}\langle A_0 u,\overline{u}\rangle_{\Omega_L} \geq a(u,u) \geq (1-\nu)|u|^2_{\mH^2(\Omega_L)} \geq c^2_0(1-\nu)\|u\|^2_{\mH^2(\Omega_L)}.\]
Due to Lax-Milgram theorem, the operator $A_0$ is an isomorphism. Since the operator $A_c$ is compact, we conclude that the operator $A^{\mrm{out}}$ is Fredholm of index 0. In particular, injectivity implies surjectivity.
It remains to prove injectivity. By definition of the operator $A^{\mrm{out}}$ and since we have equivalence between problems (\ref{u_ref}) and (\ref{bounded}), any element $u$ of ker $A^{\mrm{out}}$  satisfies Problem (\ref{u_ref}) with $f=0$, that is in particular Problem (\ref{modes}).
The solutions to that problem are given by (\ref{general_form}). The radiation conditions (RC) eventually imply that $u=0$.
\end{proof}
\subsection{Source term problem in the perturbed strip with radiation conditions}
Let us now address the source term Problem (\ref{u_model}) (with a hole) when $C=M$. We remind the reader that this problem states, for a compactly supported function $f \in \mL^2(D)$, find $u$ in $\mH^2_{\rm loc}(D)$ such that
\be
\left\{
\begin{array}{rcll}
\Delta^2 u -k^4 u& = & f  & \mbox{ \rm{in} }D \\
u =Mu & = & 0 & \mbox{ \rm{on} }\partial \Omega \\
Mu=Nu & = & 0 & \mbox{ \rm{on} } \partial \mathscr{O} \\
\multicolumn{4}{c}{u \mbox{ \rm{satisfies} } \mrm{(RC)}.} 
\end{array}\right.
\label{u_pert}
\ee
Again, we assume that $k\in (n\pi;(n+1)\pi)$ with $n \in \mathbb{N}$ (note that when $k\in(0;\pi)$, using the result of Proposition \ref{RmkBelowThreshold}, 
one can prove Fredholmness of (\ref{u_pert}) in $\{u \in \mH^2(D)\,|\, u=0\mbox{ \rm{on} }\partial\Om\}$). We define the domain $D_L:=\{(x,y)\in D\,|\,|x|<L\}$ where $L$ is chosen large enough so that both the hole $\mathscr{O}$ and $f$ are supported in $D_L$ (see Figure \ref{PictureBoundedDomains} right). We use the DtN operators $T_{\pm}$ defined in \S\ref{paragraphDtN}. Problem (\ref{u_pert}) is equivalent to finding $u \in\mH^2(D_L)$ such that
\be
\left\{
\begin{array}{ccll}
\Delta^2 u -k^4 u& = & f  & \mbox{ \rm{in} }D_L \\
u = Mu & = & 0 & \mbox{ \rm{on} }\Gamma_L \\
Mu = Nu & = & 0 & \mbox{ \rm{on} } \partial \mathscr{O} \\[3pt]
\bigg(\begin{array}{c}
Nu\\
Mu
\end{array}\bigg) & = & T_{\pm } \bigg(\begin{array}{c}
u\\
\partial_n u
\end{array}\bigg) & \mbox{ \rm{on} } \Sigma_{\pm L}.
\end{array}\right.
\label{bounded_pert}
\ee
We now introduce a variational formulation of (\ref{bounded_pert}) exactly as we did in the reference strip. First we define the Hilbert space $\mW_L:=\{u \in \mH^2(D_L)\,|\, u=0\mbox{ \rm{on} }\Gamma_L\}$. 
Problem (\ref{bounded_pert}) is equivalent to the variational formulation: find $u \in \mW_L$ such that
for all $v \in \mW_L$,
\be b(u,v)-k^4(u,v)_{\mL^2(D_L)}-t(u,v)=m(v).\label{weak_pert}\ee 
Here $t$ is defined in (\ref{deft}) while the sesquilinear (resp. antilinear) form $b$ (resp. $m$) is the analogous of $a$ (resp. $\ell$) defined in (\ref{defa}) (resp. (\ref{deft})) with $\Omega_L$ replaced by $D_L$. Define the linear and bounded operator $B^{\mrm{out}}:\mW_L\to\mW_L^{\ast}$ such that 
\begin{equation}\label{DefOperateursSimply}
\langle B^{\mrm{out}} u,\overline{v}\rangle_{D_L}=b(u,v)-k^4(u,v)_{\mL^2(D_L)}-t(u,v),\qquad \forall (u,v)\in\mW_L\times\mW_L.
\end{equation}
Here $\langle \cdot,\cdot\rangle_{D_L}$ refers to the bilinear duality pairing between $\mW_L^{\ast}$ and $\mW_L$. Working as in the proof of Theorem \ref{th_ref} (in Lemma \ref{poincare}, replace the space $\mV_L$ by $\mW_L$) and using the Fredholm theory, we obtain the main result of this section. 
\begin{theorem}\label{thmIsomObstacleSimply}
Assume that $k\in (n\pi;(n+1)\pi)$ with $n \in \mathbb{N}$. The operator $B^{\mrm{out}}$ defined in (\ref{DefOperateursSimply}) is Fredholm of index zero. As a consequence,\\[3pt]
$a)$ If $\ker\,B^{\mrm{out}}=\{0\}$, then $B^{\mrm{out}}$ is an isomorphism.\\[3pt]
$b)$ If $\ker\,B^{\mrm{out}}=\mrm{span}(z_1,\dots,z_d)$ for some $d\ge1$, then the equation $B^{\mrm{out}}u=F\in\mW_L^{\ast}$ admits a solution (defined up to an element of $\ker\,B^{\mrm{out}}$) if and only if $F$ satisfies the compatibility conditions $\langle F,\overline{z_j}\rangle_{D_L}=0$ for $j=1,\dots,d$.
\end{theorem} 
\begin{remark}\label{RmqExistenceUniqueness}
From Theorem \ref{thmIsomObstacleSimply}, we deduce that if Problem (\ref{u_pert}) for $f=0$ has only the zero solution in $\mH_{\rm loc}^2(D)$, then Problem (\ref{u_pert}) has a unique solution in $\mH_{\rm loc}^2(D)$ for any $f \in \mL^2(D)$ which is compactly supported. 
\end{remark}
\begin{remark}\label{RmqTrappedModes}
Assume that $u \in \mH_{\rm loc}^2(D)$ satisfies Problem (\ref{u_pert}) with $f=0$. Then $u$ is a trapped mode, in the sense that $u \in \mH^2(D)$.
Indeed, if $m=0$, setting $v=u$ in (\ref{weak_pert}), we obtain
\[{\rm Im}\, t(u,u)=0.\]
Using the decomposition $(u,\partial_{x} u)=\sum_{p=1}^{+\infty} (g^{\pm}_p,h^{\pm}_p)\theta_p$ on $\Sigma_{\pm L}$, we find
\[{\rm Im}\, t(u,u)= \sum_{\mu=\pm}\sum_{p=1}^{n}\eta_p(\gamma_p^2 |g^\mu_p|^2+ |h^\mu_p|^2 + 2\gamma_p {\rm Re}(g^\mu_p\overline{h^\mu_p}))=\sum_{\mu=\pm}\sum_{p=1}^{n} \eta_p |\gamma_p g^\mu_p + h^\mu_p|^2.\]
We deduce that $\gamma_p g^{\pm}_p + h^{\pm}_p=0$ for $p=1,\dots,n$. Then working as in (\ref{a+}), we find that the coefficients $a^{\pm}_p$ in (\ref{outgoing}) satisfy $a^{\pm}_p=0$ for $p=1,\dots,n$ (the projection on the propagating modes is null). We infer that $u$ is exponentially decaying for $|x|>L$. As a consequence, $u$ belongs to $\mH^2(D)$.  
\end{remark}
\subsection{Scattering problem in the perturbed strip with radiation conditions}
Finally, we use the results of the previous paragraph to study the following scattering problem: find the total field $u$ such that
\begin{equation}\label{PbScattering}
\left\{
\begin{array}{rcll}
\Delta^2 u -k^4 u& = & 0  & \mbox{ \rm{in} }D \\
u =M u & = & 0 & \mbox{ \rm{on} }\partial \Omega \\
Mu=Nu& = & 0 & \mbox{ \rm{on} } \partial \mathscr{O} \\
\multicolumn{4}{c}{u-u_i \mbox{ \rm{satisfies} } \mrm{(RC)}}
\end{array}\right.
\end{equation}
where $u_i$ is an incident field which solves
\[
\left\{
\begin{array}{rcll}
\Delta^2 u_i -k^4 u_i& = & 0  & \mbox{ \rm{in} } \Omega\\
u_i =Cu_i& = & 0 & \mbox{ \rm{on} }\partial \Omega. 
\end{array}\right.
\]
In the following, we take $k\in (n\pi;(n+1)\pi)$ with $n \in \mathbb{N}^{\ast}$ and $u_i\in\{w_p^{\pm}\,|\,p=1,\dots,n\}$, where $w_p^{\pm}$ is the propagating mode such that
\begin{equation}\label{NormalisationModesSimplySupp}
w_p^{\pm}(x,y)=(2\eta_p)^{-1/2}e^{\pm i\eta_p x}\theta_p(y)=\eta_p^{-1/2}e^{\pm i\sqrt{k^2-\pi^2p^2} x}\sin(\pi p y).
\end{equation}
The normalization in (\ref{NormalisationModesSimplySupp}) is chosen so that the scattering matrix below is unitary. 
\begin{theorem}\label{ThmScaSimply1}
Assume that $k\in(n\pi;(n+1)\pi)$ with $n \in \mathbb{N}^{\ast}$. Then for $u_i=w_p^{\pm}$, $p=1,\dots,n$, Problem (\ref{PbScattering}) admits a solution $u_p^{\pm}$. This solution is uniquely defined if and only if trapped modes are absent at the wavenumber $k$. 
\end{theorem}
\begin{proof}
Let $\zeta$ be a smooth cut-off function which depends only on $x$, which vanishes in a neighborhood of the hole $\mathscr{O}$, and which is equal to one for $|x|\ge L-\eps$ for some small given $\eps>0$. Theorem \ref{thmIsomObstacleSimply} guarantees that there is a function $v\in\mH_{\rm loc}^2(D)$ which solves the problem
\begin{equation}\label{eq:scat}
\left\{
\begin{array}{rcll}
\Delta^2 v -k^4 v& = & f  & \mbox{ \rm{in} }D \\
v =Mv & = & 0 & \mbox{ \rm{on} }\partial \Omega \\
Mv=Nv & = & 0 & \mbox{ \rm{on} } \partial \mathscr{O} \\
\multicolumn{4}{c}{v \mbox{ \rm{satisfies} } \mrm{(RC)}.} 
\end{array}\right.
\end{equation}
with $f:=-(\Delta^2(\zeta u_i) -k^4 (\zeta u_i))$. Indeed, first we observe that $f$ belongs to $\mL^2(D)$ and is compactly supported. Now, if trapped modes are absent at the given wavenumber $k$, $B^{\mrm{out}}$ is an isomorphism and the existence of $v$ is clear. If $\ker\,B^{\mrm{out}}=\mrm{span}(z_1,\dots,z_d)$ for some $d\ge1$, one observes that for $j=1,\dots,d$, we have
\[
\begin{array}{ll}
\dsp\int_{D_L}f\,\overline{z_j}\,dxdy&\dsp =-\int_{D_L}(\Delta^2(\zeta u_i) -k^4 (\zeta u_i))\,\overline{z_j}\,dxdy\\[7pt]
&\dsp=-\int_{\Sigma_{\pm L}} \left((Nu_i)\overline{z_j}+(Mu_i) \deri{\overline{z_j}}{n}\right)\,ds+\int_{\Sigma_{\pm L}} \left(u_iN\overline{z_j}+ \deri{u_i}{n}M\overline{z_j}\right)\,ds\\[7pt]
&\dsp\qquad-\int_{D_L}\zeta u_i(\Delta^2 \overline{z_j} -k^4 \overline{z_j})\,dxdy=0.
\end{array}
\]
To obtain the second identity, we used twice the integration by parts formula of Lemma \ref{IPP} (observe that $\zeta$ vanishes in a neighbourhood of $\mathscr{O}$, $\zeta$ depends only on $x$ so that the integrals on $\Gamma_L$ vanish and finally $\zeta$ is equal to $1$ in the neighborhood of $\Sigma_{\pm L}$). To obtain the third equality, we used the formulas \eqref{MN}, the orthonormality of the family $(\theta_p)$ in $\mL^2(I)$ and the fact that the trapped modes $z_j$ satisfy \eqref{eq:scat} with $f=0$ and do not decompose on the propagating modes (see Remark \ref{RmqTrappedModes}). Once we have the guarantee that $v$ is well-defined, we can set $u:=v+\zeta u_i$. One can verify that $u$ is a solution to problem (\ref{PbScattering}).
\end{proof}
\noindent For $p=1,\dots,n$, denote $\Psi_p$ the solution of \eqref{PbScattering} for $u_i=w_p^-$ and $\Psi_{n+p}$ the solution for $u_i=w_p^+$. Introduce  $\chi^{\pm}\in\mathscr{C}^{\infty}(\R^2)$ a cut-off function equal to one for $\pm x\ge 2L$ and to zero for $\pm x\le L$, for a given $L>0$. Decompose the $\Psi_p$ as
\[
\begin{array}{lcl}
\Psi_p&\hspace{-0.2cm}=&\hspace{-0.2cm}\chi^+\,w^-_p+\chi^+\dsp\sum_{m=1}^{n} s_{p\,m}\,w^{+}_{m} + \chi^-\sum_{m=1}^{n} s_{p\,n+m}\,w^{-}_{m} + \tilde{\Psi}_p,\\[15pt] 
\Psi_{n+p}&\hspace{-0.2cm}=&\hspace{-0.2cm}\chi^-\,w^+_p+\chi^+\dsp\sum_{m=1}^{n} s_{n+p\,m}\,w^{+}_{m} + \chi^-\sum_{m=1}^{n} s_{n+p\,n+m}\,w^{-}_{m} + \tilde{\Psi}_{n+p},
\end{array}
\]
where the $\tilde{\Psi}_p$, $p=1,\dots,2n$, are functions which are exponentially decaying at infinity and where the $s_{p\,m}$, $1\le p,m\le 2n$, are complex numbers. Define the scattering matrix 
\begin{equation}\label{DefScatteringMatrixSimply}
\mathbb{S}:=(s_{p\,m})_{1\le p,m\le 2n}\in\Cplx^{2n\times2n}.
\end{equation}
\begin{theorem}\label{ThmScaSimply2}
For all $k\in(n\pi;(n+1)\pi)$, $n \in \mathbb{N}^{\ast}$, the scattering matrix (\ref{DefScatteringMatrixSimply}) is uniquely defined (even in presence of trapped modes), unitary ($\mathbb{S}\,\overline{\mathbb{S}}^{\top}=\mrm{Id}^{2n\times2n}$) and symmetric ($\mathbb{S}^{\top}=\mathbb{S}$). 
\end{theorem}
\begin{proof}
If trapped modes are absent at the wavenumber $k$, the $\Psi_p$'s are uniquely defined and the scattering matrix as well. In the presence of trapped modes, assume that Problem (\ref{PbScattering}) admits two solutions $u^1$ and $u^2$ for a given $u_i\in\{w_p^{\pm}\,|\,p=1,\dots,n\}$. Then $u^1-u^2$ is a trapped mode which, according to Remark \ref{RmqTrappedModes}, do not decompose on the propagating modes. This is enough to show that $\mathbb{S}$ is uniquely defined. The unitarity and the symmetry of $\mathbb{S}$ will be established in the clamped case (see Theorem \ref{TheoremDecompoKernel}). The proof is exactly the same here.
\end{proof}

\section{Well-posedness in the clamped case}\label{SectionClamped}

In this section, we suppose that $k$ is not a threshold wavenumber, i.e. $k\neq k_n$ for $n\in\N$, where $k_n$ is defined in \eqref{defThreshold}.
\subsection{Preliminaries}\label{sub:Preliminaries_clamped}
In this section, our goal is to study Problem (\ref{u_model}) with clamped boundary conditions. To proceed, first we shall work on the problem set in the reference strip, without hole:
\begin{equation}\label{PbInitial}
\left\{
\begin{array}{rcll}
\Delta^2 u-k^4u & = & f & \mbox{ in }\Om\\[3pt]
u= \partial_nu  & = & 0  & \mbox{ on }\partial\Om,
\end{array}
\right.
\end{equation}
where $f$ is a given source term in a space to determine. In order to define radiation conditions at infinity, we will have to establish a modal decomposition for the solutions of (\ref{PbInitial}) with $f=0$ similar to (\ref{general_form}) for the simply supported case. Expansion (\ref{general_form}) was derived thanks to a result of Hilbert basis (see Remark \ref{RmkBasis}). For the clamped problem, we do not know if the family of eigenfunctions of the symbol $\mathscr{L}$ defined in (\ref{defOperator}) forms a Hilbert basis of $\mL^2(I)$. As a consequence, as mentioned in the introduction, we will establish the modal decomposition with a different strategy based on the joint use of the Fourier-Laplace transform  in the unbounded direction, of weighted Sobolev spaces defined in Section \ref{sub:WeightedSobolev} and of the residue theorem. 
The foundation of the theory is due to Kondratiev \cite{Kond67}. For a modern presentation of the technique, one may consult the monographs \cite{NaPl94,KoMR97,KoMR01}. To help the reader, below we try to give enough details to get a self-consistent presentation of the approach. Again, we emphasize that the method we develop in this section to study the clamped problem can also be used to consider the simply supported problem. 
\subsubsection{The weighted Sobolev spaces}\label{sub:WeightedSobolev}

For $\beta\in\R$, define the space $\mathring{\mW}^2_{\beta}(\Om)$ as the completion of $\mathscr{C}^{\infty}_0(\Om)$ for the norm
\begin{equation}\label{FirstWeightedSpaces}
\|v\|_{\mW^2_{\beta}(\Om)}=\Big
(\sum_{\alpha,\,\gamma\in\N,\ \alpha+\gamma\le 2}\|e^{\beta x}\partial^{\alpha}_x\partial^{\gamma}_y v\|^2_{\mrm{L}^2(\Om)}\Big)^{1/2}.
\end{equation}
Observe that for $\beta=0$, we have $\mW^2_{0}(\Om)=\mH^2_0(\Om)$ where $\mH^2_0(\Om)$ stands for the usual Sobolev space. We denote $\mathring{\mW}^2_{\beta}(\Om)^{\ast}$ the topological dual space of $\mathring{\mW}^2_{\beta}(\Om)$ endowed with the norm
\[
\|f\|_{\mathring{\mW}^2_{\beta}(\Om)^{\ast}}=\sup_{v\in\mathring{\mW}^2_{\beta}(\Om)\setminus\{0\}}\cfrac{|\langle f,\overline{v}\rangle_{\Om}|}{\|v\|_{\mW^2_{\beta}(\Om)}}.
\]
Here $\langle \cdot,\cdot\rangle_{\Om}$ refers to the bilinear duality pairing between $\mathring{\mW}^2_{\beta}(\Om)^{\ast}$ and $\mathring{\mW}^2_{\beta}(\Om)$. For $\beta\in\R$, define the linear and bounded operator $A_{\beta}:\mathring{\mW}^2_{\beta}(\Om)\to \mathring{\mW}^{2}_{-\beta}(\Om)^{\ast}$ such that 
\begin{equation}\label{DefOperateurs}
\langle A_{\beta} u,\overline{v}\rangle_{\Om}=\int_{\Om} \Delta u\,\overline{\Delta v}-k^4 u\,\overline{v}\,dxdy,\qquad \forall (u,v)\in\mathring{\mW}^2_{\beta}(\Om)\times \mathring{\mW}^2_{-\beta}(\Om).
\end{equation}
Define the partial Fourier-Laplace transform $\mathcal{L}_{x\to\lambda}$ with respect to the variable $x$ such that, for $\lambda\in\Cplx$, 
\[
\hat{v}(\lambda,\cdot)=(\mathcal{L}_{x\to\lambda}v)(\lambda,\cdot):=\int_{-\infty}^{+\infty} e^{-\lambda x}v(x,\cdot)\,dx.
\]
 It is an isomorphism between 
 \[
 \mW^2_{\beta}(\Om)\quad\text{and}\quad\widehat{\mW}^2_{\beta}:=\Big\{\hat{v}\in \mL^2(\ell_{-\beta},\mH^2_0(I)),\;\int_{\ell_{-\beta}}\|\hat{v}(\lambda,\cdot)\|_{\mH^2(I,|\lambda|)}^2\,d\lambda\;<+\infty\Big\}
 \]
  where $\ell_{-\beta} = \{\lambda=-\beta+i s,\;s\in\R\}$, for all $\beta\in\R$ and where
\begin{equation}\label{normParam1}
	\|\varphi\|_{\mH^2(I,\,|\lambda|)}:=\Big
(\sum_{\alpha,\,\gamma\in\N,\ \alpha+\gamma\le 2}\||\lambda|^{\alpha}\partial^{\gamma}_y \varphi\|^2_{\mrm{L}^2(I)}\Big)^{1/2},\qquad\forall \varphi\in\mH_0^2(I).
\end{equation}
Note that for a fixed $\lambda$, the norms $\|\cdot\|_{\mH^2(I,\,|\lambda|)}$ and $\|\cdot\|_{\mH^2(I)}$ are equivalent on $\mH^2_0(I)$. However the constants of equivalence depend on $|\lambda|$.
\\\\
We have also the Plancherel formula
 \begin{equation}\label{eq:Plancherel_1/2}
 \|v\|^2_{\mW^2_{\beta}(\Om)}	 = \frac{1}{2\pi i}\,\int_{\ell_{-\beta}} \|\hat{v}(\lambda,\cdot)\|_{\mH^2(I,|\lambda|)}^2\,d\lambda:=\|\hat{v}\|_{\widehat{\mW}^2_{\beta}}^2,
 \end{equation}
 and the inverse $\mathcal{L}_{x\to\lambda}^{-1}$ is given by
 \[
 	\forall \hat{v}\in \widehat{\mW}^2_{\beta},\quad \big(\mathcal{L}_{x\to\lambda}^{-1}\hat{v}\big)(x,\cdot)=\frac{1}{2\pi i}\int_{\ell_{-\beta}}e^{\lambda x}\,\hat{v}(\lambda,\cdot)\,d\lambda.
 \]
Let us denote $\widehat{\mW}^{2\ast}_{\beta}$ the topological dual space of $\widehat{\mW}^2_{\beta}$ which can be characterized as
\[
	\widehat{\mW}^{2\ast}_{\beta}=\Big\{\hat{g}\in \mL^2(\ell_{\beta},\mH^{-2}(I)),\;\int_{\ell_{\beta}}\|\hat{g}(\lambda,\cdot)\|_{\mH^{-2}(I,|\lambda|)}^2\,d\lambda\;<+\infty\Big\}
\]
where 
\begin{equation}\label{normParam2}
\|g\|_{\mH^{-2}(I,\,|\lambda|)}:=\sup_{\varphi\in\mH^2_0(\Om)\setminus\{0\}}\cfrac{|\langle g,\overline{\varphi}\rangle_{I}|}{\|\varphi\|_{\mH^{2}(I,\,|\lambda|)}},\qquad\forall g\in\mH^{-2}(I),
\end{equation}
$\langle \cdot,\cdot\rangle_{I}$ being the duality product between $\mH^{-2}(I)$ and $\mH^2_0(I)$.
The partial Laplace Fourier Transform $\mathcal{L}_{x\to\lambda}$ can be defined by duality for functions in ${\mathring{\mW}^2_{\beta}(\Om)^{\ast}}$ as
 \[
 	\forall f\in\mathring{\mW}^{2}_{\beta}(\Om)^{\ast},\;\hat{v}\in \widehat{\mW}^2_{\beta},\quad \langle \mathcal{L}_{x\to\lambda} f, \hat{v}\rangle_{\hat\Om}:= \langle  f,\mathcal{L}_{x\to\lambda}^{-1}\hat{v}\rangle_{\Om}
 \]
where $\langle \cdot,\cdot\rangle_{\Om}$ refers to the bilinear duality pairing between $\mathring{\mW}^2_{\beta}(\Om)^{\ast}$ and $\mathring{\mW}^2_{\beta}(\Om)$ and $\langle \cdot,\cdot\rangle_{\hat\Om}$ refers to the one between $\widehat{\mW}^{2\ast}_{\beta}$ and $\widehat{\mW}^2_{\beta}$. Finally, we have also a Plancherel formula
 \begin{equation}\label{eq:Plancherel_1/2ast}
 \|f\|^2_{\mathring{\mW}^{2}_{\beta}(\Om)^{\ast}}	 = \frac{1}{2\pi i}\,\int_{\ell_{\beta}} \|\hat{f}(\lambda,\cdot)\|_{\mH^{-2}(I,|\lambda|)}^2\,d\lambda.
 \end{equation}
We can now apply $\mathcal{L}_{x\to\lambda}$ to the equation $A_{\beta}u=f$, one is led to study the symbol $\mathscr{L}(\lambda):\mH^2(I,|\lambda|)\to\mH^{-2}(I,|\lambda|)$ for $\lambda\in-\beta+i\mathbb{R}$, defined in (\ref{defOperator}) and such that 
\begin{equation}\label{defSymbol}
\langle \mathscr{L}(\lambda) \varphi,\overline{\psi}\rangle_{I} =\int_{I}(\lambda^2\varphi+d_{yy} \varphi)(\lambda^2\overline{\psi}+d_{yy} \overline{\psi})-k^4\varphi\overline{\psi}\,dy,\qquad\forall \varphi,\,\psi\in\mH^2_0(I).
\end{equation}
In the following, we shall denote $(\cdot,\cdot)_I$ the usual inner product of $\mrm{L}^2(I)$. Studying the properties of the symbol $\mathscr{L}(\cdot)$ defined in (\ref{defSymbol}) leads to consider $1\mrm{D}$ problems set on $I$ depending on a complex parameter $\lambda$. 

\subsubsection{Properties of the symbol}\label{paragraphSymbol}

In this paragraph, we study the properties of the symbol $\mathscr{L}(\cdot)$ defined in (\ref{defSymbol}).
For a fixed $\lambda$, as the norms $\|\cdot\|_{\mH^2(I,\,|\lambda|)}$ and $\|\cdot\|_{\mH^2(I)}$ are equivalent on $\mH^2_0(I)$, it suffices to study $\mathscr{L}(\cdot)$ as an operator from $\mH^2_0(I)$ in $\mH^{-2}(I)$ and establish estimates in the $\lambda$ dependent norms of $\mH^2(I,|\lambda|)$ and $\mH^{-2}(I,|\lambda|)$.
\begin{lemma}\label{resultPrelim}
There is $\tau_0>0$ such that for $\lambda=i\tau$ with $|\tau|\ge\tau_0$, $\mathscr{L}(\lambda):\mH^2_0(I)\to\mH^{-2}(I)$ is an isomorphism.  
\end{lemma}
\begin{proof}
For $\lambda=i\tau$, we have 
\begin{equation}\label{SymbolComplexe}
\langle \mathscr{L}(i\tau) \varphi,\overline{\psi}\rangle_{I} =\int_{I}d_{yy} \varphi\, d_{yy}\overline{\psi}+2\tau^2d_{y}\varphi\,d_{y}\overline{\psi}+\tau^4\varphi\overline{\psi}-k^4\varphi\overline{\psi}\,dy.
\end{equation}
Therefore, the result is a consequence of the Lax-Milgram theorem (take $\tau_0=k$).
\end{proof}
\noindent We remind the reader that we say that $\lambda\in\Cplx$ is an eigenvalue of $\mathscr{L}$ if there is a non-zero $\varphi\in\mH^2_0(I)$ such that $\mathscr{L}(\lambda)\varphi=0$. We denote $\Lambda$ the set of eigenvalues of $\mathscr{L}$. From Lemma \ref{resultPrelim}, according to the analytic Fredholm theorem, we deduce the following result. 
\begin{corollary}\label{coroInvertibility}
For all $\lambda\in\Cplx$, $\mathscr{L}(\lambda):\mH^2_0(I)\to\mH^{-2}(I)$ is an isomorphism if and only if $\lambda$ is not an eigenvalue of $\mathscr{L}$. The set of eigenvalues of $\mathscr{L}$ is discrete  and does not have any accumulation point in $\Cplx$. 
\end{corollary}
\noindent In order to apply the inverse Fourier-Laplace transform, we need estimates for $\mathscr{L}(\lambda)^{-1}$ on lines $\{\lambda\in\Cplx\,|\,\Re e\,\lambda=\beta\}$, $\beta\in\R$, in the parameter dependent norms (\ref{normParam1}), (\ref{normParam2}). 
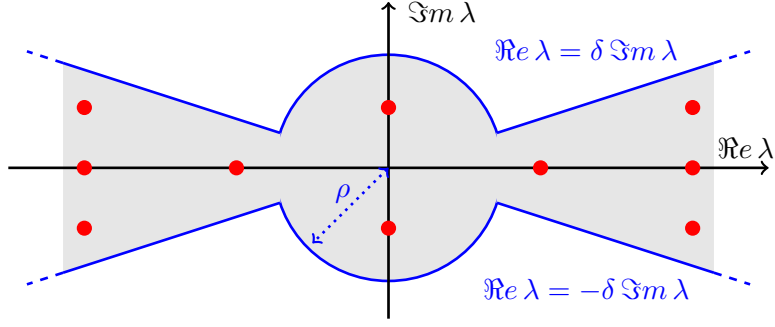
\begin{figure}[!ht]
\centering
\begin{tikzpicture}
\draw [draw=blue,fill=gray!20,line width=1pt] (0,0) circle (1.5);
\draw[draw=none,fill=gray!20](1.42,0.4635)--(4.28,1.39)--(4.28,-1.39)--(1.42,-0.4635)--cycle;
\draw[draw=none,fill=gray!20](-1.42,0.4635)--(-4.28,1.39)--(-4.28,-1.39)--(-1.42,-0.4635)--cycle;
\draw[draw=black,line width=1pt,->](-5,0)--(5,0);
\draw[draw=black,line width=1pt,->](0,-2)--(0,2.2);
\draw[draw=blue,line width=1pt,dotted,<->](0,0)--(-1,-1);
\draw[rotate=18,draw=blue,line width=1pt](1.5,0)--(4.5,0);
\draw[rotate=-18,draw=blue,line width=1pt](1.5,0)--(4.5,0);
\draw[rotate=18,draw=blue,line width=1pt,dashed](4.5,0)--(5,0);
\draw[rotate=-18,draw=blue,line width=1pt,dashed](4.5,0)--(5,0);
\draw[rotate=18,draw=blue,line width=1pt](-1.5,0)--(-4.5,0);
\draw[rotate=-18,draw=blue,line width=1pt](-1.5,0)--(-4.5,0);
\draw[rotate=18,draw=blue,line width=1pt,dashed](-4.5,0)--(-5,0);
\draw[rotate=-18,draw=blue,line width=1pt,dashed](-4.5,0)--(-5,0);
\node at (-0.6,-0.6) [anchor=south] {\textcolor{blue}{$\rho$}};
\node at (2.6,1.3) [anchor=south] {\textcolor{blue}{$\Re e\,\lambda = \delta\,\Im m\,\lambda$}};
\node at (2.6,-1.9) [anchor=south] {\textcolor{blue}{$\Re e\,\lambda = -\delta\,\Im m\,\lambda$}};
\filldraw [red,draw=none] (2,0) circle (0.1);
\filldraw [red,draw=none] (-2,0) circle (0.1);
\filldraw [red,draw=none] (0,0.8) circle (0.1);
\filldraw [red,draw=none] (0,-0.8) circle (0.1);
\filldraw [red,draw=none] (4,0) circle (0.1);
\filldraw [red,draw=none] (-4,0) circle (0.1);
\filldraw [red,draw=none] (4,0.8) circle (0.1);
\filldraw [red,draw=none] (4,-0.8) circle (0.1);
\filldraw [red,draw=none] (-4,0.8) circle (0.1);
\filldraw [red,draw=none] (-4,-0.8) circle (0.1);
\node at (4.7,0) [anchor=south] {$\Re e\,\lambda$};
\node at (0.7,1.8) [anchor=south] {$\Im m\,\lambda$};
\end{tikzpicture}
\caption{Lemma \ref{lemmaPapillon} ensures that the eigenvalues of $\mathscr{L}$ are located in an infinite bow tie of the complex plane.\label{fig local vp}}
\end{figure}
\begin{lemma}\label{lemmaPapillon}
There are real positive constants $\rho$, $\delta$ such that for all $\lambda\in\Cplx$ satisfying 
\[
|\lambda|>\rho\qquad\mbox{ and }\qquad |\Re e\,\lambda|<\delta\,|\Im m\,\lambda|
\]
(see Figure \ref{fig local vp}), $\mathscr{L}(\lambda):\mH^2_0(I)\to\mH^{-2}(I)$ is an isomorphism. Moreover, if $\varphi\in\mH^2_0(I)$ satisfies $\mathscr{L}(\lambda)\varphi=g\in\mH^{-2}(I)$, then there holds 
\begin{equation}\label{EstimUnifResolvante}
\|\varphi\|_{\mH^2(I,\,|\lambda|)} \le C\,\|g\|_{\mH^{-2}(I,\,|\lambda|)},
\end{equation}
where $C>0$ is independent of $g$ and $\lambda$.
\end{lemma}
\begin{proof}
Lemma \ref{resultPrelim} together with identity (\ref{SymbolComplexe}) ensure that (\ref{EstimUnifResolvante}) holds for $\lambda\in i\R$ with $|\lambda|\ge 2k$. Now let us consider the case $\lambda\notin i\R$. We write $\lambda$ as $\lambda=\pm i|\lambda|e^{i\psi}$ with $\psi\in(-\pi/2;\pi/2)$. Set $\tilde{\lambda}=\pm i|\lambda|$. Since $|\tilde{\lambda}|=|\lambda|$, by definition of the parameter dependent norm (\ref{normParam1}), for $\varphi\in\mH^2_0(I)$, we have $\|\varphi\|_{\mH^2(I,\,|\lambda|)}=\|\varphi\|_{\mH^2(I,\,|\tilde{\lambda}|)}$. Define $\tilde{g}=\mathscr{L}(\tilde{\lambda})\varphi$. Assume that $|\lambda|\ge k$. In that case, according to the first step of the proof, we have
\begin{equation}\label{EstimParam1}
\|\varphi\|_{\mH^2(I,\,|\lambda|)}=\|\varphi\|_{\mH^2(I,\,|\tilde{\lambda}|)} \le C\,\|\tilde{g}\|_{\mH^{-2}(I,\,|\tilde{\lambda}|)}.
\end{equation}
Here and in what follows, $C>0$ is a constant which can change from one line to another but which is independent of $\lambda$, $\varphi$. Now we can write
\[
\|\tilde{g}\|_{\mH^{-2}(I,\,|\tilde{\lambda}|)}=\|\tilde{g}\|_{\mH^{-2}(I,\,|\lambda|)}  \le \|g\|_{\mH^{-2}(I,\,|\lambda|)}+\|\tilde{g}-g\|_{\mH^{-2}(I,\,|\lambda|)}.
\]
A direct calculation gives, for all $\psi\in\mH^2_0(I)$,
\[
\langle\tilde{g}-g,\overline{\psi}\rangle_I=\langle \mathscr{L}(\tilde{\lambda})\varphi-\mathscr{L}(\lambda) \varphi,\overline{\psi}\rangle_{I} = (\tilde{\lambda}^2-\lambda^2)\Big(\int_{I}\varphi(\lambda^2\overline{\psi}+2d_{yy} \overline{\psi})\,dy +\int_{I}\varphi(\tilde{\lambda}^2\overline{\psi})\,dy\Big).
\]
We deduce that 
\begin{equation}\label{EstimParam2}
\|\tilde{g}-g\|_{\mH^{-2}(I,\,|\lambda|)}\le C\,|\tilde{\lambda}^2-\lambda^2|\,\|\varphi\|_{\mrm{L}^2(I)}\le C\,|e^{2i\psi}-1|\,\|\varphi\|_{\mH^2(I,\,|\lambda|)}.
\end{equation}
Thus for all $\varsigma>0$, there is $\delta$ small enough so that one has $\|\tilde{g}-g\|_{\mH^{-2}(I,\,|\lambda|)}\le \varsigma\,\|\varphi\|_{\mH^2(I,\,|\lambda|)}$ for all $\lambda=\pm i|\lambda|e^{i\psi}$ such that $|\psi|<\delta$. Gathering the latter estimate, (\ref{EstimParam1}) and (\ref{EstimParam2}) leads to 
\[
\|\varphi\|_{\mH^2(I,\,|\lambda|)} \le C\,\|g\|_{\mH^{-2}(I,\,|\lambda|)}+C\,\varsigma\,\|\varphi\|_{\mH^2(I,\,|\lambda|)}.
\]
Taking $\varsigma$ sufficiently small ($\varsigma=1/(2C)$ for example), finally we obtain (\ref{EstimUnifResolvante}).
\end{proof}
\noindent From this lemma, we deduce the following result.
\begin{theorem}\label{MainThm}
Let $\beta\in\R$ be such that $\mathscr{L}$ has no eigenvalue on the line $\Re e\,\lambda=-\beta$. Then the operator $A_{\beta}:\mathring{\mW}^2_{\beta}(\Om)\to \mathring{\mW}^{2}_{-\beta}(\Om)^{\ast}$ defined in (\ref{DefOperateurs}) is an isomorphism.
\end{theorem}
\begin{remark}\label{isoBelowk1}
Proposition \ref{propoCard} guarantees that for $k\in(0;k_1)$ ($k_1$ is the first positive threshold defined in (\ref{defThreshold})), we have $\Lambda\cap\R i=\emptyset$. From Theorem \ref{MainThm}, we deduce that when $k\in(0;k_1)$, the operator $A_{0}$ is an isomorphism from $\mH^2_0(\Om)$ to $\mH^{-2}(\Om)$. 
\end{remark}
\begin{proof}
Assume that $\mathscr{L}$ has no eigenvalue on the line $\Re e\,\lambda=-\beta$. Let us first suppose that $u\in\mathring{\mW}^2_{\beta}(\Om)$ is such that $A_{\beta}u=0$. Applying the partial Fourier-Laplace transform with respect to $x$, we obtain
\[
\mathscr{L}(\lambda)\hat{u}(\lambda,\cdot)=0,\qquad\forall\lambda\in\Cplx.
\]
From Corollary \ref{coroInvertibility}, we deduce that for all $\lambda\in\ell_{-\beta}$, $\hat{u}(\lambda,\cdot)=0$. From the properties of the inverse Fourier-Laplace transform, we deduce that $u\equiv0$. This shows that $A_{\beta}$ is injective.\\\\
We prove now that $A_\beta$ is onto. Let $f\in\mathring{\mW}^{2}_{-\beta}(\Om)^{\ast}$. Lemma \ref{lemmaPapillon} guarantees that for $\lambda\in\Cplx$ such that $\Re e\,\lambda=-\beta$ and $|\Im m\,\lambda|\ge \nu_{\beta}$, we have the estimate
\begin{equation}\label{estimResolv}
\|\mathscr{L}(\lambda)^{-1}\hat{f}(\lambda,\cdot)\|_{\mH^2(I,\,|\lambda|)} \le C\,\|\hat{f}(\lambda,\cdot)\|_{\mH^{-2}(I,\,|\lambda|)},
\end{equation}
where $C>0$ is independent of $\lambda$ and $\nu_{\beta}$ depends only $\beta$. For $\lambda\in[-\beta-i\nu_{\beta};-\beta+i\nu_{\beta}]$, the operator $\mathscr{L}(\lambda)$ is invertible according to Corollary \ref{coroInvertibility}. The continuity of $\lambda\mapsto\mathscr{L}(\lambda)^{-1}$, ensured by the analytic Fredholm theorem, guarantees that Estimate (\ref{estimResolv}) also holds for $\lambda$ in the compact set $[-\beta-i\nu_{\beta};-\beta+i\nu_{\beta}]$. Therefore (\ref{estimResolv})  is valid for all $\lambda$ such that $\Re e\,\lambda=-\beta$ with a constant $C>0$ independent of $\lambda$.\\
By definition
\[
	f\in \mathring{\mW}^{2}_{-\beta}(\Om)^{\ast}\quad\Rightarrow\quad \frac{1}{2\pi i}\,\int_{\ell_{-\beta}} \|\hat{f}(\lambda,\cdot)\|_{\mH^{-2}(I,|\lambda|)}^2\,d\lambda<+\infty.
\]
We deduce that
\begin{equation}\label{RepresentationSolution}
u(x,\cdot)=\frac{1}{2\pi i}\,\int_{\ell_{-\beta}}e^{\lambda x}\mathscr{L}(\lambda)^{-1}\hat{f}(\lambda,\cdot)\,d\lambda \quad \in\quad\mathring{\mW}^2_{\beta}(\Om)
\end{equation}
is solution of $A_{\beta}u=f$ with, by the Plancherel formulas \eqref{eq:Plancherel_1/2} and \eqref{eq:Plancherel_1/2ast}, $\|u\|_{\mathring{\mW}^2_{\beta}(\Om)}\leq C \|f\|_{\mathring{\mW}^{2}_{-\beta}(\Om)^{\ast}}.$
\end{proof}

\subsection{Source term problem in the reference strip with radiation conditions}
For $k<k_1$ ($k_1$ is the first positive threshold defined in (\ref{defThreshold})), as noticed in Remark \ref{isoBelowk1}, Problem \eqref{PbInitial} is well posed in $\mH^2_0(\Omega)$ in particular for locally supported $\mL^2$ source term. For $k>k_1$, the problem is not well posed in this setting. Indeed, since in that case $\mathscr{L}$ has an eigenvalue on the line $\Re e\,\lambda=0$, one can show that the range of $A_0$ is not closed. For $\beta\ne0$, the solution to Problem (\ref{PbInitial}) defined via the operator $A_{\beta}$ is \textit{a priori} exponentially growing as $x\to+\infty$ or as $x\to-\infty$.  The results of the previous section do not provide a solution which is physically acceptable. In what follows, we explain how to impose radiation conditions at infinity to construct a solution to Problem (\ref{PbInitial}) which decomposes as the sum of outgoing propagating modes (defined later) plus an exponentially decaying remainder.\\
\newline
In order to measure exponentially growing or decaying behaviours as $|x|\to\pm\infty$, for $\beta\in\R$, introduce the weighted Sobolev space $\mathring{\mathcal{W}}^2_{\beta}(\Om)$ defined as the completion of $\mathscr{C}^{\infty}_0(\Om)$ for the norm
\[
\|v\|_{\mathcal{W}^2_{\beta}(\Om)}=\Big
(\sum_{\alpha,\,\gamma\in\N,\ \alpha+\gamma\le 2}\|e^{-\beta |x|}\partial^{\alpha}_x\partial^{\gamma}_y v\|^2_{\mrm{L}^2(\Om)}\Big)^{1/2}.
\]
Remark the absolute value in the weight $e^{-\beta |x|}$. Due to this absolute value, observe that 
\be\label{eq:inclusion_W_beta}
\beta^1\le\beta^2\qquad\Rightarrow\qquad \mathring{\mathcal{W}}^2_{\beta^1}(\Om)\subset \mathring{\mathcal{W}}^2_{\beta^2}(\Om).
\ee
Note that this property is not true for the spaces $\mathring{\mW}^2_{\beta}(\Om)$ introduced in (\ref{FirstWeightedSpaces}). Observe also that we have $\mathring{\mathcal{W}}^2_{0}(\Om)=\mH^2_0(\Om)$. Let $\langle \cdot,\cdot\rangle_{\Om}$ stand for the bilinear duality pairing between $\mathring{\mathcal{W}}^2_{\beta}(\Om)^{\ast}$ and $\mathring{\mathcal{W}}^2_{\beta}(\Om)$, where $\mathring{\mathcal{W}}^2_{\beta}(\Om)^{\ast}$ is the topological dual space of $\mathring{\mathcal{W}}^2_{\beta}(\Om)$ endowed with the norm
\begin{equation}\label{DefNormAdjoint}
\|f\|_{\mathring{\mathcal{W}}^2_{\beta}(\Om)^{\ast}}=\sup_{v\in\mathring{\mathcal{W}}^2_{\beta}(\Om)\setminus\{0\}}\cfrac{|\langle f,\overline{v}\rangle_{\Om}|}{\|v\|_{\mathcal{W}^2_{\beta}(\Om)}}.
\end{equation}
Due to \eqref{eq:inclusion_W_beta}, we have
\be\label{eq:inclusion_W_betaast}
\beta^1\le\beta^2\qquad\Rightarrow\qquad \mathring{\mathcal{W}}^2_{\beta^2}(\Om)^{\ast}\subset \mathring{\mathcal{W}}^2_{\beta^1}(\Om)^{\ast}.
\ee
For $n\in\N^{\ast}$, pick $k\in(k_n;k_{n+1})$, the threshold wavenumbers $k_n$ being defined in (\ref{defThreshold}), and choose $\beta>0$, once for all, small enough such that $\{\lambda\in\Lambda\,|\,-\beta \le \Re e\,\lambda\le \beta\}=\Lambda\cap\R i\setminus\{0\}$. According to Corollary \ref{coroInvertibility} and Lemma \ref{lemmaPapillon}, we know that such a $\beta$ exists. We denote $\eta_1<\dots<\eta_P$ the positive real numbers (belonging to $(0;k)$ according to the proof of Lemma \ref{resultPrelim}) such that 
\begin{equation}\label{DefExposants}
\Lambda\cap\R i=\{\pm i\eta_p\}_{p=1}^{P}.
\end{equation}
For $p=1,\dots,P$, we define the propagating modes $w^{\pm}_{p}$ as
\begin{equation}\label{DefModesPropa}
w^{\pm}_{p}(x,y)=e^{\pm i\eta_p x}\varphi_p(y),
\end{equation}
where $\varphi_p$ is a non zero element of $\ker\,\mathscr{L}(i\eta_p)$. Observe that we have $(\Delta^2-k^4)w^{\pm}_{p}=0$. We normalize the $\varphi_p$ so that 
\begin{equation}\label{relationNormalization}
4\,\eta_p\int_{I}|d_y\varphi_p(y)|^2+\eta_p^2|\varphi_p(y)|^2\,dy=1.
\end{equation}
This special choice for the normalization will appear naturally in (\ref{calculFlux}). In the next step of the analysis, we shall use the following decomposition result. It can be proved exactly in the same manner as Theorem 5.4.2 of \cite{KoMR97} working with the residue theorem on formula (\ref{RepresentationSolution}) which is a result of the use of the Fourier transform in the unbounded direction. The proof of Proposition \ref{propoDecomposition} uses the fact that all modal exponents have an algebraic multiplicity of $1$, since the wavenumber is not a threshold wavenumber (see Proposition \ref{PropositionAlgMult}). In this statement and in what follows, $\chi^{\pm}\in\mathscr{C}^{\infty}(\R^2)$ is a cut-off function equal to one for $\pm x\ge 2L$ and to zero for $\pm x\le L$, for a given $L>0$. 

\begin{proposition}\label{propoDecomposition}
Assume that $k\in(k_n;k_{n+1})$, $n\in\N^{\ast}$, the threshold wavenumbers $k_n$ being defined in (\ref{defThreshold}). Assume that $u\in\mathring{\mathcal{W}}^2_{\beta}(\Om)$ is such that $(\Delta^2-k^4)u\in\mathring{\mathcal{W}}^2_{\beta}(\Om)^{\ast}\subset\mathring{\mathcal{W}}^2_{-\beta}(\Om)^{\ast}$. Then there holds the following representation
\begin{equation}\label{FormulaDecompoClamped}
u = \chi^+\sum_{p=1}^{P} (a^+_p w^{+}_{p}+a^{-}_p w^{-}_{p}) + \chi^-\sum_{p=1}^{P} (b^{-}_p w^{-}_{p}+b^{+}_p w^{+}_{p}) + \tilde{u}, 
\end{equation}
with coefficients $a^{\pm}_p$, $b^{\pm}_p\in\Cplx$ and $\tilde{u}\in \mathring{\mathcal{W}}^2_{-\beta}(\Om)$.
\end{proposition}

\begin{remark}
Observe that Formula (\ref{FormulaDecompoClamped}) for the clamped problem is the equivalent of (\ref{general_form}) for the simply supported problem. But again we emphasize that the tools to derive the two decompositions are different (see the discussion at the beginning of the section). 
\end{remark}
\noindent In the sequel, we will say that for any $u\in \mathring{\mathcal{W}}^2_{\beta}(\Om)$,
\begin{equation}\label{outgoingClamped}\tag{\text{RC}}
u\text{ is outgoing iff}\quad u= \chi^+\sum_{p=1}^{P} a_p\, w^{+}_{p} + \chi^-\sum_{p=1}^{P} b_p\,w^{-}_{p} + \tilde{u}, 
\end{equation}
with coefficients $a_p$, $b_p\in\Cplx$ and $\tilde{u}\in \mathring{\mathcal{W}}^2_{-\beta}(\Om)$. We introduce the space with detached asymptotic (see, e.g., the reviews \cite{Naza99a,Naza99b}) $\mathcal{W}^{\mrm{out}}(\Om)$ that consists of functions in $ \mathring{\mathcal{W}}^2_{\beta}(\Om)$ that satisfies \eqref{outgoingClamped}. The space $\mathcal{W}^{\mrm{out}}(\Om)$ is a Hilbert space for the inner product naturally associated with the norm 
\[
\|u\|_{\mathcal{W}^{\mrm{out}}(\Om)} = \Big( \sum_{p=1}^{P}|a_p|^2+\sum_{p=1}^{P}|b_p|^2+\|\tilde{u}\|^2_{\mathcal{W}^2_{-\beta}(\Om)} \Big)^{1/2}.
\]
For $u\in\mathcal{W}^{\mrm{out}}\subset \mathring{\mathcal{W}}^2_{\beta}(\Om)$, the map $\phi\mapsto\mathfrak{a}(u,\phi)$ with
\[
\mathfrak{a}(u,\phi)=\int_{\Om} \Delta u\,\overline{\Delta \phi}-k^4 u\,\overline{\phi}\,dxdy
\]
is well-defined in $\mathring{\mathcal{W}}^{2}_{-\beta}(\Om)$. Although $u\notin\mathring{\mathcal{W}}^{2}_{-\beta}(\Om)$ in general when $u\in\mathcal{W}^{\mrm{out}}$, we will extend it as a map in $\mathring{\mathcal{W}}^{2}_{\beta}(\Om)$. For $\phi\in\mathscr{C}^{\infty}_0(\Om)$, applying Green's formula yields
\begin{equation}\label{AfterGreen}
\mathfrak{a}(u,\phi)=
\sum_{p=1}^Pa_p\int_{\Om} (\Delta\Delta-k^4\mrm{Id})(\chi^+ w^{+}_{p})\,\overline{\phi}\,dxdy+b_p\int_{\Om} (\Delta\Delta-k^4\mrm{Id})(\chi^- w^{-}_{p})\,\overline{\phi}\,dxdy+\mathfrak{a}(\tilde{u},\phi).
\end{equation}
Since the support of $(\Delta\Delta-k^4\mrm{Id})(\chi^{\pm} w^{\pm}_{p})$ is compact, $p=1,\dots,P$, we deduce that there is a constant $C>0$ independent of $\phi\in\mathscr{C}^{\infty}_0(\Om)$ such that 
\begin{equation}\label{EstimateContinuous}
|\mathfrak{a}(u,\phi)|\le C\,\|u\|_{\mathcal{W}^{\mrm{out}}(\Om)}\|\phi\|_{\mathcal{W}^2_{\beta}(\Om)}. 
\end{equation}
By density of $\mathscr{C}^{\infty}_0(\Om)$ in $\mathring{\mathcal{W}}^2_{\beta}(\Om)$, we deduce that $\phi\mapsto\mathfrak{a}(u,\phi)$ can be uniquely extended as a continuous map in $\mathring{\mathcal{W}}^{2}_{\beta}(\Om)$. This discussion allows us to define the linear operator $\mathscr{A}^{\mrm{out}}$ such that
\begin{equation}\label{defOpFredholm}
\begin{array}{lccc}
\mathscr{A}^{\mrm{out}}: & \hspace{-0.2cm}\mathcal{W}^{\mrm{out}}(\Om) & \hspace{-0.1cm}\longrightarrow & \hspace{-0.1cm}\mathring{\mathcal{W}}^{2}_{\beta}(\Om)^{\ast}\\
 & \hspace{-0.2cm}u = \chi^+\dsp\sum_{p=1}^{P} a_p\, w^{+}_{p} + \chi^-\sum_{p=1}^{P} b_p\, w^{-}_{p} + \tilde{u} & \hspace{-0.1cm}\longmapsto & \hspace{-0.1cm}\mathscr{A}^{\mrm{out}}u
\end{array} \hspace{-0.2cm}
\end{equation}
where $\mathscr{A}^{\mrm{out}}u$ is the unique element of $\mathring{\mathcal{W}}^{2}_{\beta}(\Om)^{\ast}$ such that $\langle \mathscr{A}^{\mrm{out}}u,\overline{\phi}\rangle_{\Om}=\mathfrak{a}(u,\phi)$ for all $\phi\in\mathscr{C}^{\infty}_0(\Om)$. We deduce from (\ref{AfterGreen}) that for $v\in\mathring{\mathcal{W}}^{2}_{\beta}(\Om)$, we have
\[
\langle \mathscr{A}^{\mrm{out}}u,\overline{v}\rangle_{\Om} =\sum_{p=1}^Pa_p\int_{\Om} (\Delta\Delta-k^4\mrm{Id})(\chi^+ w^{+}_{p})\,\overline{v}\,dxdy+b_p\int_{\Om} (\Delta\Delta-k^4\mrm{Id})(\chi^- w^{-}_{p})\,\overline{v}\,dxdy+\mathfrak{a}(\tilde{u},v).
\]
\begin{theorem}\label{thmIsom}
Assume that $k\in(k_n;k_{n+1})$, $n\in\N^{\ast}$, the threshold wavenumbers $k_n$ being defined in (\ref{defThreshold}). The operator $\mathscr{A}^{\mrm{out}}$ defined in (\ref{defOpFredholm}) is an isomorphism.
\end{theorem}
\begin{remark}
Let us reformulate Theorem \ref{thmIsom} in order to compare it with Theorem \ref{th_ref}. For $f \in \mathring{\mathcal{W}}^{2}_{\beta}(\Om)^{\ast}$, Problem (\ref{PbInitial}) has a unique solution $u$ in $\mathcal{W}^{\mrm{out}}(\Om)$. In particular, for all $\beta>0$, any compactly supported function $f \in L^2(\Om)$ belongs to $\mathring{\mathcal{W}}^{2}_{\beta}(\Om)^{\ast}$ while the solution $u \in \mathcal{W}^{\mrm{out}}(\Om)$ belongs to $\mH^2_{\rm loc}(\Om)$ and satisfies the radiation conditions. Note that the result of Theorem \ref{thmIsom} is slightly stronger than the one of Theorem \ref{th_ref} concerning the assumptions for the source term. Indeed the functions of $\mathring{\mathcal{W}}^{2}_{\beta}(\Om)^{\ast}$ do not need to be compactly supported. 
\end{remark}
\noindent In order to prove Theorem \ref{thmIsom}, we need to establish an intermediate result. Define $\mathcal{W}^{\dagger}(\Om)$ the space of functions $v$ of $\mathring{\mathcal{W}}^2_{\beta}(\Om)$ that admit the representation 
\begin{equation}\label{decompositionSpaceDetached}
v = \chi^+\sum_{p=1}^{P} (a^+_p w^{+}_{p}+a^{-}_p w^{-}_{p}) + \chi^-\sum_{p=1}^{P} (b^{-}_p w^{-}_{p}+b^{+}_p w^{+}_{p}) + \tilde{v}, 
\end{equation}
with coefficients $a^{\pm}_p$, $b^{\pm}_p\in\Cplx$ and $\tilde{v}\in \mathring{\mathcal{W}}^2_{-\beta}(\Om)$. Define also the symplectic (sesquilinear and anti-hermitian) form $q_{\Om}(\cdot,\cdot)$ such that for all $u$, $v\in\mathcal{W}^{\dagger}(\Om)$, we have
\begin{equation}\label{FormSymp1}
q_{\Om}(u,v)=\langle (\Delta^2-k^4)u,\overline{v}\rangle_{\Om}-\langle (\Delta^2-k^4)\overline{v},u\rangle_{\Om}.
\end{equation}
Note that for $u$, $v\in\mathcal{W}^{\dagger}(\Om)$, the maps $(\Delta^2-k^4)u$, $(\Delta^2-k^4)\overline{v}$ are defined as elements of $\mathring{\mathcal{W}}^{2}_{\beta}(\Om)^{\ast}$ using an extension similar to what has been done above. As a consequence, we have
\[
\begin{array}{lcl}
q_{\Om}(u,v)&\hspace{-0.2cm}=&\hspace{-0.5cm}\dsp\phantom{-}\sum_{\pm}\sum_{p=1}^Pa^{\pm}_p(u)\int_{\Om} (\Delta\Delta-k^4\mrm{Id})(\chi^+ w^{\pm}_{p})\,\overline{v}\,dxdy+b^{\pm}_p(u)\int_{\Om} (\Delta\Delta-k^4\mrm{Id})(\chi^- w^{\pm}_{p})\,\overline{v}\,dxdy\\[8pt]
&\hspace{-0.2cm}&\hspace{-0.5cm}-\dsp\sum_{\pm}\sum_{p=1}^P\overline{a^{\pm}_p(v)}\int_{\Om}u\,(\Delta\Delta-k^4\mrm{Id})(\chi^+ \overline{w^{\pm}_{p}})\,dxdy+\overline{b^{\pm}_p(v)}\int_{\Om}u\,(\Delta\Delta-k^4\mrm{Id})(\chi^- \overline{w^{\pm}_{p}})\,dxdy\\[12pt]
&\hspace{-0.2cm}&\hspace{-0.5cm}+\mathfrak{a}(\tilde{u},v)-\mathfrak{a}(u,\tilde{v}).
\end{array}
\] 
Here  $a^{\pm}_p(u)$, $b^{\pm}_p(u)$ (resp. $a^{\pm}_p(v)$, $b^{\pm}_p(v)$ ) refer to the constants appearing in (\ref{decompositionSpaceDetached}) in the decomposition of $u$ (resp. $v$). In the next proposition, we show some biorthogonality relations for the modes with respect to the form $q_{\Om}(\cdot,\cdot)$. The proof is a computation, it can be skipped without altering the understanding. 

\begin{proposition}\label{PropositionFluxNRJ}
Assume that $k\in(k_n;k_{n+1})$, $n\in\N^{\ast}$, the threshold wavenumbers $k_n$ being defined in (\ref{defThreshold}). For $\nu,\,\mu\in\{+,-\}$, $j,l\in\{+,-\}$ and $m$, $p\in\{1,\dots,P\}$, for all $\tilde{u}$, $\tilde{v}\in\mathring{\mathcal{W}}^2_{-\beta}(\Om)$, we have  
\[
\begin{array}{ll}
q_{\Om}(\chi^{\nu}w^{j}_m+\tilde{u},\chi^{\mu}w^{l}_p+\tilde{v})=-ij\,\nu\,\delta_{\nu,\,\mu}\,\delta_{j,\,l}\,\delta_{m,\,p}.
\end{array}
\]
\end{proposition}
\begin{proof}
First, integrating by parts, we find that $q_{\Om}(\chi^{\nu}w^{j}_m+\tilde{u},\chi^{\mu}w^{l}_p+\tilde{v})=q_{\Om}(\chi^{\nu}w^{j}_m,\chi^{\mu}w^{l}_p)$ for all $\tilde{u}$, $\tilde{v}\in\mathring{\mathcal{W}}^2_{-\beta}(\Om)$. On the other hand, observing that $(\Delta\Delta-k^4)w_m^j=0$ for all $j\in\{+,-\}$, $m\in\{1,\dots,P\}$, and that $\chi^{\pm}=1$ for $\pm x\ge2L$, we can write 
\[
q_{\Om}(\chi^{\nu}w^{j}_m,\chi^{\mu}w^{l}_p)=\int_{\Om_{H}}\Delta\Delta(\chi^{\nu}w^{j}_m)\,\overline{\chi^{\mu}w^{l}_p}-\chi^{\nu}w^{j}_m\,\Delta\Delta(\overline{\chi^{\mu}w^{l}_p}) \,dxdy,\qquad\forall H\ge2L.
\]
Here we use the notation $\Om_H:=\{(x,y)\in\Om\,|\,|x|\le H\}$. Integrating by parts, we get
\[
\begin{array}{l}
q_{\Om}(\chi^{\nu}w^{j}_m,\chi^{\mu}w^{l}_p)=\dsp\delta_{\nu,\,\mu}\int_{\Sigma_{H}}\partial_n\Delta w^{j}_m\,\overline{w^{l}_p}-w^{j}_m\,\partial_n\Delta\overline{w^{l}_p} \,dy\\[10pt]
\hspace{2.7cm}-\dsp\delta_{\nu,\,\mu}\int_{\Sigma_{H}}\Delta w^{j}_m\,\partial_n\overline{w^{l}_p}-\partial_nw^{j}_m\,\Delta\overline{w^{l}_p} \,dy,\qquad\forall H\ge 2L,
\end{array}
\]
with $\Sigma_{H}:=\{-H\}\times(0;1)\cup\{H\}\times(0;1)$ and $\partial_n=\pm\partial_x$ at $x=\pm H$. We deduce
\begin{equation}\label{expressionA}
\begin{array}{l}
q_{\Om}(\chi^{\nu}w^{j}_m,\chi^{\mu}w^{l}_p)=\dsp\delta_{\nu,\,\mu}\,e^{i\nu(j\eta_m-l\eta_p)H}\,\mathfrak{J},\qquad\forall H\ge 2L,
\end{array}
\end{equation}
where the quantity $\mathfrak{J}$ is independent of $H\ge 2L$. Since $q_{\Om}(\chi^{\nu}w^{j}_m,\chi^{\mu}w^{l}_p)$ is also independent of $H\ge 2L$, we must have $q_{\Om}(\chi^{\nu}w^{j}_m,\chi^{\mu}w^{l}_p)=0$ if $j\eta_m-l\eta_p\ne0\Leftrightarrow[j\ne l\mbox{ or }m\ne p]$. To conclude the proof, it remains to study the case $\nu=\mu$, $j=l$ and $m=p$. Writing more precisely the quantity $\mathfrak{J}$ in (\ref{expressionA}), we find
\begin{equation}\label{calculFlux}
q_{\Om}(\chi^{\nu}w^{j}_m,\chi^{\nu}w^{j}_m)=-4ij\,\nu\,\eta_m\int_{I}|d_y\varphi_m(y)|^2+\eta_m^2|\varphi_m(y)|^2\,dy=-ij\,\nu.
\end{equation}
To obtain the second equality in (\ref{calculFlux}), we used (\ref{relationNormalization}). 
\end{proof}
\noindent \textit{Proof of Theorem \ref{thmIsom}.} From (\ref{EstimateContinuous}), we see that the operator 
$\mathscr{A}^{\mrm{out}}$ defined in (\ref{defOpFredholm}) is continuous. On the other hand, if 
\[
u=\chi^+\dsp\sum_{p=1}^{P} a_p\, w^{+}_{p} + \chi^-\sum_{p=1}^{P} b_p\, w^{-}_{p} + \tilde{u}
\]
belongs to $\ker\,\mathscr{A}^{\mrm{out}}$, then $q_{\Om}(u,u)=0$. From Proposition \ref{PropositionFluxNRJ}, this implies 
\[
i\sum_{p=1}^{P} |a_p|^2+i\sum_{p=1}^{P} |b_p|^2=0.
\]
We deduce that $u=\tilde{u}\in\mathring{\mathcal{W}}^2_{-\beta}(\Om)$ and so $u$ is in $\ker\,A_{-\beta}$ and in $\ker\,A_{\beta}$ which are both reduced to $\{0\}$ (Theorem  \ref{MainThm} together with the fact that $\mathscr{L}$ has no eigenvalue on the lines $\Re e\,\lambda=\pm\beta$). Therefore, $\mathscr{A}^{\mrm{out}}$ is injective. To conclude the proof, it remains to show that $\mathscr{A}^{\mrm{out}}$ is onto. Consider $f\in\mathring{\mathcal{W}}^{2}_{\beta}(\Om)^{\ast}\subset\mathring{\mW}^{2}_{-\beta}(\Om)^{\ast}$. Since $A_{\beta}$ is onto (Theorem \ref{MainThm}), there is some $v\in\mathring{\mW}^{2}_{\beta}(\Om)\subset\mathring{\mathcal{W}}^{2}_{\beta}(\Om)$ such that $A_{\beta}v=f$. According to Proposition \ref{propoDecomposition}, $v$ admits the following decomposition 
\[
v = \chi^+\sum_{p=1}^{P} (a^+_p w^{+}_{p}+a^{-}_p w^{-}_{p}) + \chi^-\sum_{p=1}^{P} (b^{-}_p w^{-}_{p}+b^{+}_p w^{+}_{p}) + \tilde{u}, 
\]
with coefficients $a^{\pm}_p$, $b^{\pm}_p\in\Cplx$ and $\tilde{u}\in \mathring{\mathcal{W}}^2_{-\beta}(\Om)$. Set 
\[
u:=v-\sum_{p=1}^{P} a^-_p w^-_{p}-\sum_{p=1}^{P} b^+_p w^+_{p}.
\] 
One can see that $u$ belongs to the space $\mathcal{W}^{\mrm{out}}(\Om)$. On the other hand, observing that the $w^{\pm}_p$ satisfy $(\Delta^2-k^4)w^{\pm}_p=0$, we obtain $\mathscr{A}^{\mrm{out}}u=f$. This shows that $\mathscr{A}^{\mrm{out}}$ is onto.\hfill$\square$

\subsection{Problems in the perturbed strip with radiation conditions}
We previously saw that for the simply supported strip, the unperturbed and perturbed cases where handled quite similarly.
In the case of the clamped strip, the perturbed case is significantly more difficult than the unperturbed one, in the sense that additional arguments have to be introduced.  
Let us come back to the original Problem (\ref{u_model}) with a hole $\mathscr{O}$ in the clamped case: 
\begin{equation}\label{PbObstacle}
\left\{
\begin{array}{rcll}
\Delta^2 u-k^4u & = & f & \mbox{ in }D\\[3pt]
u= \partial_n u  & = & 0  & \mbox{ on }\partial\Om\\[3pt]
Mu= Nu  & = & 0  & \mbox{ on }\partial\mathscr{O},
\end{array}
\right.
\end{equation}
where $f$ will be specified later on. To set ideas, we assume in this paragraph that $L$ is chosen so that $\overline{\mathscr{O}}\subset(-L;L)\times(0;1)$. 
Problem (\ref{PbObstacle}) leads to consider the variational equality
\[
\mathfrak{b}(u,v)=\mathfrak{m}(v),\qquad\forall v\in\{\phi|_{D}\,|\,\phi\in\mathscr{C}^{\infty}_0(\Om)\}, 
\]
\[
\begin{array}{lccl}
\mbox{ with } & \mathfrak{b}(u,v)&=&\dsp\int_{\Om} \nu\Delta u\,\overline{\Delta v}+(1-\nu)\left(\frac{\partial^2u}{\partial x^2}\,\frac{\partial^2\overline{v}}{\partial x^2}+2\frac{\partial^2u}{\partial x\partial y}\,\frac{\partial^2\overline{v}}{\partial x\partial y}+\frac{\partial^2u}{\partial y^2}\,\frac{\partial^2\overline{v}}{\partial y^2}\right)-k^4 u\,\overline{v}\,dxdy\\[12pt]
&\mathfrak{m}(v)&=&\langle f,\overline{v}\rangle_{\Om}.
\end{array}
\]
 Observe that the functions of $\{\phi|_{D}\,|\,\phi\in\mathscr{C}^{\infty}_0(\Om)\}$ do not  necessarily vanish on $\partial\mathscr{O}$. Now, we introduce notation similar to the one of the two previous paragraphs in the geometry $D$ instead of $\Om$. For $\beta\in\R$, define the weighted Sobolev space $\mathring{\mathcal{W}}^2_{\beta}(D)$ as the completion of $\{\phi|_{D}\,|\,\phi\in\mathscr{C}^{\infty}_0(\Om)\}$ for the norm
\[
\|v\|_{\mathcal{W}^2_{\beta}(D)}=\Big
(\sum_{\alpha,\,\gamma\in\N,\ \alpha+\gamma\le 2}\|e^{-\beta |x|}\partial^{\alpha}_x\partial^{\gamma}_y v\|^2_{\mrm{L}^2(D)}\Big)^{1/2}.
\]
Again, remark the absolute value in the weight $e^{-\beta |x|}$. We denote $\mathring{\mathcal{W}}^2_{\beta}(D)^{\ast}$ the topological dual space of $\mathring{\mathcal{W}}^2_{\beta}(D)$ endowed with the norm (\ref{DefNormAdjoint}), $\Om$ being replaced by $D$. We define the linear and bounded operator $\mathscr{B}_{\beta}:\mathring{\mathcal{W}}^2_{\beta}(D)\to \mathring{\mathcal{W}}^{2}_{-\beta}(D)^{\ast}$ such that 
\begin{equation}\label{DefOperateursObstacle}
\langle \mathscr{B}_{\beta} u,\overline{v}\rangle_{D}=\mathfrak{b}(u,v),\qquad \forall (u,v)\in\mathring{\mathcal{W}}^2_{\beta}(D)\times \mathring{\mathcal{W}}^2_{-\beta}(D).
\end{equation}
One can easily prove that $\mathscr{B}_{\beta}^\ast=\mathscr{B}_{-\beta}$.
\\\\
As in the previous section, for $n\in\N^{\ast}$, pick $k\in(k_n;k_{n+1})$. In what follows, the weight $\beta>0$ is chosen small once for all enough such that $\{\lambda\in\Lambda\,|\,-\beta \le \Re e\,\lambda\le \beta\}=\Lambda\cap\R i\setminus\{0\}$. Using again the notation introduced in (\ref{DefExposants}), (\ref{DefModesPropa}) for the $w^{\pm}_p$, we define the space $\mathcal{W}^{\mrm{out}}(D)$ that consists of functions $v\in \mathring{\mathcal{W}}^2_{\beta}(D)$ that admit the representation 
\[
v = \chi^+\sum_{p=1}^{P} a_p\, w^{+}_{p} + \chi^-\sum_{p=1}^{P} b_p\,w^{-}_{p} + \tilde{v}, 
\]
with coefficients $a_p$, $b_p\in\Cplx$ and $\tilde{v}\in \mathring{\mathcal{W}}^2_{-\beta}(D)$. We remind the reader that $\chi^{\pm}\in\mathscr{C}^{\infty}(\R^2)$ is a cut-off function equal to one for $\pm x\ge 2L$ and to zero for $\pm x\le L$. The space $\mathcal{W}^{\mrm{out}}(D)$ is a Hilbert space for the inner product naturally associated with the norm 
\[
\|v\|_{\mathcal{W}^{\mrm{out}}(D)} = \Big( \sum_{p=1}^{P}|a_p|^2+\sum_{p=1}^{P}|b_p|^2+\|\tilde{v}\|^2_{\mathcal{W}^2_{-\beta}(D)} \Big)^{1/2}.
\]
Working as we did in (\ref{defOpFredholm}) for $\mathscr{A}^{\mrm{out}}$, we define the linear operator $\mathscr{B}^{\mrm{out}}$ such that
\begin{equation}\label{defOpFredholmObstacle}
\begin{array}{lccc}
\mathscr{B}^{\mrm{out}}: & \hspace{-0.2cm}\mathcal{W}^{\mrm{out}}(D) & \hspace{-0.1cm}\longrightarrow & \hspace{-0.1cm}\mathring{\mathcal{W}}^{2}_{\beta}(D)^{\ast}\\
 & \hspace{-0.2cm}u = \chi^+\dsp\sum_{p=1}^{P} a_p\, w^{+}_{p} + \chi^-\sum_{p=1}^{P} b_p\, w^{-}_{p} + \tilde{u} & \hspace{-0.1cm}\longmapsto & \hspace{-0.1cm}\mathscr{B}^{\mrm{out}}u
\end{array} \hspace{-0.2cm}
\end{equation}
where $\mathscr{B}^{\mrm{out}}u$ is defined as the functional such that for all $v\in\mathring{\mathcal{W}}^{2}_{\beta}(D)$
\[
\langle \mathscr{B}^{\mrm{out}}u,\overline{v}\rangle_{D} = \dsp\sum_{p=1}^{P} a_p\int_{D} (\Delta\Delta-k^4\mrm{Id})(\chi^+w^{+}_{p})\overline{v}\,dxdy+\dsp\sum_{p=1}^{P} b_p\int_{D} (\Delta\Delta-k^4\mrm{Id})(\chi^-w^{-}_{p})\overline{v}\,dxdy+\mathfrak{b}(\tilde{u},v).
\]
\noindent As in the previous section, in order to prove our main theorem for  $\mathscr{B}^{\mrm{out}}$, we need to establish intermediate results. Let us define $\mathcal{W}^{\dagger}(D)$ the space of functions $v$ of $\mathring{\mathcal{W}}^2_{\beta}(D)$ that admit the representation 
\[
v = \chi^+\sum_{p=1}^{P} (a^+_p w^{+}_{p}+a^{-}_p w^{-}_{p}) + \chi^-\sum_{p=1}^{P} (b^{-}_p w^{-}_{p}+b^{+}_p w^{+}_{p}) + \tilde{v}, 
\]
with coefficients $a^{\pm}_p$, $b^{\pm}_p\in\Cplx$ and $\tilde{v}\in \mathring{\mathcal{W}}^2_{-\beta}(D)$. Let us introduce also the symplectic form $q_{D}(\cdot,\cdot)$ such that for all $u$, $v\in\mathcal{W}^{\dagger}(D)$, 
\begin{equation}\label{DefFormSympl}
q_{D}(u,v)=\langle \mathscr{B}_{\beta}u,\overline{v}\rangle_{D}-\langle \mathscr{B}_{\beta}\overline{v},u\rangle_{D}.
\end{equation}
Here $\mathscr{B}_{\beta}u$ and $\mathscr{B}_{\beta}\overline{v}$ must be regarded as elements of $\mathring{\mathcal{W}}^2_{\beta}(D)^{\ast}$ defined using the extension by continuity process presented in (\ref{EstimateContinuous}). Working exactly as in the proof of Proposition \ref{PropositionFluxNRJ}, one can establish the following result.
\begin{proposition}\label{PropositionFluxNRJObstacle}
Assume that $k\in(k_n;k_{n+1})$, $n\in\N^{\ast}$, the threshold wavenumbers $k_n$ being defined in (\ref{defThreshold}). For $\nu,\,\mu\in\{+,-\}$, $j,l\in\{+,-\}$ and $m$, $p\in\{1,\dots,P\}$, for all $\tilde{u}$, $\tilde{v}\in\mathring{\mathcal{W}}^2_{-\beta}(D)$, we have  
\[
\begin{array}{ll}
q_{D}(\chi^{\nu}w^{j}_m+\tilde{u},\chi^{\mu}w^{l}_p+\tilde{v})=-ij\,\nu\,\delta_{\nu,\,\mu}\,\delta_{j,\,l}\,\delta_{m,\,p}.
\end{array}
\]
\end{proposition}
\noindent 
The following theorem is the equivalent of Theorems \ref{ThmScaSimply1} and \ref{ThmScaSimply2} in the simply supported case. In other words, it solves the corresponding scattering problems. Moreover, such theorem is used in the proof of the main result of this section, that is Theorem \ref{thmIsomObstacle}. We postpone the proof of Theorem \ref{TheoremDecompoKernel} to the end of this section. 
\begin{theorem}\label{TheoremDecompoKernel}
Assume that $k\in(k_n;k_{n+1})$, $n\in\N^{\ast}$, the threshold wavenumbers $k_n$ being defined in (\ref{defThreshold}).\\
1) The operators $\mathscr{B}_{\pm\beta}$ are of Fredholm type.\\
2) Moreover, we have $\dim\,\ker\,\mathscr{B}_{\beta}-\dim\,\ker\,\mathscr{B}_{-\beta}=2P$ and there are functions $\Psi_1,\dots,\Psi_{2P}\in\ker\,\mathscr{B}_{\beta}$ admitting the decomposition, for $p=1,\dots,P$,  
\begin{equation}\label{DefElemNoyau}
\begin{array}{lcl}
\Psi_p&\hspace{-0.2cm}=&\hspace{-0.2cm}\chi^+\,w^-_p+\chi^+\dsp\sum_{m=1}^{P} s_{p\,m}\,w^{+}_{m} + \chi^-\sum_{m=1}^{P} s_{p\,P+m}\,w^{-}_{m} + \tilde{\Psi}_p,\\[15pt] 
\Psi_{P+p}&\hspace{-0.2cm}=&\hspace{-0.2cm}\chi^-\,w^+_p+\chi^+\dsp\sum_{m=1}^{P} s_{P+p\,m}\,w^{+}_{m} + \chi^-\sum_{m=1}^{P} s_{P+p\,P+m}\,w^{-}_{m} + \tilde{\Psi}_{P+p}. 
\end{array}
\end{equation}
Here, the $\tilde{\Psi}_p$, $p=1,\dots,2P$, belong to $\mathring{\mathcal{W}}^2_{-\beta}(D)$ and the scattering matrix $\mathbb{S}:=(s_{p\,m})_{1\le p,m\le 2P}\in\Cplx^{2P\times2P}$ is uniquely defined, unitary ($\mathbb{S}\,\overline{\mathbb{S}}^{\top}=\mrm{Id}^{2P\times2P}$) and symmetric ($\mathbb{S}^{\top}=\mathbb{S}$). 
\end{theorem}
\noindent  Now we state the main result of the section, which is the equivalent of Theorem \ref{thmIsomObstacleSimply} in the simply supported case.
\begin{theorem}\label{thmIsomObstacle}
Assume that $k\in(k_n;k_{n+1})$, $n\in\N^{\ast}$, the threshold wavenumbers $k_n$ being defined in (\ref{defThreshold}).\\
1) The operator $\mathscr{B}^{\mrm{out}}$ defined in (\ref{defOpFredholmObstacle}) is Fredholm of index zero and $\ker\,\mathscr{B}^{\mrm{out}}=\ker\,\mathscr{B}_{-\beta}$. As a consequence,\\[3pt]
$a)$ If $\ker\,\mathscr{B}_{-\beta}=\{0\}$, then $\mathscr{B}^{\mrm{out}}$ is an isomorphism.\\[3pt]
$b)$ If $\ker\,\mathscr{B}_{-\beta}=\mrm{span}(z_1,\dots,z_d)$ for some $d\ge1$, then the equation $\mathscr{B}^{\mrm{out}}u=f\in\mathring{\mathcal{W}}^{2}_{\beta}(D)^{\ast}$ admits a solution (defined up to an element of $\ker\,\mathscr{B}_{-\beta}$) if and only if $f $ satisfies the compatibility conditions $\langle f,\overline{z_j}\rangle_{D}=0$ for $j=1,\dots,d$.\\
\newline
2) If $u\in\mathcal{W}^{\mrm{out}}(D)$ is such that $\mathscr{B}^{\mrm{out}}u=f\in\mathring{\mathcal{W}}^{2}_{\beta}(D)^{\ast}$, then we have
\[
u-\chi^{+}\sum_{p=1}^{P}c_p\,w^+_p-\chi^{-}\sum_{p=1}^{P}c_{p+P}\,w^-_p\in\mathring{\mathcal{W}}^{2}_{-\beta}(D)\qquad\mbox{with}\qquad c_p=i\langle f,\Psi_p \rangle_D, \ p=1,\dots,2P,
\]
the $\Psi_p\in\ker\,\mathscr{B}_{\beta}$ being defined in (\ref{DefElemNoyau}).
\end{theorem}
\begin{remark}
When $k\in(0;k_1)$, using the result of Remark \ref{isoBelowk1}, one can prove that Fredholmness of (\ref{PbObstacle}) holds in $\{u \in \mH^2(D)\,|\, u=\partial_nu=0\mbox{ \rm{on} }\partial\Om\}$. In particular, we do not need to impose radiation conditions. 
\end{remark}
\begin{remark}
From Theorem \ref{thmIsomObstacle}, we deduce that if Problem (\ref{PbObstacle}) for $f=0$ has only the zero solution in $\mathring{\mathcal{W}}^2_{-\beta}(D)$, then for any $f \in \mathring{\mathcal{W}}^{2}_{\beta}(D)^{\ast}$, Problem (\ref{PbObstacle}) has a unique solution $u$ in $\mathcal{W}^{\mrm{out}}(D)$. In order to make the connection with the result of Theorem \ref{thmIsomObstacleSimply}, observe that any compactly supported function $f \in L^2(D)$ belongs to $\mathring{\mathcal{W}}^{2}_{\beta}(D)^{\ast}$ while the solution $u \in \mathcal{W}^{\mrm{out}}(D)$ belongs to $\mH^2_{\rm loc}(D)$ and satisfies the radiation conditions. Notice also that the equality $\ker\,\mathscr{B}^{\mrm{out}}=\ker\,\mathscr{B}_{-\beta}$ is the equivalent of the result established in Remark \ref{RmqTrappedModes}. It says that the elements of the kernel of the problem, if they exist, are exponentially decaying at infinity. In other words, they are trapped modes. 
\end{remark}

\noindent \textit{Proof of Theorem \ref{thmIsomObstacle}.} 1) \textit{i)} First we show that $\ker\,\mathscr{B}^{\mrm{out}}=\ker\,\mathscr{B}_{-\beta}$. Clearly we have 
\[
\ker\,\mathscr{B}_{-\beta}\subset\ker\,\mathscr{B}^{\mrm{out}}.
\]
It is then sufficient to establish that $\ker\,\mathscr{B}^{\mrm{out}}\subset\ker\,\mathscr{B}_{-\beta}$. Assume that 
\[
u=\chi^+\dsp\sum_{p=1}^{P} a_p\, w^{+}_{p} + \chi^-\sum_{p=1}^{P} b_p\, w^{-}_{p} + \tilde{u}
\]
belongs to $\ker\,\mathscr{B}^{\mrm{out}}$. Then one has $q_{D}(u,u)=0$. Using Proposition \ref{PropositionFluxNRJObstacle}, this implies
\[
i\sum_{p=1}^{P} |a_p|^2+i\sum_{p=1}^{P} |b_p|^2=0
\]
and shows that $u\in\ker\,\mathscr{B}_{-\beta}$.\\ 
\newline
\noindent \textit{ii)} Now, let us prove that $\mathscr{B}^{\mrm{out}}$ has a closed range and that $\dim\,\coker\,\mathscr{B}^{\mrm{out}}=\dim\,\ker\, \mathscr{B}_{-\beta}$. Theorem \ref{TheoremDecompoKernel} ensures that $\mathscr{B}_{-\beta}$ is a Fredholm operator. Therefore $\ker\,\mathscr{B}_{-\beta}$ is of finite dimension. Assume that $\ker\,\mathscr{B}_{-\beta}=\mrm{span}(z_1,\dots,z_d)$ where the functions $z_1,\dots,z_d$ are linearly independent. The case $\ker\, \mathscr{B}_{-\beta}=\{0\}$, simpler to study, is left to the reader. Consider some $f\in\mathring{\mathcal{W}}^{2}_{\beta}(D)^{\ast} $ satisfying the compatibility conditions $\langle f,\overline{z_j}\rangle_{D}=0$ for $j=1,\dots,d$. This is equivalent to $f\in (\ker\,\mathscr{B}_{-\beta})^\perp$. Since $\mathscr{B}_{\beta}=\mathscr{B}_{-\beta}^{*}$ and since the range of $\mathscr{B}_{\beta}$ is closed (because $\mathscr{B}_{\beta}$ is of Fredholm type), this is also equivalent to the fact that $f$ belongs to the range of $\mathscr{B}_{\beta}$. Then there is some $v\in\mathring{\mathcal{W}}^{2}_{\beta}(D)$ such that $\mathscr{B}_{\beta}v=f$. Moreover, multiplying $v$ by a well suited cut-off function and using Proposition \ref{propoDecomposition}, one obtains that $v$ admits the decomposition 
\[
v = \chi^+\sum_{p=1}^{P} (a^+_p w^{+}_{p}+a^{-}_p w^{-}_{p}) + \chi^-\sum_{p=1}^{P} (b^{-}_p w^{-}_{p}+b^{+}_p w^{+}_{p}) + \tilde{u}, 
\]
with coefficients $a^{\pm}_p$, $b^{\pm}_p\in\Cplx$ and $\tilde{u}\in \mathring{\mathcal{W}}^2_{-\beta}(D)$. Set 
\[
u:=v-\sum_{p=1}^{P} a^-_p \Psi^{\phantom{+}}_{p}-\sum_{p=1}^{P} b^+_p \Psi^{\phantom{+}}_{P+p}.
\] 
One observes that $u$ belongs to the space $\mathcal{W}^{\mrm{out}}(D)$. Besides, since the $\Psi_p$ are in $\ker\,\mathscr{B}_{\beta}$, we obtain $\mathscr{B}^{\mrm{out}}u=\mathscr{B}_{\beta}v=f$. This shows on the one hand that the range of $\mathscr{B}_{\beta}$ is included in the one of $\mathscr{B}^{\mrm{out}}$. Since $\mathcal{W}^{\mrm{out}}(D)\subset \mathring{\mathcal{W}}^{2}_{\beta}(D)$ the two ranges coincide and then the range of $\mathscr{B}^{\mrm{out}}$ is closed. This shows on the other hand that $\dim\,\coker\, \mathscr{B}^{\mrm{out}}\le d=\dim\,\ker\, \mathscr{B}_{-\beta}$. Now assume by contradiction that $\dim\,\coker\, \mathscr{B}^{\mrm{out}}< d$. In that case, we can find $f=\mathscr{B}^{\mrm{out}}u$ with $u\in\mathcal{W}^{\mrm{out}}(D)$ such that $\langle f,\overline{z_j}\rangle_D\ne0$ for some $j\in\{1,\dots,d\}$. Then we have $\langle \mathscr{B}^{\mrm{out}}u,\overline{z_j}\rangle_D=\langle \mathscr{B}_{\beta}u,\overline{z_j}\rangle_D=\overline{\langle \mathscr{B}_{-\beta}z_j,\overline{u}\rangle_D}=0$ which contradicts the fact that $\langle f,\overline{z_j}\rangle_D\ne0$. Thus there holds $\dim\,\coker\, \mathscr{B}^{\mrm{out}}= d=\dim\,\ker\, \mathscr{B}^{\mrm{out}}$ so that $\mrm{ind}\,\mathscr{B}^{\mrm{out}} = \dim\,\ker\,\mathscr{B}^{\mrm{out}} - 
\dim\,\coker\, \mathscr{B}^{\mrm{out}}=0$.\\
\newline
Finally, we show statement 2). For $f\in\mathring{\mathcal{W}}^{2}_{\beta}(D)^{\ast}$ (satisfying the compatibility conditions if $\ker\, \mathscr{B}_{-\beta}\ne\{0\}$), consider
\[
u=\chi^{+}\sum_{p=1}^{P}c_p\,w^+_p+\chi^{-}\sum_{p=1}^{P}c_{p+P}\,w^-_p+\tilde{u}\in\mathcal{W}^{\mrm{out}}(D)
\]
a solution to the equation $\mathscr{B}^{\mrm{out}}u=f$. Then for $p=1,\dots,2P$, using Theorem \ref{TheoremDecompoKernel}, we find $\langle f,\Psi_p\rangle_D=q_D(u,\overline{\Psi_p})=-i\,c_p$. This leads to the desired result.~\hfill$\square$\\
\newline
We conclude this section by giving the proof of Theorem \ref{TheoremDecompoKernel}.

\begin{figure}[!ht]
\centering
\begin{tikzpicture}[scale=1.2]
\draw[draw=black,line width=1pt,dashed](-5,2)--(-5.7,2);
\draw[draw=black,line width=1pt,dashed](5,2)--(5.7,2);
\draw[draw=black,line width=1pt,dashed](-5,0)--(-5.7,0);
\draw[draw=black,line width=1pt,dashed](5,0)--(5.7,0);
\draw[draw=none,line width=1pt,fill=gray!20](-5,2)--(-1,2)--(-1,0)--(-5,0);
\draw[draw=none,line width=1pt,fill=gray!20](5,2)--(1,2)--(1,0)--(5,0);
\draw[draw=none,line width=1pt,pattern=north west lines, pattern color=black,very thin](-2,2)--(2,2)--(2,0)--(-2,0)--cycle;
\draw[draw=black,line width=1pt](-5,2)--(5,2);
\draw[draw=black,line width=1pt](5,0)--(-5,0);
\begin{scope}[scale=0.2,shift={(2,4)}]
\draw [line width=1pt,fill=white]  plot[smooth, tension=.7] coordinates {(-4,2.5) (-3,3) (-2,2.8) (-0.8,2.5) (-0.5,1.5) (0.5,0) (0,-2)(-1.5,-2.5) (-4,-2) (-3.5,-0.5) (-5,1) (-4,2.5)};
\end{scope}
\node at (-2.6,1.7) {$\om_1$};
\draw [draw=none,fill=white] (-0.25,1.55) rectangle (0.25,1.9);
\node at (0,1.7) {$\om_2$};
\node at (2.6,1.7) {$\om_3$};
\node at (0,0.8) {$\mathscr{O}$};
\node at (1,-0.1) [anchor=north] {$L$};
\node at (2,-0.1) [anchor=north] {$2L$};
\node at (-1,-0.1) [anchor=north] {$-L$};
\node at (-2,-0.1) [anchor=north] {$-2L$};
\end{tikzpicture}\vspace{-0.2cm}
\caption{Partition of unity used in the proof of Theorem \ref{TheoremDecompoKernel}.\label{FigPartition}}
\end{figure}
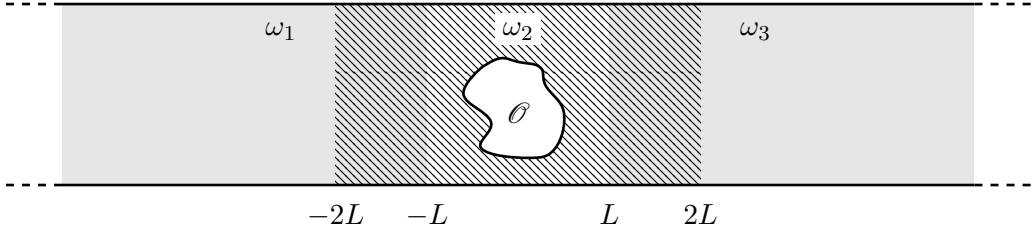

\noindent \textit{Proof of Theorem \ref{TheoremDecompoKernel}.} 
1) First we show that $\mathscr{B}_{\pm\beta}$ are Fredholm operators. Since $\mathscr{B}_{-\beta}$ is the adjoint of $\mathscr{B}_{+\beta}$, it is sufficient to establish the result for $\mathscr{B}_{+\beta}$. 
The strategy is the following. In order to prove that $\mathscr{B}_{\beta}$ is a Fredholm operator, the first step consists in proving
that $\mrm{range}\,\mathscr{B}_{\beta}$ is closed and $\ker\,\mathscr{B}_{\beta}$ is finite-dimensional. This will be a consequence of inequality
(\ref{AprioriEstim}) and Lemma \ref{PeetreLemma} in appendix. The second step consists in proving that $\coker\,\mathscr{B}_{\beta}$ is finite-dimensional, which 
will be a consequence of the existence of a right regularizer of $\mathscr{B}_{\beta}$ and of \cite[Lem. 2.23]{McLe00}.
\\
Define the domains
\[
\om_1:=(-\infty;-L)\times(0;1)\qquad\om_2:=\{(x,y)\in D\,|\,|x|<2L\}\qquad\om_3:=(+L;+\infty)\times(0;1)
\]
(see Figure \ref{FigPartition}). For $\nu=1,\dots,3$, let $\zeta_\nu$ and $\psi_\nu$ be $\mathscr{C}^{\infty}$ functions (with support in $\overline{D}$) satisfying 
the conditions 
\[
\mbox{supp }\zeta_{\nu} \subset \mbox{supp }\psi_{\nu}\subset\omega_{\nu},\quad\qquad\zeta_{\nu}\psi_{\nu}=\zeta_{\nu},\quad\qquad\sum_{\nu=1}^3\zeta_{\nu}=1\mbox{ in }D.
\]
Note in particular that $\zeta_2=\zeta_3=0$ for $x\le-L$, $\zeta_1=\zeta_3=0$ for $|x|\le L$ and $\zeta_1=\zeta_2=0$ for $x\ge L$. Define the space $\mH^2_{\PetitCarre}(\om_2):=\{\varphi\in\mH^2(\om_2)\,|\,\varphi=\partial_n\varphi=0\mbox{ on }\partial\om_2\setminus\partial \mathscr{O}\}$ endowed with the inner product of $\mH^2(\om_2)$. Introduce the unique linear continuous operator $\mathscr{B}_{\PetitCarre}: \mH^2_{\PetitCarre}(\om_2)\to 
\mH^2_{\PetitCarre}(\om_2)^{\ast}$ such that for all $u$, $v\in\mH^2_{\PetitCarre}(\om_2)$, 
\[
\left< \mathscr{B}_{\PetitCarre} u,\overline{v}\right>_{\om_2} = \dsp\int_{\om_2} \nu\Delta u\,\overline{\Delta v}+(1-\nu)\left(\frac{\partial^2u}{\partial x^2}\,\frac{\partial^2\overline{v}}{\partial x^2}+2\frac{\partial^2u}{\partial x\partial y}\,\frac{\partial^2\overline{v}}{\partial x\partial y}+\frac{\partial^2u}{\partial y^2}\,\frac{\partial^2\overline{v}}{\partial y^2}\right)\,dxdy.
\]
According to Lemma \ref{poincare}, we know that for $\nu\in[0;1)$, the operator $\mathscr{B}_{\PetitCarre}$ is an isomorphism.\\

\noindent Let us prove the following \textit{a priori} estimate:
\begin{equation}\label{AprioriEstim}
\Vert u \Vert_{\mathcal{W}^{2}_{\beta}(D)} \le C\,(\Vert  \mathscr{B}_{\beta}u\Vert_{\mathring{\mathcal{W}}^{2}_{-\beta}(D)^{\ast}}+
\Vert u\Vert_{\mH^1(\om_2)}),\quad\forall u\in\mathring{\mathcal{W}}^{2}_{\beta}(D).
\end{equation}
For $u\in\mathring{\mathcal{W}}^{2}_{\beta}(D)$, noticing that $\mathscr{B}_{\beta}(\zeta_1u)=A_{-\beta}(\zeta_1u)$ and $\mathscr{B}_{\beta}(\zeta_3u)=A_{\beta}(\zeta_3u)$ because the supports of $\zeta_1u$, $\zeta_3u$ do not meet $\partial\mathscr{O}$, we can write 
\begin{equation}\label{estimApriori}
\begin{array}{lcl}
\Vert u \Vert_{\mathcal{W}^{2}_{\beta}(D)} & \le & C\,(\Vert \zeta_1 u \Vert_{\mW^{2}_{-\beta}(\Om)}+
\Vert \zeta_2 u \Vert_{\mH^{2}(\omega_2)}+\Vert \zeta_3 u \Vert_{\mW^{2}_{\beta}(\Om)})\\
& \le & C\,(\Vert A_{-\beta}(\zeta_1u)\Vert_{\mW^{2}_{\beta}(\Om)^{\ast}}+
\Vert \mathscr{B}_{\PetitCarre}(\zeta_2 u) \Vert_{\mH^2_{\PetitCarre}(\om_2)^{\ast}}+\Vert A_{\beta}(\zeta_3 u) \Vert_{\mW^{2}_{-\beta}(\Om)^{\ast}})\\
& \le & C\,(\Vert \mathscr{B}_{\beta}(\zeta_1u)\Vert_{\mathring{\mathcal{W}}^{2}_{-\beta}(D)^{\ast}}+
\Vert \mathscr{B}_{\beta}(\zeta_2 u) \Vert_{\mathring{\mathcal{W}}^{2}_{-\beta}(D)^{\ast}}+\Vert \mathscr{B}_{\beta}(\zeta_3 u) \Vert_{\mathring{\mathcal{W}}^{2}_{-\beta}(D)^{\ast}})\\
 & \le & C\,(\sum_{j=1}^3 \Vert \zeta_j \mathscr{B}_{\beta}u\Vert_{\mathring{\mathcal{W}}^{2}_{-\beta}(D)^{\ast}}+\Vert [\mathscr{B}_{\beta} ,\zeta_j]u\|_{\mathring{\mathcal{W}}^{2}_{-\beta}(D)^{\ast}})\\[3pt]
  & \le & C\,(\Vert \mathscr{B}_{\beta}u\Vert_{\mathring{\mathcal{W}}^{2}_{-\beta}(D)^{\ast}}+\sum_{j=1}^3 \Vert [\mathscr{B}_{\beta} ,\zeta_j]u\|_{\mathring{\mathcal{W}}^{2}_{-\beta}(D)^{\ast}}).
\end{array}
\end{equation}
Here we use the notation $[\mathscr{B}_{\beta} ,\zeta_j]u= \mathscr{B}_{\beta}(\zeta_j u)-\zeta_j \mathscr{B}_{\beta}u$. Now, let us establish the estimate  
\begin{equation}\label{estimInterCommu}
\|[\mathscr{B}_{\beta} ,\zeta_1]u\|_{\mathring{\mathcal{W}}^{2}_{-\beta}(D)^{\ast}} \le C\|u\|_{\mH^1(\om_2)}.
\end{equation}
An algebraic computation using the fact that the support of $\zeta_1$ does not meet $\partial\mathscr{O}$ shows that for $\phi\in\{\phi|_{D}\,|\,\phi\in\mathscr{C}^{\infty}_0(\Om)\}$, we have
\[
\langle [\mathscr{B}_{\beta} ,\zeta_1]u,\phi\rangle_{D}=\int_{D}(u\Delta\zeta_1+2\nabla u\cdot\nabla\zeta_1)\Delta \phi-\Delta u(\phi\Delta\zeta_1+2\nabla \phi\cdot\nabla\zeta_1)\,dxdy.
\]
Integrating by parts in the term involving $\Delta u$, we obtain $|\langle [\mathscr{B}_{\beta} ,\zeta_1]u,\phi\rangle_{D}|\le C\|u\|_{\mH^1(\om_2)}\|\phi\|_{\mathcal{W}^{2}_{-\beta}(D)}$ where $C>0$ is independent of $u$. Taking the suprememum over $\{\phi|_{D}\,|\,\phi\in\mathscr{C}^{\infty}_0(\Om)\}$ leads to (\ref{estimInterCommu}). Dealing with the terms $[\mathscr{B}_{\beta} ,\zeta_2]u$ and $[\mathscr{B}_{\beta} ,\zeta_3]u$ of (\ref{estimApriori}) in a similar manner, we obtain the \textit{a priori} estimate (\ref{AprioriEstim}). Finally, observing that the map $u \mapsto u|_{\om_2}$ from $\mathring{\mathcal{W}}^{2}_{\beta}(D)$ to $\mH^1(\om_2)$ 
is compact (because $\om_2$ is bounded), one deduces from Lemma \ref{PeetreLemma} in Appendix that $\mrm{range}\,\mathscr{B}_{\beta}$ is closed and $\ker\,\mathscr{B}_{\beta}$ has finite dimension.\\

\noindent Now, let us build a right regularizer (also called a right parametrix), \textit{i.e.} an 
operator $\mrm{R}$ such that $ \mathscr{B}_{\beta}\mrm{R}-\mrm{Id}$ is a compact operator of 
$\mathring{\mathcal{W}}^{2}_{-\beta}(D)^{\ast}$. According to \cite[Lem. 2.23]{McLe00}, this will 
prove that $\coker\,\mathscr{B}_{\beta}$ is finite-dimensional. Define the operator
\[
\mrm{R}:=\zeta_1\,(A_{-\beta})^{-1}\,(\psi_1\,\cdot)+\zeta_2\,(\mathscr{B}_{\PetitCarre})^{-1}\,(\psi_2\,\cdot)+\zeta_3\,(A_{\beta})^{-1}\,(\psi_3\,\cdot).
\]
For all $f\in\mathring{\mathcal{W}}^{2}_{-\beta}(D)^{\ast}$, one finds
\[
\begin{array}{lcl}
 \mathscr{B}_{\beta}(\mrm{R}f) & = &  \mathscr{B}_{\beta}\,(\zeta_1\,(A_{-\beta})^{-1}\,(\psi_1f))+\mathscr{B}_{\beta}\,(\zeta_2\,(\mathscr{B}_{\PetitCarre})^{-1}\,(\psi_2f))+\mathscr{B}_{\beta}\,(\zeta_3\,(A_{\beta})^{-1}\,(\psi_3f))\\[3pt]
& = & A_{-\beta}\,(\zeta_1\,(A_{-\beta})^{-1}\,(\psi_1f))+\mathscr{B}_{\PetitCarre}\,(\zeta_2\,(\mathscr{B}_{\PetitCarre})^{-1}\,(\psi_2f))+A_{\beta}\,(\zeta_3\,(A_{\beta})^{-1}\,(\psi_3f))\\[3pt]
 & = & \sum_{j=1}^3\zeta_jf+[A_{-\beta},\zeta_1]\,(A_{-\beta})^{-1}(\psi_1f)+[ \mathscr{B}_{\PetitCarre},\zeta_2]\,(\mathscr{B}_{\PetitCarre})^{-1}(\psi_2f)+[A_{\beta},\zeta_3]\,(A_{\beta})^{-1}(\psi_3f).
\end{array}
\]
One can prove working as in (\ref{estimInterCommu}) that 
$[A_{-\beta},\zeta_1]$, $[ \mathscr{B}_{\PetitCarre},\zeta_2]$ and $[A_{\beta},\zeta_3]$ are compact as operators from $\mathring{\mathcal{W}}^{2}_{\beta}(D)$
to $\mathring{\mathcal{W}}^{2}_{-\beta}(D)^{\ast}$. Thus, $\mrm{R}$ 
is indeed a right regularizer and $\coker\,\mathscr{B}_{\beta}$ is finite-dimensional. This concludes the proof that $\mathscr{B}_{\beta}$ is a Fredholm operator.\\
\newline
2) Now, we focus our attention on the indices of $\mathscr{B}_{\pm\beta}$. Since $\mathscr{B}_{\beta}$ is the adjoint of $\mathscr{B}_{-\beta}$, we have 
\begin{equation}
\ind\,\mathscr{B}_{\beta}=-\ind\,\mathscr{B}_{-\beta}.
\label{R1}
\end{equation}
On the other hand, \cite[Chap. 4, Prop. 3.1, p. 110]{NaPl94} guarantees that the quantity $\ind\,\mathscr{B}_{\beta}-\ind\,\mathscr{B}_{-\beta}$ is equal to twice (because of the two outlets at infinity) the sum of the algebraic multiplicities of the eigenvalues of $\mathscr{L}$ located in the strip $-\beta<\lambda<\beta$. From Propositions \ref{propoCard}, \ref{PropositionAlgMult} and  \ref{PropositionSpectrumSymbol}, we infer that 
\begin{equation}
\ind\,\mathscr{B}_{\beta}-\ind\,\mathscr{B}_{-\beta}=4P. 
\label{R2}
\end{equation}
Gathering (\ref{R1}) and (\ref{R2}), we obtain $\ind\,\mathscr{B}_{\beta}=-\ind\,\mathscr{B}_{-\beta}=2P$. In particular, since $\dim\,\coker\,\mathscr{B}_{\beta}=\dim\,\ker\,\mathscr{B}_{-\beta}$ (again we use the fact that $\mathscr{B}^{\ast}_{\beta}=\mathscr{B}_{-\beta}$), we get $\dim\,\ker\,\mathscr{B}_{\beta}-\dim\,\ker\,\mathscr{B}_{-\beta}=2P$. Let $v_1,\dots,v_{2P}$ be functions of $\ker\,\mathscr{B}_{\beta}$ which are linearly independent modulo $\mathring{\mathcal{W}}^{2}_{-\beta}(D)$. Moreover multiplying the $v_p$ by cut-off functions and using Proposition \ref{propoDecomposition}, one finds  for $p=1,\dots,P$, the function $v_p$ decomposes as 
\[
\begin{array}{lcl}
v_p &=& \dsp\chi^+\sum_{m=1}^{P} (a_{p\,m}\,w^{+}_{m}+b_{p\,m}\,w^{-}_{m}) + \chi^-\sum_{m=1}^{P} (a_{p\,P+m}\,w^{-}_{m}+b_{p\,P+m}\,w^{+}_{m}) + \tilde{v}_p, \\[15pt]
v_{P+p} &=& \dsp\chi^+\sum_{m=1}^{P} (a_{P+p\,m}\,w^{+}_{m}+b_{P+p\,m}\,w^{-}_{m}) + \chi^-\sum_{m=1}^{P} (a_{P+p\,P+m}\,w^{-}_{m}+b_{P+p\,P+m}\,w^{+}_{m}) + \tilde{v}_{P+p}, 
\end{array}
\]
with coefficients $a_{p\,m}$, $b_{p\,m}\in\Cplx$ and $\tilde{v}_p\in \mathring{\mathcal{W}}^2_{-\beta}(D)$. Let us show that the matrix $\mathbb{B}:=(b_{p\,m})_{1\le m,p\le 2P}\in\Cplx^{2P\times2P}$ is invertible. If it is not, then there is a non zero $U$ in $\ker\,\mathbb{B}^{\top}$. Define the function $v=\sum_{p=1}^{2P}U_p\,v_p\in\ker\,\mathscr{B}_{\beta}$. Since $\sum_{m=1}^{2P}b_{p\,m}\,U_p=0$ for $m=1,\dots,2P$, we find that $v$ belongs to $\mathcal{W}^{\mrm{out}}(D)$.
 Computing $q_{D}(v,v)$ as in the Step 2 of the proof of Theorem \ref{thmIsomObstacle}, one deduces that $v\in\mathring{\mathcal{W}}^2_{-\beta}(D)$ and so $v\in\ker\,\mathscr{B}_{-\beta}$. But this is impossible because $U\ne0$ and the $v_1,\dots,v_{2P}$ are linearly independent modulo $\mathring{\mathcal{W}}^{2}_{-\beta}(D)$. Thus the matrix $\mathbb{B}$ is  invertible of inverse $\mathbb{B}^{-1}:=(\hat{b}_{p\,j})_{1\le p,j\le 2P}$. Then we construct the functions $\Psi_p$ introduced in (\ref{DefElemNoyau}) setting $\Psi_p=\sum_{j=1}^{2P}\hat{b}_{p\,j}v_j$, $p=1,\dots,2P$.\\
\newline
Besides, using Proposition \ref{PropositionFluxNRJObstacle} and the fact that the $\Psi_p$ defined in (\ref{DefElemNoyau}) belong to $\ker\,\mathscr{B}_{\beta}$, we find for $m$, $p\in\{1,\dots,2P\}$
\[
0=q_D(\Psi_m,\Psi_p)=i\Big(\delta_{m,\,p}-\sum_{j=1}^{2P}s_{m\,j}\,\overline{s_{j\,p}}\Big).
\]
Thus we deduce $\mathbb{S}\,\overline{\mathbb{S}}^{\top}=\mrm{Id}^{2P\times2P}$. In other words, $\mathbb{S}$ is unitary. Using again Proposition \ref{PropositionFluxNRJObstacle}, we also find, for $m\ne p$,
\[
0=q_D(\Psi_m,\overline{\Psi_p})=i(s_{m\,p}-s_{p\,m}).
\]
From this, we infer that $\mathbb{S}$ is symmetric. 
~\hfill$\square$

\section{Selection of the outgoing modes}\label{SectionOutgoing}
\label{selection}
For each type of boundary conditions on the edges of the strip (simply supported or clamped), we defined the outgoing solution, the interesting one from a physical point of view, as the solution decomposing on the propagating modes $w^{\pm}_p$ (and not $w^{\mp}_p$), $p=1,\dots,P$, as $x\to\pm\infty$. This choice was arbitrary. In particular, let us mention that a functional framework where we impose to the solution to decompose on the propagating modes $w^{\mp}_p$, $p=1,\dots,P$, as $x\to\pm\infty$ also leads to a Fredholm operator of index zero. In this section, we explain why our choice is physically relevant in the case of the clamped strip (the case of the simply supported strip would be treated similarly). To proceed, we come back to the time dependent equation from which the harmonic Problem (\ref{PbInitial}) has been derived. We prove that the waves associated with the propagating modes $w^{\pm}_p$ have a positive group velocity as $x\to\pm\infty$. In other words, these waves propagate energy to $\pm\infty$. Positive (resp. negative) group velocity is known as the usual criterion to discriminate what are the outgoing modes as $x\to+\infty$ (resp. $x\to-\infty$). In order to justify that this choice is pertinent, in a second step we prove that it leads to select the solution with satisfies the so-called limiting absorption principle. The idea of this limiting absorption principle consists in adding some small loss (dissipation) to the medium. In this case, we can establish that Problem (\ref{PbInitial}) (with $k$ replaced by a complex $k$ to take into account dissipation) admits a unique solution in $\mathring{\mathcal{W}}^2_{0}(\Om)\subset\mH^2(\Om)$ (to simplify, we shall work in $\Om$ but everything is similar for Problem (\ref{PbObstacle}) in $D$). This solution decomposes as the sum of a slowly exponentially decaying part plus a rapidly exponentially decaying component as $x\to\pm\infty$. The decay of the slowly exponentially decaying part is characterized by the position in the complex plane of the eigenvalues of the symbol of the operator with dissipation. What we will do is to study the limit of this complex eigenvalues to check that they converge to the ones which have been selected for the problem without absorption (the $i\eta_p$, $p=1,\dots,P$).

\subsection{Group velocities} 

In this paragraph, we compute the group velocities of the waves associated to the propagating modes $w^{\pm}_p$, $p=1,\dots,P$, defined in (\ref{DefModesPropa}). Let us start from the equation of the motion of the plate given by 
\begin{equation}\label{eqnTemporelle}
\rho h\cfrac{\partial^2 W}{\partial t^2}+D\Delta^2W=0\qquad\Leftrightarrow\qquad \cfrac{\partial^2 W}{\partial t^2}+c^2\Delta^2 W=0\quad\mbox{ with }\quad c:=\sqrt{\frac{D}{\rho h}}.
\end{equation}
Looking for waves with the time harmonic dependence in $e^{-i\om t}$ (this is a convention) leads us to set $W_p^{\pm}(x,y,t):=w^{\pm}_p(x,y)\,e^{-i\om t}=e^{i(\pm\eta_p x-\om t)}\varphi_p(y)$. Plugging $W_p^{\pm}$ in (\ref{eqnTemporelle}), we obtain the already known equation for $w^{\pm}_p$ :
\[
-\om^2 w^{\pm}_p+c^2\Delta^2w^{\pm}_p=0\qquad\Leftrightarrow\qquad \Delta^2w^{\pm}_p-k^4w^{\pm}_p=0\quad\mbox{ with }\quad k^2:=\om/c.
\]
By definition, the group velocity of the waves $W_p^{\pm}(x,y,t)=e^{i(\eta x-\om t)}\varphi_p(y)$ with $\eta=\pm\eta_p$, is given by
\begin{equation}\label{VGroupe1}
v_g(W_p^{\pm})=\frac{\partial \om}{\partial\eta}|_{\eta=\pm\eta_p}=2c k\frac{\partial k}{\partial\eta}|_{\eta=\pm\eta_p}.
\end{equation}
In order to compute $v_g(W_p^{\pm})$, we differentiate the relation $\mathscr{L}(i\eta)=\partial^4_y\varphi-2\eta^2\partial^2_y\varphi+(\eta^4-k^4)\varphi=0$ (see the definition (\ref{defOperator}) of operator $\mathscr{L}$) with respect to $\eta\in\R$ to obtain 
\[
\mathscr{L}(i\eta)(\frac{\partial\varphi}{\partial \eta})-4\eta(\partial^2_y\varphi-\eta^2\varphi)-4k^3\frac{\partial k}{\partial\eta}\varphi=0.
\]
Taking $\eta=\pm \eta_p$, multiplying by $\overline{\varphi}$ and integrating by parts, we deduce that
\begin{equation}\label{VGroupe2}
\pm4\eta_p\int_{I}|\partial_y\varphi_p(y)|^2+\eta_p^2|\varphi_p(y)|^2\,dy=4k^3\frac{\partial k}{\partial\eta}|_{\eta=\pm\eta_p}\int_{I}|\varphi_p(y)|^2\,dy.
\end{equation}
Gathering (\ref{VGroupe1}), (\ref{VGroupe2}) and using the normalisation (\ref{relationNormalization}), we find 
\[
v_g(W_p^{\pm})=\pm 2c\,\cfrac{\dsp4\eta_p\int_{I}|\partial_y\varphi_p(y)|^2+\eta_p^2|\varphi_p(y)|^2\,dy}{4k^2\dsp\int_{I}|\varphi_p(y)|^2\,dy} =\cfrac{\pm 2c}{4k^2\dsp\int_{I}|\varphi_p(y)|^2\,dy}.
\]
This shows that from a physical point of view, $W_p^{+}$ (resp. $W_p^{-}$) is the outgoing wave as $x\to+\infty$ (resp. $x\to-\infty$). As a consequence, in time harmonic regime, we have to look for a solution which decomposes on the propagating modes $w_p^{\pm}$ as $x\to\pm\infty$. This explains our choice in (\ref{defOpFredholm}). Let us translate this into a criterion for the symplectic form $q_{\Om}(\cdot,\cdot)$ defined in (\ref{FormSymp1}). For $\nu=\pm$, computation (\ref{calculFlux}) gives 
\[
iq_{\Om}(\chi^{\nu}w^{\pm}_p,\chi^{\nu}w^{\pm}_p)=\pm 4\,\nu\eta_p\int_{I}|\partial_y\varphi_p(y)|^2+\eta_p^2|\varphi_p(y)|^2\,dy=\nu v_g(W_p^{\pm})\,\cfrac{4k^2\dsp\int_{I}|\varphi_p(y)|^2\,dy}{2c}\ .
\]
Therefore at infinity ($x\to\pm\infty$), we have to select the propagating modes providing a positive value for the form $iq_{\Om}(w_p,w_p)$. Finally, note that if we define the phase velocity
\[
v_{\phi}(W_p^{\pm})=\frac{\om}{\pm\eta_p},
\]
we obtain the identities
\[
v_{\phi}(W_p^{\pm})\,v_{g}(W_p^{\pm})=2c^2\cfrac{\dsp4\eta_p\int_{I}|\partial_y\varphi_p(y)|^2+\eta_p^2|\varphi_p(y)|^2\,dy}{4\dsp\int_{I}|\varphi_p(y)|^2\,dy}=\cfrac{ 2c^2}{4\dsp\int_{I}|\varphi_p(y)|^2\,dy}\ .
\]
In particular for this problem, we see that the group and phase velocities of the waves have the same sign.
\subsection{Limiting absorption principle}
In this paragraph, we add small dissipation, modelled by a parameter $\gamma$, to the system. This dissipation ensures that the problem is well-posed in a usual setting. Then we make $\gamma$ tend to zero. This process allows us to define the physical solution and we will show that this solution is the same as the one selected via the group velocity. As a first step, we have to explain how to take into account dissipation in the time harmonic Problem (\ref{PbInitial}). To proceed, again we start from the time dependent problem. Consider the damped equation
\begin{equation}\label{DampedEquation}
\cfrac{\partial^2 W^{\gamma}}{\partial t^2}+\gamma\cfrac{\partial W^{\gamma}}{\partial t}+c^2\Delta^2 W^{\gamma}=0\ \mbox{ in }\Om
\end{equation}
with the boundary conditions $W^{\gamma}=\partial_nW^{\gamma}=0$ on $(0;+\infty)\times\partial\Om$ and appropriate initial conditions (compactly supported in space). Multiplying by $\partial_tW^{\gamma}$ and integrating in $\Om$, we obtain the energy balance
\[
\frac{d E(t)}{dt}=-\int_{\Om}\gamma\,\left(\cfrac{\partial W^{\gamma}}{\partial t}\right)^2\,dxdy\qquad\mbox{ with }\qquad E(t)=\int_{\Om}\cfrac{1}{2}\left(\cfrac{\partial W^{\gamma}}{\partial t}\right)^2+c^2\,|\Delta W^{\gamma}|^2\,dxdy.
\]
Therefore, we see that $\gamma$ must be chosen positive so that the term involving $\gamma$ in (\ref{DampedEquation}) corresponds to some dissipation (loss of energy). Applying the Fourier transform with respect to the time variable defined by 
\[
w^{\gamma}(x,y,\om):=\int_{-\infty}^{+\infty}e^{-i\om t}\,W^{\gamma}(x,y,t)\,dt
\]
(note the convention of a time harmonic regime in $e^{-i\om t}$), we are led to study the equation 
\begin{equation}\label{ProblemDissipation}
-\om^2 w^{\gamma}-i\gamma\,w^{\gamma}+c^2\Delta^2w^{\gamma}=0\qquad\Leftrightarrow\qquad \Delta^2w^{\gamma}-(k^{\gamma})^4w^{\gamma}=0\quad\mbox{ with }\quad (k^{\gamma})^4:=k^4+i\gamma/c^2.
\end{equation}
As a consequence, taking into account dissipation of the system boils down to add a positive imaginary part to $k^4$ (observe that a convention of a time harmonic regime in $e^{i\om t}$ leads to add a negative imaginary part to $k^4$). Due to the imaginary part of $k^\gamma$, using the Lax-Milgram theorem, one can prove that Problem (\ref{ProblemDissipation}) supplemented with the same boundary conditions as in (\ref{PbObstacle}) admits a unique solution $w^{\gamma}\in\mathring{\mathcal{W}}^2_{0}(\Om)\subset\mH^2(\Om)$. On the other hand, all the analysis presented in the previous sections for Problem (\ref{PbInitial}) can be adapted to consider this new problem with $k$ replaced by $k^{\gamma}$. In particular, as in (\ref{defSymbol}), we can define a symbol associated with this problem named $\mathscr{L}^{\gamma}$. We let $\Lambda^{\gamma}$ refer to the set of eigenvalues of $\mathscr{L}^{\gamma}$. We denote $i\eta_p^{\gamma}$, $p=1,\dots,P$, the $P$ elements of $\Lambda^{\gamma}$ which have the largest negative real part. These elements are uniquely defined for $\gamma$ small enough and get closer and closer to 
$\R i$ as $\gamma\to0^+$. We assume that they are ordered so that $0\le |\eta_1^{\gamma}|\le\dots\le|\eta_P^{\gamma}|$. Using a result similar to the one of Proposition \ref{propoDecomposition}, we obtain the decomposition
\[
w^{\gamma}(x,y)=\chi^+(x,y)\dsp\sum_{p=1}^{P} a^{\gamma}_p\, e^{+ i\eta^{\gamma}_p x}\varphi^{\gamma}_p(y) + \chi^-(x,y)\sum_{p=1}^{P} b^{\gamma}_p\, e^{- i\eta^{\gamma}_p x}\varphi^{\gamma}_p(y) + \tilde{w}^{\gamma}(x,y)
\]
where $a^{\gamma}_p$, $b^{\gamma}_p$ are complex numbers, $\varphi^{\gamma}_p$ is a non zero element of $\ker\,\mathscr{L}^{\gamma}(i\eta^{\gamma}_p)$ and $\tilde{w}^{\gamma}$ is rapidly exponentially decaying as $x\to\pm\infty$. Note that $w^{\gamma}$ is exponentially decaying at $x\to\pm\infty$ because $\Re e\,(i\eta^{\gamma}_p)<0$. The question we are interested in is as follows. Does the limit of $i\eta_p^{\gamma}$ as $\gamma\to0^+$ is equal to $i\eta_p$ or to $-i\eta_p$? To answer this question, we compute
\begin{equation}\label{VGroupe1bis}
\cfrac{\partial (i\eta_p^{\gamma})}{\partial\gamma}|_{\gamma=0}=-\cfrac{1}{4c^2k^3}\,\cfrac{\partial\eta_p^{\gamma}}{\partial k^\gamma}|_{\gamma=0}.
\end{equation}
We start from the equation 
\[
\partial^4_y\varphi-2\eta^2\partial^2_y\varphi+(\eta^4-k^4)\varphi=0.
\]
We differentiate it with respect to $k$ to obtain
\[
\mathscr{L}(i\eta)(\frac{\partial\varphi}{\partial k})-4\frac{\partial \eta}{dk}\eta\,(\partial^2_y\varphi-\eta^2\varphi)-4k^3\varphi=0.
\]
Taking $\gamma=0$, multiplying by $\overline{\varphi}$ and integrating by parts, we deduce by denoting $\eta^{0}_p=\lim_{\gamma\to0^+}\eta_p^{\gamma}$ that
\begin{equation}\label{VGroupe2bis}
4\eta^{0}_p\cfrac{\partial\eta_p^{\gamma}}{\partial k^\gamma}|_{\gamma=0}\int_{I}|\partial_y\varphi_p(y)|^2+\eta_p^2|\varphi_p(y)|^2\,dy=4k^3\int_{I}|\varphi_p(y)|^2\,dy.
\end{equation}
Since by definition of $\eta_p^{\gamma}$, there holds $\cfrac{\partial(i\eta_p^{\gamma})}{\partial \gamma}|_{\gamma=0}<0$, we deduce from (\ref{VGroupe1bis}) and (\ref{VGroupe2bis}) that $\eta^{0}_p>0$. Therefore indeed  we have $\eta^{0}_p=\eta_p$ and the picture is as illustrated in Figure \ref{FigureLimitingAbsorptionPrinciple}: the limiting absorption principle leads us to call outgoing modes the modes $w^{\pm}_p$ as $x\to\pm\infty$. This is coherent with what we got from considerations based on the group velocity in the previous paragraph.

\begin{figure}[!ht]
\centering
\begin{tikzpicture}
\draw[draw=black,line width=1pt,->](-4,0)--(4.2,0);
\draw[draw=black,line width=1pt,->](0,-1.5)--(0,2);
\filldraw [red,draw=none] (0,1.2) circle (0.1);
\filldraw [red,draw=none] (0,-1.2) circle (0.1);
\draw (-0.15,1.19) node[cross] {};
\draw (0.15,-1.19) node[cross] {};
\draw (-0.4,1.16) node[cross] {};
\draw (0.4,-1.16) node[cross] {};
\draw (-0.7,1.09) node[cross] {};
\draw (0.7,-1.09) node[cross] {};
\draw (-1.1,0.97) node[cross] {};
\draw (1.1,-0.97) node[cross] {};
\node at (0.1,1.2) [red,anchor=west] {$+i\eta_p$};
\node at (-0.1,-1.2) [red,anchor=east] {$-i\eta_p$};
\node at (4.4,0) [anchor=west] {$\Re e\,\lambda$};
\node at (0.7,1.7) [anchor=south] {$\Im m\,\lambda$};
\node at (-2,1.5) [anchor=south] {\small \textcolor{blue}{$+i\eta_p^\gamma$ when $\gamma\to0^+$}};
\node at (2,-1.5) [anchor=north] {\small \textcolor{blue}{$-i\eta_p^\gamma$ when $\gamma\to0^+$}};
\draw[blue,dotted,->] (-1.4,1.27) .. controls (-0.85,1.5) .. (-0.35,1.49);
\draw[blue,dotted,->] (1.4,-1.27) .. controls (0.85,-1.5) .. (0.35,-1.49);
\end{tikzpicture}
\caption{Schematic view of the behaviour of the eigenvalues of the symbol $\mathscr{L}^{\gamma}$ as the dissipation $\gamma$ tends to zero.\label{FigureLimitingAbsorptionPrinciple}}
\end{figure}
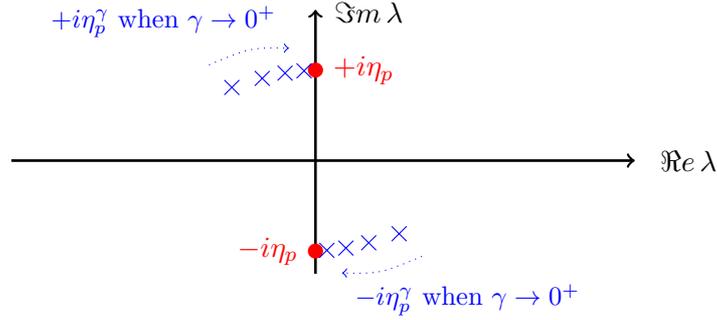
\section{Concluding remarks}\label{Conclusion}
In this article, we proved the well-posedness in the Fredholm sense of time harmonic problems set in an unbounded strip for a thin plate model. We considered two types of boundary conditions: either the strip is simply supported or the strip is clamped. To show these results, we used two different strategies, both relying on some modal decomposition. In the simply supported case, a strong result of Hilbert basis for the eigenfunctions of the symbol of the transverse problem (\ref{SpectralVaria1}) allowed us to easily obtain the modal decomposition. For the clamped problem, this result is not freely available and instead we worked with the Fourier transform in the unbounded direction, with weighted Sobolev spaces and with the residue theorem (Kondratiev approach). The second approach is more systematic than the first one. For example, it would allow one to deal with waveguides of the form $\Om=\R\times\om$, where $\om$ is a bounded domain of $\R^{d-1}$, $d\ge2$. Its main drawback maybe is that it leads to an analysis which is slightly longer than the one we get with the first method. In this work, we have considered a setting with a hole in the waveguide. We could also have considered other types of perturbations (change of material, local perturbation of the geometry,...). Our article does not address the question of uniqueness of the solution. Can one find conditions on the geometry and on the wavenumber $k$ so that trapped modes are absent? Can one show the existence of settings where trapped modes exist? These are interesting questions to study. A natural direction after this analysis would be to investigate how to approximate the solutions. For the simply supported case, the fact that we know explicitly the Dirichlet-to-Neumann operator allows one to adapt the methods used to deal with the Helmholtz problem. In the clamped case, another technique must be found. Could the technique of Perfectly Matched Layers be used and justified?

\section*{Appendix}
In this appendix, we state a result established in \cite{Tart87} which is an extension of the well-known Peetre's lemma \cite{Peet61} (see also \cite{Wlok87}), that is particularly useful to prove that an operator is of Fredholm type.

\begin{lemma}\label{PeetreLemma}
Let $(\mrm{X}, \Vert \;\Vert_{\mrm{X}}),(\mrm{Y}, \Vert \;\Vert_{\mrm{Y}})$ and 
$(\mrm{Z}, \Vert \;\Vert_{\mrm{Z}})$ be three  Banach 
spaces. Let $K: \mrm{X}\to \mrm{Z}$ be a linear compact map and 
 $B: \mrm{X}\to \mrm{Y}$ be a continuous linear map. Suppose that there exists $C>0$ such that 
\begin{equation}
\label{estim Peetre}
\Vert x\Vert_{\mrm{X}}\leq C\big(\;
\Vert B x\Vert_{\mrm{Y}}+\Vert K x\Vert_{\mrm{Z}}\,\big),\;\;\forall x\in\mrm{X}.
\end{equation}
Then 
$\dim\,\ker\,B  <\infty$ and $\mrm{range}\,B$ is closed in $\mrm{Y}$.
\end{lemma}

\bibliography{Bibli}
\bibliographystyle{plain}

\end{document}